\RequirePackage{fix-cm}
\documentclass[smallextended]{svjour3}
\usepackage[utf8]{inputenc}
\usepackage{booktabs}
\usepackage{amssymb,amsmath,epsfig,graphics,psfrag,graphicx,color,mathtools,enumitem,comment,xcolor}
\usepackage[title]{appendix}
\renewenvironment{proof}{\noindent{\it Proof.}}{\hfill$\square$}

\usepackage[colorlinks=true,breaklinks=true,linkcolor=lightblue,citecolor=lightblue]{hyperref}
\usepackage[margin=2.cm]{geometry}

\usepackage{arydshln}
\usepackage{booktabs}
\usepackage[caption=false]{subfig} 

\newcommand{\ds}{\,\mbox{d}s}

\numberwithin{equation}{section}
\numberwithin{figure}{section}
\numberwithin{theorem}{section}

\definecolor{lightgreen}{rgb}{0.22,0.50,0.25}
\definecolor{lightblue}{rgb}{0.22,0.45,0.70}
\definecolor{mygray}{rgb}{0.7,0.7,0.7}

\newcommand\cero{\boldsymbol{0}}

\newcommand{\by}{\boldsymbol{y}}
\newcommand{\bx}{\boldsymbol{x}}

\newcommand{\bn}{\boldsymbol{n}}
\newcommand{\bu}{\boldsymbol{u}}
\newcommand{\bv}{\boldsymbol{v}}
\newcommand{\bw}{\boldsymbol{w}}
\newcommand{\ff}{\boldsymbol{f}}
\newcommand\bchi{\boldsymbol{\chi}}
\newcommand\bphi{\boldsymbol{\phi}}
\newcommand\beps{\boldsymbol{\varepsilon}}
\newcommand\bsigma{\boldsymbol{\sigma}}
\newcommand\btau{\boldsymbol{\tau}}
\newcommand\bzeta{\boldsymbol{\varsigma}}
\newcommand\bnabla{\boldsymbol{\nabla}}

\newcommand\bPi{\boldsymbol{\Pi}}

\newcommand\bI{\mathbf{I}}
\newcommand\bQ{\mathbf{Q}}
\newcommand\bS{\mathbf{S}}

\newcommand\bW{\mathbf{W}}
\newcommand\bZ{\mathbf{Z}}
\newcommand\bX{\mathbf{X}}

\newcommand{\cred}[1]{\textcolor{black}{#1}}
\newcommand{\cblue}[1]{\textcolor{black}{#1}}

\allowdisplaybreaks
\begin{document}
\titlerunning{Nitsche methods for NS/poroelasticity interface problems}
\authorrunning{Bansal, Barnafi, Pandey, Ruiz-Baier}
\title{A Nitsche method for Navier--Stokes/generalized poroelasticity interface problems}
\author{Aparna Bansal \and Nicol\'as A. Barnafi \and Dwijendra Narain Pandey \and Ricardo Ruiz-Baier}
\institute{
Aparna Bansal \at Department of Mathematics, Indian Institute of Technology Roorkee, Roorkee 247667, India. \\
\email{a\_bansal@ma.iitr.ac.in}.
\and Nicol\'as A. Barnafi \at Instituto de Ingeniería Matemática y Computacional \& Facultad de Ciencias Biológicas, Pontificia Universidad Católica de Chile, Av Vicuña Mackenna 4860, Santiago, Chile; and Centro de Modelamiento Matemático (CNRS IRL2807), Santiago, Chile. \\
\email{nicolas.barnafi@uc.cl}.
\and Dwijendra Narain Pandey \at Department of Mathematics, Indian Institute of Technology Roorkee, Roorkee 247667, India. \\  \email{dwijpfma@iitr.ac.in}. 
\and Ricardo Ruiz-Baier \at School of Mathematics,
Monash University, 9 Rainforest Walk, Melbourne, Victoria 3800, Australia; 
 and Universidad Adventista de Chile, Casilla 7-D, Chill\'an, Chile. \\ 
 \email{ricardo.ruizbaier@monash.edu}.} 
\date{\today}

\maketitle
\begin{abstract}
We consider a time-dependent coupled Navier--Stokes/generalized poroelastic flow problem and propose a unified and monolithic finite element discretization based on implicit time stepping. To handle the fluid-structure interface we employ a Nitsche-type formulation. The resulting discrete problem is shown to be well-posed using the theory of differential-algebraic equations (DAEs) and the Banach fixed-point theorem. We prove stability and derive a priori error estimates for the fully discrete scheme.  The stability and convergence of the method are ensured by a properly chosen penalty parameter independent of the mesh size. Numerical tests are presented to confirm the theoretical convergence rates and to illustrate the ability of the method to capture the coupled dynamics accurately.
\end{abstract}

\section{Introduction}
 Modeling the interaction between incompressible free-fluid  and flow through deformable porous media—comm\-only termed fluid–poroelastic structure interaction (FPSI)—has attracted significant attention in recent years. This multiphysics phenomenon arises in diverse applications across the geo\-sci\-ences, bio\-me\-dicine, and industrial engineering, including groundwater movement and contaminant transport in deformable fractured aquifers, hydraulic fracturing, blood flow in arteries, interfacial transport within the eye or brain, the design of artificial organs, and the performance of industrial filtration devices.

In the free-fluid region, motion is typically governed by the Navier–Stokes equations, while the poroelastic medium is modeled by the generalized poroelasticity system, which may be derived from linearized poro-hyperelasticity. These two flow regimes are coupled through physical interface conditions: continuity of normal velocity, balance of normal stresses, balance of contact forces, and, for tangential flow components, the Beavers–Joseph–Saffman slip-with-friction condition. This formulation generalizes classical coupled free-fluid and porous-medium models such as the Stokes–Darcy system~\cite{MR2655899,MR1936102,MR2391227,MR2813344} and extends concepts from fluid–str\-uct\-ure interaction (FSI) frameworks~\cite{MR2546594,MR2536639} to account for fluid flow within a deformable porous matrix.

The mathematical and numerical analysis of FPSI problems has been addressed in a growing body of literature, particularly for the Stokes–Biot coupling, where the Stokes system models the free-fluid and the Biot equations describe the poroelastic medium. Early and recent contributions have established the well-posedness, stability, and error estimates for various discretization strategies in coupled Stokes–Biot and related fluid–poroelastic interaction problems. These include mixed finite element formulations for nonlinear and non-Newtonian models~\cite{MR4022710}, Lagrange multiplier techniques~\cite{MR3851065}, fully coupled Biot–Navier–Stokes schemes~\cite{MR2573342}, and Nitsche-based partitioning approaches~\cite{MR3347244}, as well as operator-splitting schemes for multilayered poroelastic structures~\cite{MR3343599,MR4487559,MR4089796}. Stabilized and hybridizable discontinuous Galerkin methods have been developed for improved numerical robustness~\cite{MR4598417}. Extensions include the treatment of nonlinear geometric effects in poroviscoelastic structures~\cite{MR4768306}, total-pressure approaches for interfacial flows in ocular fluid dynamics~\cite{MR4353225}, porohyperelastic coupling~\cite{MR4301411}, and non-matching interface meshes~\cite{MR2149168:}.

Beyond the typical Stokes–Biot model, more general FPSI frameworks have incorporated nonlinear fluid models, non-Newtonian rheologies, anisotropic or heterogeneous permeability, and nonlinear poroelastic constitutive laws, motivated by applications in vascular biomechanics, tissue engineering, and geomechanics. These studies highlight the complex interplay between free-fluid dynamics and the poroelastic response, and underscore the need for robust, physically consistent, and computationally efficient methods capable of handling the wide parameter regimes encountered in practice.

To enforce interface constraints in a finite element FE discretization, two main schemes have been developed. The Lagrange multiplier formulation is conceptually straightforward: it enlarges the discrete system by introducing additional multiplier fields and requires the satisfaction of inf--sup conditions to guarantee stability. In contrast, the Nitsche approach enforces interface constraints weakly through a penalty or stabilization parameter, without introducing extra unknowns. The Nitsche method, first proposed by J.~A.~Nitsche~\cite{MR341903}, is a consistent boundary-penalty technique for weakly imposing Dirichlet conditions in the variational formulation. As a result, the algebraic system remains symmetric (or skew-symmetric) and positive definite whenever the underlying operators are elliptic, yielding better-conditioned linear systems and more efficient solvers. Moreover, the Nitsche formulation facilitates the treatment of non-matching or locally refined meshes at the interface, since no mortar spaces or multiplier interpolations are required. The stability of the scheme requires only a mesh-independent penalty parameter $\gamma > 0$, thereby avoiding the need for introducing additional variables requiring further inf-sup stability conditions, as is necessary in Lagrange multiplier methods. Finally, because Nitsche's method integrates naturally into standard finite element assembly routines, it greatly simplifies implementation in existing software libraries. These features make the Nitsche scheme an attractive, robust, and computationally efficient alternative for coupled fluid--poroelastic simulations. Several studies have applied the Nitsche scheme to FSI~\cite{MR2498525,MR3160359} and Stokes--Darcy problems~\cite{MR2899252,MR4114096,MR4262392}. Moreover, a Nitsche-based framework for the coupled Stokes--Biot problem was developed in~\cite{MR3347244,MR4524684}. In this work, we employ the Nitsche scheme to enforce continuity of the normal velocity across the interface.

We present mathematical and numerical analyses of the fully dynamic Navier--Stokes–generalized poroelastic system, adopting the Brinkman model for fluid flow. This formulation ensures mass conservation within the porous domain and accounts for viscous effects in a manner consistent with thermodynamic principles~\cite{MR4253885}. The generalized poroelastic model—referred to as the linearized porohyperelastic model—was initially studied  in~\cite{MR3918635}, offering a more accurate alternative to the classical Biot model, especially in thermodynamically consistent settings involving variable porosity. More recently, a Stokes–generalized poroelasticity model was introduced in~\cite{bansal2026lagrange}, where continuous and discrete formulations were analyzed using a Lagrange multiplier approach. 

To the best of the authors' knowledge, this work presents the first work addressing the monolithic discretization of the Navier--Stokes/generalized poroelasticity interface problems using Nitsche's method. 
Significant contributions of this work are summarized as follows:
\begin{itemize}
   \item  \cred{The present work extends the analysis in~\cite{MR3918635,bansal2026lagrange} by incorporating a nonlinear convective term in the fluid region and employing the Nitsche method to weakly enforce the continuity of the normal velocity across the interface.}
    \item  \cred{The well-posedness of the semi-discrete problem is established using the theory of differential-algebraic equations (DAEs) and the Banach fixed-point theorem, under certain conditions on the penalty parameter~$\gamma$.}
    \item  \cred{We establish the well-posedness and perform a comprehensive convergence analysis of the fully discrete scheme, thereby providing a complete theoretical foundation for the coupled nonlinear problem.}
    \item  \cred{Furthermore, we conduct numerical experiments for various examples to validate the theoretical results. Additionally, we provide numerical evidence demonstrating the applicability of the proposed method to two-dimensional channel flow with rigid obstacles between porous layers, as well as to a three-dimensional simulation of blood flow through a microfluidic chip with cylindrical poroelastic obstacles.}
\end{itemize}

The remainder of the paper is organized as follows. Section~\ref{section2} introduces the notation, preliminaries, and mathematical model. Section~\ref{section3} discusses the weak formulation of the continuous model. Section~\ref{section4} presents the Nitsche formulation for the time-dependent model, establishes its well-posedness using the theory of DAEs and the Banach fixed-point theorem, and provides a stability analysis. Section~\ref{section5} analyzes the fully discrete scheme, including well-posedness, stability analysis, and error estimates, and explicitly derives the dependence of these estimates on the stabilization parameter~$\gamma$. In Section~\ref{section6}, we provide numerical experiments to test the theoretical results regarding spatio-temporal convergence. We also simulate the 2D flow in a channel with rigid obstacles between porous layers, and address a 3D simulation of blood flow through a microfluidic chip with cylindrical poroelastic obstacles. We conclude in Section~\ref{sec:concl} with a summary of our results and outline possible extensions. 

\section{Multiphysics model problem} \label{section2}
\subsection{Notation and preliminaries}
Throughout this manuscript, we utilize the classical Sobolev spaces $L^2(\Omega)$ and $H^1(\Omega)$, equipped with their respective norms $\|\bullet\|_{L^2(\Omega)}$ and $\|\bullet\|_{H^1(\Omega)}$. The $L^2$-inner product is denoted by $(\bullet,\bullet)$, and, for any arbitrary Hilbert space $H$, the duality pairing with its dual space $H'$ is represented by $\langle\bullet, \bullet\rangle_{H', H}$.  We follow the convention of denoting scalars, vectors, and tensors by $a$, $\boldsymbol{a}$, and $\mathbb{A}$, respectively. We further define the Bochner spaces $L^p(0, T ; X)$ and $L^{\infty}(0, T ; X)$ for any Banach space $X$, with norms given by $(\int_0^T\|x(s)\|_X^q \, \mathrm{d}s)^{1/q}$ and $ \operatorname{ess} \sup_{s \in (0, T)} \|x(s)\|_X$, respectively. Weak time derivatives are considered in $W^{k, p}(0, T ; X)$, defined as $\{x \in L^p(0, T ; X): D^{\alpha}x \in L^p(0, T ; X) \text{ for all } n \in \mathbb{N}, \alpha \leq k\}$, where $1 \leq p \leq \infty$. For simplicity, $C$ denotes a generic positive constant independent of the mesh size $h$ but possibly dependent on model parameters. We also use $\epsilon$ for arbitrary constants (with different values in different contexts) arising from Young's inequality. Inequalities with constants independent of $h$ are denoted by $\lesssim$ or $\gtrsim$, omitting the constants. Homogeneous boundary conditions are assumed for the analysis, as suitable lifting operators are known~\cite{MR3974685}; non-homogeneous conditions are treated in Section~\ref{section6}.
\subsection{Governing equations}
Let us consider a bounded Lipschitz domain $\Omega \subset \mathbb{R}^d$, with $d \in \{2, 3\}$, together with a partition into non-overlapping, connected subdomains $\Omega_{\mathrm{S}}$ and $\Omega_{\mathrm{P}}$, representing regions occupied by a free fluid governed by the Navier--Stokes equations and a poroelastic material governed by a general, thermodynamically consistent, linearized poro-hyperelastic system, respectively. The interface between the two subdomains is denoted by $\Sigma = \partial \Omega_{\mathrm{S}} \cap \partial \Omega_{\mathrm{P}}$. The boundary of the domain $\Omega$ is separated according to the boundaries of the two individual subdomains, that is, $\partial \Omega = \Gamma_{\mathrm{S}} \cup \Gamma_{\mathrm{P}}$.
The free fluid region $\Omega_{\mathrm{S}}$ is governed by the Navier--Stokes equations, with   fluid velocity $\bu_f^{\mathrm{S}}$ and   fluid pressure $p^{\mathrm{S}}$ as main unknowns:
\begin{subequations}
\begin{align}
\rho_f \partial_t \bu_f^{\mathrm{S}} -\bnabla \cdot \bsigma_f^{\mathrm{S}}(\bu_f^{\mathrm{S}}, p^{\mathrm{S}}) + \bu_f^{\mathrm{S}} \cdot \nabla \bu_f^{\mathrm{S}} = \ff_{\mathrm{S}} \quad \quad \quad \quad & \text{in } \Omega_{\mathrm{S}} \times(0, T], \label{stokes1}\\
\nabla \cdot \bu_f^{\mathrm{S}} = 0 \quad \quad \quad \quad & \text{in } \Omega_{\mathrm{S}} \times(0, T],  \label{stokes2}
\end{align}
\end{subequations}
where  $\beps(\bu_f^{\mathrm{S}}) = \frac12 (\bnabla \bu_f^{\mathrm{S}}+ (\bnabla \bu_f^{\mathrm{S}})^T)$ denotes the strain rate tensor; $\bsigma_f^{\mathrm{S}}(\bu_f^{\mathrm{S}}, p^{\mathrm{S}}) = 2 \mu_f \beps(\bu_f^{\mathrm{S}}) - p^{\mathrm{S}} \mathbf{I}$, stress tensor; $\ff_{\mathrm{S}}: (0,T] \rightarrow \mathbf{L}^2(\Omega_{\mathrm{S}})$, external load; $\mu_f$, fluid viscosity, and $\rho_f$, fluid density.

The poroelastic region $\Omega_{\mathrm{P}}$ is governed by the linearized poro-hyperelastic model, which includes viscoelastic properties. The primary variables are the relative fluid velocity $\bu_r^{\mathrm{P}}$, interstitial pressure $p^{\mathrm{P}}$, solid displacement $\by_s^{\mathrm{P}}$, and solid velocity $\bu_{s}^{\mathrm{P}}$. Furthermore, we adopt the notation $\bsigma_f^{\mathrm{S}}$, $\bsigma_f^{\mathrm{P}}$, and $\bsigma_s^{\mathrm{P}}$ to denote $\bsigma_f^{\mathrm{S}}(\bu_{f}^{\mathrm{S}}, p^{\mathrm{S}})$, $\bsigma_f^{\mathrm{P}}(\bu_r^{\mathrm{P}} + \bu_s^{\mathrm{P}}, p^{\mathrm{P}})$, and $\bsigma_s^{\mathrm{P}}(\by_s^{\mathrm{P}}, p^{\mathrm{P}})$, respectively. The resulting model is then defined as
\begin{subequations}
\begin{align}
    \rho_f \phi (\partial_{t} \bu_r^{\mathrm{P}} + \partial_t \bu_{s}^{\mathrm{P}}) - \bnabla \cdot \bsigma_f^{\mathrm{P}} - p^{\mathrm{P}} \nabla \phi + \phi^2 \kappa^{-1} \bu_r^{\mathrm{P}} - \theta( \bu_{s}^{\mathrm{P}} + \bu_r^{\mathrm{P}}) &= 2 \rho_f \phi \ff_{\mathrm{P}} 
    , \label{relativeporo2} \\
    (1-\phi)^2 {K}^{-1} \partial_t {p^{\mathrm{P}}} + \partial_t(\nabla \cdot \by_s^{\mathrm{P}}) + \nabla \cdot(\phi \bu_r^{\mathrm{P}}) &= {\rho}_f^{-1} \theta  
    , \label{relativeporo3} \\
    \rho_f \phi \partial_t \bu_r^{\mathrm{P}} + \rho_p \partial_t \bu_{s}^{\mathrm{P}} - \bnabla \cdot\bsigma_f^{\mathrm{P}} - \bnabla \cdot \bsigma_s^{\mathrm{P}} - \theta \bu_r^{\mathrm{P}} - \theta \bu_{s}^{\mathrm{P}} &= \rho_p \ff_{\mathrm{P}} + \rho_f \phi \ff_{\mathrm{P}} 
    , \label{relativeporo1} \\
    \rho_p \bu_{s}^{\mathrm{P}} &= \rho_p \partial_t \by_s^{\mathrm{P}}, 
    \label{relativeporo4}
\end{align}
\end{subequations}
in  $\Omega_{\mathrm{P}} \times(0, T]$, 
where $\rho_p = \rho_s(1-\phi) + \rho_f \phi$ denotes the density of the saturated porous medium. Equation~\eqref{relativeporo2} expresses the conservation of momentum for the fluid phase (a generalized Stokes law with the Brinkman effect); equation~\eqref{relativeporo3} represents mass conservation; equation~\eqref{relativeporo1} is the conservation of total momentum; and the last equation relates the solid displacement and velocity. We note that the fourth equation is multiplied by $\rho_p$ to maintain the symmetry of the block operator problem. The relevant parameters are given by: $\phi=\phi(\mathbf{x})$, porosity; $\rho_{f}, \rho_{s}$, fluid and solid densities, respectively; $\mu_f$, fluid viscosity; $\kappa$, permeability tensor; $\ff_{\mathrm{P}}: (0,T] \rightarrow  \mathbf{L}^{2}(\Omega_{\mathrm{P}})$, external load; $\theta : (0,T] \rightarrow L^2(\Omega_{\mathrm{P}})$, fluid source/sink; $K$, bulk modulus; and $\lambda_p, \mu_p$, Lam\'{e} parameters. The parameters $\rho_s, \rho_f, \mu_f, \lambda_p, \mu_p$ are assumed to be positive constants.

Let us now define two stress tensors in the poroelastic sub-domain as
\begin{subequations}
\begin{align}
\bsigma_f^{\mathrm{P}}(\bu_r^{\mathrm{P}} + \bu_s^{\mathrm{P}}, p^{\mathrm{P}}) &:= 2 \mu_f \phi \beps(\bu_r^{\mathrm{P}}) + 2 \mu_f \phi \beps( \partial_t \by_s^{\mathrm{P}}) - \phi p^{\mathrm{P}} \mathbf{I}, \label{stress1} \\
\bsigma_s^{\mathrm{P}} (\by_s^{\mathrm{P}}, p^{\mathrm{P}})&:=  2 \mu_p \beps(\by_s^{\mathrm{P}}) + \lambda_p \nabla \cdot \by_s^{\mathrm{P}} \mathbf{I} - (1-\phi) p^{\mathrm{P}} \mathbf{I}. \label{stress3}
\end{align}
\end{subequations}

This system is complemented by the following set of boundary conditions, where we set  $\Gamma_{\mathrm{P}} =\Gamma^\mathrm{D}_{\mathrm{P}} \cup \Gamma^\mathrm{N}_{\mathrm{P}}$ 
\begin{subequations}
\begin{align}
    \bu_f^{\mathrm{S}} &= \cero \quad \text{on} \quad \Gamma_{\mathrm{S}} \times (0, T], \quad  \by_s^{\mathrm{P}}  = \cero \quad \text{on} \quad \Gamma_{\mathrm{P}} \times (0, T], \label{boundary_data_1} \\
    \bu_r^{\mathrm{P}} & = \cero \quad \text{on} \quad \Gamma_{\mathrm{P}}^{\mathrm{D}} \times (0, T], \quad \bsigma_f^{\mathrm{P}} \bn_{\mathrm{P}} = \cero \quad \text{on} \quad \Gamma_{\mathrm{P}}^{\mathrm{N}} \times (0, T].\label{boundary_data_2}
\end{align}
\end{subequations}
To avoid restricting the mean value of the pressure, we assume that $|\Gamma^\mathrm{N}_{\mathrm{P}}| > 0$. We also assume that $\Gamma^\mathrm{N}_{\mathrm{P}}$ is not adjacent to the interface $\Sigma$, i.e., $\mathrm{dist}(\Gamma^\mathrm{N}_{\mathrm{P}}, \Sigma) \geq s > 0$. The interface conditions on the fluid–poroelastic interface $\Sigma$ consist of mass conservation~\eqref{int1}, balance of normal stresses~\eqref{int2}, and balance of contact forces~\eqref{int3}. Conditions~\eqref{int4}--\eqref{int5} together represent the Beavers--Joseph--Saffman (BJS) slip condition modeling tangential friction, with~\eqref{int4} involving both fluid and solid velocities, and~\eqref{int5} involving only the poroelastic fluid velocity on~$\Sigma$:
\begin{subequations}
\begin{align}
\bu_f^{\mathrm{S}} \cdot \bn_{\mathrm{S}} + (\partial_t \by_s^{\mathrm{P}} + \bu_r^{\mathrm{P}}) \cdot \bn_{\mathrm{P}} = 0 &\qquad \text{on} \quad \Sigma \times (0, T], \label{int1} \\
-(\bsigma_f^{\mathrm{S}} \bn_{\mathrm{S}}) \cdot \bn_{\mathrm{S}} = -(\bsigma_f^{\mathrm{P}} \bn_{\mathrm{P}}) \cdot \bn_{\mathrm{P}} &\qquad \text{on} \quad \Sigma \times (0, T], \label{int2} \\
\bsigma_f^{\mathrm{S}} \bn_{\mathrm{S}} + \bsigma_f^{\mathrm{P}} \bn_{\mathrm{P}} + \bsigma_s^{\mathrm{P}} \bn_{\mathrm{P}} = \cero &\qquad \text{on} \quad \Sigma \times (0, T], \label{int3} \\
-(\bsigma_f^{\mathrm{S}} \bn_{\mathrm{S}}) \cdot \btau_{f, j} = \mu_{f} \alpha_{\mathrm{BJS}} \sqrt{Z_j^{-1}} (\bu_f^{\mathrm{S}} - {\partial_t \by_s^{\mathrm{P}}}) \cdot \btau_{f, j} &\qquad \text{on} \quad \Sigma \times (0, T], \label{int4} \\
-(\bsigma_f^{\mathrm{P}} \bn_{\mathrm{P}}) \cdot \btau_{f, j} = \mu_{f} \alpha_{\mathrm{BJS}} \sqrt{Z_j^{-1}} \bu_r^{\mathrm{P}} \cdot \btau_{f, j} &\qquad \text{on} \quad \Sigma \times (0, T], \label{int5}
\end{align}
\end{subequations}
where $\bn_{\mathrm{S}}$ and $\bn_{\mathrm{P}}$ are the outward unit normal vectors to $\Omega_{\mathrm{S}}$ and $\Omega_{\mathrm{P}}$, respectively, $\btau_{f, j}, 1 \leq j \leq d-1$, is an orthogonal system of unit tangent vectors on $\Sigma$, we denote $Z_j = (\kappa \btau_{f, j}) \cdot \btau_{f, j}$, and $\alpha_{\mathrm{BJS}} \geq 0$ is a friction coefficient. 
We further set initial conditions in the following manner 
\begin{gather*}
\bu_f^{\mathrm{S}}(\bx, 0) = \bu_{f, 0}(\bx), \quad  
\bu_r^{\mathrm{P}}(\bx, 0) = \bu_{r, 0}(\bx), \quad  
\by_s^{\mathrm{P}}(\bx, 0) = \by_{s, 0}(\bx), \quad  
\bu_s^{\mathrm{P}}(\bx, 0) = \bu_{s, 0}(\bx), \quad 
p^{\mathrm{P}}(\bx, 0) = p^{\mathrm{P}, 0}(\bx).
\end{gather*}

\noindent\textbf{Assumptions}\label{assumptions}
\begin{enumerate}[label=(\textbf{H.\arabic*}), ref=$\mathrm{(H.\arabic*)}$,leftmargin=*,
  align=left]
    \item \label{(H1)}
    $\phi$ is such that $\phi, 1 / \phi,(1-\phi)$ and $1/(1-\phi)$ belong to $W^{s, r}(\Omega)$ with $s>d/r$, see \cite[Lemma 13]{MR4253885} and there exist   constants $\underline{\phi}$ and $\overline{\phi}$ such that $0<\underline{\phi} \leq \phi \leq \overline{\phi}< \frac{\rho_s}{\rho_s+\rho_f}<1$ a.e. in $\Omega$.
    \item \label{(H2)} $\theta $ represents a fluid sink.
\item \label{(H3)} The permeability tensor is symmetric and positive-definite, i.e., 
$$ \exists c >0 ~~
\bx^T \kappa^{-1} \bx \geq c |\bx|^2 \quad \forall \bx \in \mathbb{R}^d.
$$
\end{enumerate}
From these assumptions, we obtain ellipticity properties to be used in both the well-posedness analysis and the energy estimates. 

\section{Continuous analysis}\label{section3}
We consider the following functional spaces (endowed with the standard norms) as
\begin{gather*}
\mathbf{V}_f\coloneqq\left\{\bu_f^{\mathrm{S}}\in \mathbf{H}^1(\Omega_{\mathrm{S}}): \bu_f^{\mathrm{S}}= \cero  \text { on } \Gamma_{\mathrm{S}}\right\}, \quad \mathrm{W}_f\coloneqq L^2(\Omega_{\mathrm{S}}),  \\ 
\mathbf{V}_{r}\coloneqq\left\{\bu_r^{\mathrm{P}}  \in \mathbf{H}^1(\Omega_{\mathrm{P}} ): \bu_r^{\mathrm{P}}= \cero  \text { on } \Gamma_{\mathrm{P}}^{\mathrm{D}} \right\}, \quad  \mathrm{W}_p\coloneqq L^2(\Omega_{\mathrm{P}}),  \\ 
\mathbf{V}_{s}\coloneqq\left\{\by_s^{\mathrm{P}} \in \mathbf{H}^1(\Omega_{\mathrm{P}}): \by_s^{\mathrm{P}}= \cero  \text { on } \Gamma_{\mathrm{P}}\right\}, \quad 
\mathbf{W}_s\coloneqq\mathbf{L}^2(\Omega_{\mathrm{P}}).
\end{gather*} 
We now define, for all  $\bu_f^{\mathrm{S}}, \bv_f^{\mathrm{S}} \in \mathbf{V}_f, \bu, \bv \in \mathbf{H}^1(\Omega_{\mathrm{P}}), \by_s^{\mathrm{P}}, \bw_s^{\mathrm{P}} \in \mathbf{V}_s$, the following bilinear forms related to the Navier--Stokes, Brinkman,   and elasticity operators:
\begin{gather*}
a_f^{\mathrm{S}}(\bu_f^{\mathrm{S}}, \bv_f^{\mathrm{S}})  \coloneqq (2 \mu_f \beps (\bu_f^{\mathrm{S}}), \beps(\bv_f^{\mathrm{S}}))_{\Omega_{\mathrm{S}}},\qquad 
a_{f}^{\mathrm{P}}(\bu,\bv)  \coloneqq (2 \mu_f \phi \beps ( \bu ), \beps( \bv ))_{\Omega_{\mathrm{P}}}, \\
a_s^{\mathrm{P}}(\by_s^{\mathrm{P}}, \bw_s^{\mathrm{P}}) \coloneqq (2 \mu_p \beps(\by_s^{\mathrm{P}}), \beps(\bw_s^{\mathrm{P}}))_{\Omega_{\mathrm{P}}}+(\lambda_p \nabla \cdot \by_s^{\mathrm{P}}, \nabla \cdot \bw_s^{\mathrm{P}})_{\Omega_{\mathrm{P}}}, \qquad 
c(\bu_f^{\mathrm{S}},\bu_f^{\mathrm{S}}, \bv_f^{\mathrm{S}}) \coloneqq 
( \bu_f^{\mathrm{S}} \cdot \bnabla \bu_f^{\mathrm{S}}, \bv_f^{\mathrm{S}}).
\end{gather*}
Also, for all 
$\bv_f^{\mathrm{S}} \in \mathbf{V}_f$, 
$q^{\mathrm{S}} \in \mathrm{W}_f$,
$\bv_r^{\mathrm{P}} \in \mathbf{V}_r$, 
$q^{\mathrm{P}} \in\mathrm{W}_p$,
$\bw_s^{\mathrm{P}} \in \mathbf{V}_s^{\mathrm{P}}$, 
$\bw,\bzeta \in \bW_s$, 
we define the following bilinear forms:
\begin{gather*}
  b^{\mathrm{S}}(\bv_f^{\mathrm{S}}, q^{\mathrm{S}})
  \coloneqq -\bigl(\nabla\!\cdot \bv_f^{\mathrm{S}},\; q^{\mathrm{S}}\bigr),\qquad 
  b_f^{\mathrm{P}}(\bv_r^{\mathrm{P}}, q^{\mathrm{P}})
  \coloneqq -\bigl(\nabla\!\cdot(\phi\,\bv_r^{\mathrm{P}}),\; q^{\mathrm{P}}\bigr),\\
  b_s^{\mathrm{P}}(\bw_s^{\mathrm{P}}, q^{\mathrm{P}})
  \coloneqq -\bigl(\nabla\!\cdot \bw_s^{\mathrm{P}},\; q^{\mathrm{P}}\bigr),\qquad 
  m_{\xi}(\bw,\bzeta)
  \coloneqq (\xi\,\bw,\;\bzeta).
\end{gather*}
Integration by parts in \eqref{stokes1}, \eqref{relativeporo2} and  \eqref{relativeporo1} leads to the interface term
$$
I_{\Sigma} \coloneqq  - \langle \bsigma_f^{\mathrm{S}} \bn_{\mathrm{S}}, \bv_f^{\mathrm{S}} \rangle_\Sigma - \langle \bsigma_f^{\mathrm{P}} \bn_{\mathrm{P}}, \bw_s^{\mathrm{P}} \rangle_\Sigma  - \langle \bsigma_{s}^{\mathrm{P}} \bn_{\mathrm{P}}, \bw_s^{\mathrm{P}} \rangle_\Sigma - \langle \bsigma_f^{\mathrm{P}} \bn_{\mathrm{P}}, \bv_r^{\mathrm{P}} \rangle_\Sigma.
$$
Using the interface conditions \eqref{int2}-\eqref{int5}, we obtain 
\begin{align*}
I_{\Sigma} &=  - \int_{\Sigma} (\bsigma_f^{\mathrm{S}} (\bu_{f}^{\mathrm{S}},p^{\mathrm{S}}) \bn_{\mathrm{S}})\bn_{\mathrm{S}} (\bn_{\mathrm{S}} \cdot \bv_{f}^{\mathrm{S}} + \bn_{\mathrm{P}} \cdot \bv_{r}^{\mathrm{P}} + \bn_{\mathrm{P}} \cdot \bw_{s}^{\mathrm{P}}) \ds \\& \quad - \int_{\Sigma} (\bsigma_f^{\mathrm{P}} \bn_{\mathrm{P}})\btau_{f,j} (\btau_{f,j} \cdot \bv_r^{\mathrm{P}}) \ds   - \int_{\Sigma} (\bsigma_f^{\mathrm{S}}  \bn_{\mathrm{S}})\btau_{f,j} (\bv_f^{\mathrm{S}}-\bw_s^{\mathrm{P}}) \cdot \btau_{f, j} \ds \\
& =:  b_{\Gamma}(\bv_{f}^{\mathrm{S}}, \bv_{r}^{\mathrm{P}}, \bw_{s}^{\mathrm{P}} ; \bu^{\mathrm{S}}_f, p^{\mathrm{S}} ) + b_{\Gamma}( \bu_{f}^{\mathrm{S}}, \bu_{r}^{\mathrm{P}}, \partial_t \by_{s}^{\mathrm{P}} ; \varsigma  \bv^{\mathrm{S}}_f, -q^{\mathrm{S}})    + a_{\mathrm{BJS}}(\bu_{f}^{\mathrm{S}}, \partial_t \by_{s}^{\mathrm{P}} ; \bv_{f}^{\mathrm{S}}, \bw_{s}^{\mathrm{P}}) + b_{\mathrm{BJS}}(\bu_r^{\mathrm{P}}  ; \bv_r^{\mathrm{P}} ), 
\end{align*}
together with the definitions
\begin{align*}
a_{\mathrm{BJS}}( \bu_{f}^{\mathrm{S}}, \by_{s}^{\mathrm{P}} ; \bv_{f}^{\mathrm{S}}, \bw_{s}^{\mathrm{P}}) & \coloneqq \sum_{j=1}^{d-1}\int_{\Sigma}\mu_f \alpha_{\mathrm{BJS}} \sqrt{Z_j^{-1}}(\bu_{f}^{\mathrm{S}}-\by_{s}^{\mathrm{P}}) \cdot \btau_{f, j}(\bv_{f}^{\mathrm{S}}-\bw_{s}^{\mathrm{P}}) \cdot \btau_{f, j}\ds, \\ 
b_{\mathrm{BJS}}(\bu_r^{\mathrm{P}}  ; \bv_r^{\mathrm{P}} ) & := 
\sum_{j=1}^{d-1}\int_{\Sigma} \mu_f \alpha_{\mathrm{BJS}} \sqrt{Z_j^{-1}}(\bu_r^{\mathrm{P}} \cdot \btau_{f, j} ) (\bv_r^{\mathrm{P}} \cdot \btau_{f, j} ) \ds, \\
b_{\Gamma}(\bv_{f}^{\mathrm{S}}, \bv_{r}^{\mathrm{P}}, \bw_{s}^{\mathrm{P}} ; \!\bu_{f}^{\mathrm{S}}, p^{\mathrm{S}} ) & \coloneqq -\!  \int_{\Sigma} (2 \mu_f \beps(\bu_{f}^{\mathrm{S}}) -\! p^{\mathrm{S}} \mathbf{I})\bn_{\mathrm{S}}\bn_{\mathrm{S}} (\bn_{\mathrm{S}} \cdot \bv_{f}^{\mathrm{S}} +\bn_{\mathrm{P}} \cdot \bv_{r}^{\mathrm{P}} + \bn_{\mathrm{P}} \cdot \bw_{s}^{\mathrm{P}}) \ds.
\end{align*}
We use a shorthand notation for trial and test  functions $\vec{\bx}  =(
\bu_f^{\mathrm{S}},
{p}^{\mathrm{S}},
\bu_r^{\mathrm{P}},
p^{\mathrm{P}},
\by_s^{\mathrm{P}},
\bu_s^{\mathrm{P}})$, $\vec{\by}  =(
\bv_f^{\mathrm{S}},
{q}^{\mathrm{S}},
\bv_r^{\mathrm{P}}, 
q^{\mathrm{P}}, \\
\bw_s^{\mathrm{P}}, 
\bv_s^{\mathrm{P}})$ and denote the corresponding product space as
\[
\vec{\mathbf{X}} := \mathbf{V}_f \times \mathrm{W}_f \times \mathbf{V}_r \times \mathrm{W}_p \times \mathbf{V}_s \times \mathbf{W}_s.
\]
Furthermore, we  define the bilinear forms $\mathrm{E}, \mathrm{H} : \vec{\mathbf{X}} \times \vec{\mathbf{X}} \to \mathbb{R}$, which contain all terms with and without time derivatives, respectively: 
\begin{align*}
E(\partial_t \vec{\bx}, \vec{\by}) &\coloneqq  m_{\rho_f}( \partial_t \bu_f^{\mathrm{S}}, \bv_f^{\mathrm{S}} ) + m_{\rho_f \phi}(\partial_t \bu_{r}^{\mathrm{P}}, \bw_{s}^{\mathrm{P}})  
+ m_{\rho_p}(\partial_t \bu_{s}^{\mathrm{P}}, \bw_{s}^{\mathrm{P}}) 
+ m_{\rho_f \phi}(\partial_t \bu_{r}^{\mathrm{P}}, \bv_{r}^{\mathrm{P}}) \\
& \quad 
 + m_{\rho_f \phi}(\partial_t \bu_{s}^{\mathrm{P}}, \bv_{r}^{\mathrm{P}})  -m_{\rho_p}(\partial_t \by_{s}^{\mathrm{P}}, \bv_{s}^{\mathrm{P}}) + a_{f}^{\mathrm{P}}(\partial_t \by_{s}^{\mathrm{P}}, \bv_{r}^{\mathrm{P}}) 
+ a_{f}^{\mathrm{P}}(\partial_t \by_{s}^{\mathrm{P}}, \bw_{s}^{\mathrm{P}})  \\
& \quad 
- m_{\theta}(\partial_t \by_{s}^{\mathrm{P}}, \bw_{s}^{\mathrm{P}})  - m_{\theta}(\partial_t \by_{s}^{\mathrm{P}}, \bv_{r}^{\mathrm{P}}) + a_{\mathrm{BJS}}(0, \partial_t \by_s^{\mathrm{P}} ; \bv_f^{\mathrm{S}}, \bw_s^{\mathrm{P}}) \\ & \quad 
+ b_{\Gamma}( \cero, \cero , \partial_t \by_{s}^{\mathrm{P}} ; \varsigma \bv^{\mathrm{S}}_{f}, -q^{\mathrm{S}}) + ((1-\phi)^2 K^{-1} \partial_t p^{\mathrm{P}}, q^{\mathrm{P}})_{\Omega_{\mathrm{P}}}   -  b_s^{\mathrm{P}}(\partial_t \by_{s}^{\mathrm{P}}, q^{\mathrm{P}}), 
\\  H(\vec{\bx}, \vec{\by}) &\coloneqq  m_{\rho_p}(\bu_{s}^{\mathrm{P}}, \bv_{s}^{\mathrm{P}})
+ a_f^{\mathrm{S}}(\bu^{\mathrm{S}}_{f}, \bv_{f}^{\mathrm{S}}) 
+ a_{f}^{\mathrm{P}}(\bu_{r}^{\mathrm{P}}, \bw_{s}^{\mathrm{P}}) 
+ a_s^{\mathrm{P}}(\by_{s}^{\mathrm{P}}, \bw_{s}^{\mathrm{P}}) 
+ a_{f}^{\mathrm{P}}(\bu_{r}^{\mathrm{P}}, \bv_{r}^{\mathrm{P}}) \\ & 
\quad  + b^{\mathrm{S}}(\bv_{f}^{\mathrm{S}}, p^{\mathrm{S}}) 
+ b_s^{\mathrm{P}}(\bw_{s}^{\mathrm{P}}, p^{\mathrm{P}})  
+ b_f^{\mathrm{P}}(\bv_{r}^{\mathrm{P}}, p^{\mathrm{P}}) 
- m_{\theta}(\bu_{r}^{\mathrm{P}}, \bw_{s}^{\mathrm{P}}) - m_{\theta}(\bu_{r}^{\mathrm{P}}, \bv_{r}^{\mathrm{P}}) \\& \quad 
+ m_{\phi^2/\kappa}(\bu_{r}^{\mathrm{P}}, \bv_{r}^{\mathrm{P}})  + b_{\Gamma}(\bv_{f}^{\mathrm{S}}, \bv_{r}^{\mathrm{P}}, \bw_{s}^{\mathrm{P}} ; \bu^{\mathrm{S}}_{f}, p^{\mathrm{S}} ) + a_{\mathrm{BJS}}(\bu_f^{\mathrm{S}}, 0 ; \bv_f^{\mathrm{S}}, \bw_s^{\mathrm{P}} ) \\&\quad + b_{\mathrm{BJS}}(\bu_r^{\mathrm{P}}  ;  \bv_r^{\mathrm{P}}) + b_{\Gamma}( \bu_{f}^{\mathrm{S}}, \bu_{r}^{\mathrm{P}}, \cero; \varsigma \bv^{\mathrm{S}}_{f}, -q^{\mathrm{S}})   - b_f^{\mathrm{P}}(\bu_{r}^{\mathrm{P}}, q^{\mathrm{P}})   - b^{\mathrm{S}}(\bu^{\mathrm{S}}_{f}, q^{\mathrm{S}}), \\  L(\vec{\bx}, \vec{\bx}, \vec{\by}) &\coloneqq 
\boldsymbol{c}(\bu_f^{\mathrm{S}},\bu_f^{\mathrm{S}}, \bv_f^{\mathrm{S}}), 
    \end{align*}
    whereas the right-hand side terms are denoted by the form $F$, given by:
    \begin{equation}\label{eq:linear-funct} F(\vec{\by}) \coloneqq (\ff_{\mathrm{S}}, \bv_{f}^{\mathrm{S}} )_{\Omega_{\mathrm{S}}} 
+ (\rho_{f} \phi \ff_{\mathrm{P}}, \bv_{r}^{\mathrm{P}})   + (\rho_p \ff_{\mathrm{P}}, \bw_{s}^{\mathrm{P}}) + (r_{\mathrm{S}}, q^{\mathrm{S}}) + (\rho_{f}^{-1} \theta, q^{\mathrm{P}}). \end{equation}

First, we multiply \eqref{stokes1}–\eqref{stokes2} and \eqref{relativeporo2}–\eqref{relativeporo4} by their respective test functions, apply integration by parts, and impose the boundary conditions \eqref{boundary_data_1}–\eqref{boundary_data_2}. The balance of normal stress, the $\mathrm{BJS}$ conditions, and the conservation of momentum \eqref{int2}–\eqref{int5} are then naturally incorporated into the derivation of the weak formulation, while the conservation of mass \eqref{int1} is enforced strongly. Hence, the continuous weak formulation reads: for $t \in (0, T]$, find $\vec{\bx}(t) \in \vec{\bX}$, with
$  \bu_f^{\mathrm{S}} \cdot \bn_{\mathrm{S}} 
  + \bigl(\partial_t \by_s^{\mathrm{P}} + \bu_r^{\mathrm{P}} \bigr) \cdot \bn_{\mathrm{P}} 
  = 0,$ and subject to the given initial conditions,
such that
\begin{equation}\label{n_weak}
    E(\partial_t \vec{\bx}, \vec{\by}) + H(\vec{\bx}, \vec{\by}) + L(\vec{\bx}, \vec{\bx}, \vec{\by})= F(\vec{\by}),
\end{equation}
for all $\vec{\by} \in \vec{\bX}$. We further define
\begin{align*}
\left|\bu_{f}^{\mathrm{S}}-\by_{s}^{\mathrm{P}}\right|_{\mathrm{BJS}}^2  & \coloneqq a_{\mathrm{BJS}}(\bu_{f}^{\mathrm{S}}, \by_{s}^{\mathrm{P}} ; \bu_{f}^{\mathrm{S}}, \by_{s,h}^{\mathrm{P}}) 
 =\sum_{j=1}^{d-1} \mu_f \alpha_{\mathrm{BJS}}\|Z_j^{-1 / 4}(\bu_{f}^{\mathrm{S}}-\by_{s}^{\mathrm{P}}) \cdot \btau_{f, j}\|_{0,\Sigma}^2, \\  \left|\bu_r^{\mathrm{P}}\right|_{\mathrm{BJS}}^2 & := b_{\mathrm{BJS}}(\bu_r^{\mathrm{P}}  ; \bv_r^{\mathrm{P}}) = \sum_{j=1}^{d-1} \mu_f \alpha_{\mathrm{BJS}} \|Z_j^{-1 / 4} \bu_r^{\mathrm{P}}\cdot \btau_{f, j}\|_{0,\Sigma}^2.
\end{align*}

Given that the primary objective of this study is to analyze the discrete formulation using Nitsche's method, we do not present the continuous analysis, which can be found in the recent work \cite{bansal2026lagrange}. 

\section{Discrete weak formulation with Nitsche}\label{section4}
Suppose that $\mathcal{T}_h^{\mathrm{S}}$ and $\mathcal{T}_h^{\mathrm{P}}$ are shape‑regular, quasi uniform partitions of $\Omega_{\mathrm{S}}$ and $\Omega_{\mathrm{P}}$, respectively, each consisting of affine elements of maximal diameter~$h$. The two partitions may be non‑matching at the interface~$\Sigma$, where the $(d-1)$‑dimensional diameter of a face on~$\Sigma$ is denoted by~$h_E$ and $\mathcal{E}_{\Sigma}$ represents the faces lying on the boundary $\Sigma $. 
To discretize the unknowns in the Navier--Stokes and generalized poroelasticity problems, we define $ X_h^k = \{\,q \in C(\Omega)\colon q|_K \in \mathbb{P}_k(K)\ \forall\,K\in\mathcal{T}_h\},$
where $\mathbb{P}_k(K)$ is the space of polynomials of degree $k\ge1$ on each $K$.  With these definitions we then set up the following  discrete spaces:
\begin{gather*}
 \mathbf{V}_{f, h}=\mathbf{V}_f \cap\left[X_h^{k+1}\right]^d, \quad \mathrm{W}_{f,h}=\mathrm{W}_f \cap X_h^k, \quad \mathbf{V}_{r, h}=\mathbf{V}_r \cap\left[X_h^{k+1}\right]^d, \\ 
\mathbf{V}_{s,h}=\mathbf{V}_s \cap\left[X_h^{k+1}\right]^d, \quad\mathrm{W}_{p,h}=\mathrm{W}_p \cap X_h^k, \quad \mathbf{W}_{s, h}=\mathbf{W}_s \cap\left[X_h^k\right]^d. 
\end{gather*}
    \cred{We recall that $\mathbf{V}_{s,h}$ corresponds to the finite element space of $\by_{s,h}^{\mathrm{P}}$, while the space for $\bv_{s,h}^{\mathrm{P}}$ is $\mathbf{W}_{s,h}$.}
The global velocity and pressure spaces are defined as 
\begin{align*}
    \mathbf{V}_h&:=\left\{\vec{\bv}_h=(\bv_{f, h}^{\mathrm{S}}, \bv_{r, h}^{\mathrm{P}}, \bw_{s,h}^{\mathrm{P}} ) \in \mathbf{V}_{f,h} \times \mathbf{V}_{r,h} \times \mathbf{V}_{s,h} 
    \right\}, \quad 
   \cblue{\mathrm{W}_h:=\left\{\vec{q}_h=(q_{h}^{\mathrm{S}},  q_{h}^{\mathrm{P}} ) \in\mathrm{W}_{f, h} \times\mathrm{W}_{p, h}\right\},}
\end{align*}
equipped with the \cblue{norms}:
\begin{align*}
   \cred{ \| \cdot \|_{\mathbf{V}_h} } &= \cred{ \| \cdot \|_{1,\Omega_{\mathrm{S}}} + \| \cdot \|_{1,\Omega_{\mathrm{P}}} + \| \cdot \|_{1,\Omega_{\mathrm{P}}},} \qquad 
  \cred{  \| \cdot \|_{\mathrm{W}_h} }  =  \cred{  \| \cdot \|_{0,\Omega_{\mathrm{S}}} + \| \cdot \|_{0,\Omega_{\mathrm{P}}}.}
\end{align*}
The semi‐discrete weak formulation reads: find \(\vec{\bx}_h(t)\in \bX_h\) such that
\begin{equation}\label{semi_weak_for}
  \bar{E}\bigl(\partial_t \vec{\bx}_h,\vec{\by}_h\bigr)
  +\bar{H}\bigl(\vec{\bx}_h,\vec{\by}_h\bigr) + L(\vec{\bx}_h, \vec{\bx}_h, \vec{\by}_h)
  =F(\vec{\by}_h),
  \quad\forall\,\vec{\by}_h\in\bX_h,
\end{equation}
where
\begin{subequations}
\begin{align}\label{semi_weak}
  \bar{E}\bigl(\partial_t \vec{\bx}_h,\vec{\by}_h\bigr)
  &:=E\bigl(\partial_t \vec{\bx}_h,\vec{\by}_h\bigr)
    +c_{\Gamma}\bigl(\mathbf{0},\mathbf{0},\partial_t \by_{s,h}^{\mathrm{P}}
    ;\,\bv_{f,h}^{\mathrm{S}},\bv_{r,h}^{\mathrm{P}},\bw_{s,h}^{\mathrm{P}}\bigr),\\
  \bar{H}\bigl(\vec{\bx}_h,\vec{\by}_h\bigr)
  &:=H\bigl(\vec{\bx}_h,\vec{\by}_h\bigr)
    +c_{\Gamma}\bigl(\bu_{f,h}^{\mathrm{S}},\bu_{r,h}^{\mathrm{P}},\mathbf{0}
    ;\,\bv_{f,h}^{\mathrm{S}},\bv_{r,h}^{\mathrm{P}},\bw_{s,h}^{\mathrm{P}}\bigr).
\end{align}\end{subequations}
The conservation of mass is enforced weakly using the Nitsche parameter \(\gamma\), i.e.,
\begin{align*}
&c_{\Gamma}\bigl(\bu_{f,h}^{\mathrm{S}},\bu_{r,h}^{\mathrm{P}},\by_{s,h}^{\mathrm{P}}   ;\,\bv_{f,h}^{\mathrm{S}},\bv_{r,h}^{\mathrm{P}},\bw_{s,h}^{\mathrm{P}}\bigr) \\
&\quad 
  :=\int_{\Sigma}
    \frac{\gamma}{h_E}
    \bigl(\bn_{\mathrm{S}}\cdot \bu_{f,h}^{\mathrm{S}}
    +\bn_{\mathrm{P}}\cdot \bu_{r,h}^{\mathrm{P}}
    +\bn_{\mathrm{P}}\cdot\partial_t \by_{s,h}^{\mathrm{P}}\bigr)
    \bigl(\bn_{\mathrm{S}}\cdot \bv_{f,h}^{\mathrm{S}}
    +\bn_{\mathrm{P}}\cdot \bv_{r,h}^{\mathrm{P}}
    +\bn_{\mathrm{P}}\cdot \bw_{s,h}^{\mathrm{P}}\bigr)
  \,\mathrm{d}s,
\end{align*}
and the initial conditions
$  \bu_{f,h}^{\mathrm{S}}(0),\;
  \bu_{r,h}^{\mathrm{P}}(0),\;
  \by_{s,h}^{\mathrm{P}}(0),\;
  p_h^{\mathrm{P}}(0),\;
  \bu_{s,h}^{\mathrm{P}}(0)$
are suitable approximations of $\bu_{f,0}$, $\bu_{r,0}$, $\by_{s,0}$,  $p^{\mathrm{P},0}$, and $\bu_{s,0}$, respectively.

\subsection{Existence and uniqueness of semi-discrete solution}
The existence and uniqueness of the solution is proved in two steps:
\begin{enumerate}
\item We introduce the set
   $ \mathbf{K}_h \coloneqq \left\{ \bv^{\mathrm{S}}_{f,h} \in \mathbf{V}_{f,h} : \|\bnabla \bv_{f,h}^{\mathrm{S}}\|_{0, \Omega_{\mathrm{S}}} \leq \frac{1}{ C_{\mathrm{I}} S_f^2 K_f^3}\|\boldsymbol{f}_{\mathrm{S}}\|_{0, \Omega_{\mathrm{S}}} \right\},$  
   where $S_f, K_f > 0$ denote the Sobolev and Korn constants, respectively. 
 \cred{  The constant $C_{\mathrm{I}} =C(C_S,\xi_2) >0$ is the inf–sup stability constant of the linearized coupled problem, 
which depends on $C_S$ and $\xi_2$ (defined later on).} We then define the discrete fixed-point operator as
    \begin{align}\label{fixed_map}
           \mathcal{J}_h: \mathbf{K}_h \rightarrow \mathbf{K}_h, \quad \bw^{\mathrm{S}}_{f,h} \mapsto \mathcal{J}_h(\bw^{\mathrm{S}}_{f,h}) = \bu^{\mathrm{S}}_{f,h},
    \end{align}
    where, for a given $\bw^{\mathrm{S}}_{f,h} \in \mathbf{K}_h$, the function $\bu^{\mathrm{S}}_{f,h}$ denotes the first component of the solution to the linearised version of problem~\eqref{semi_weak_for}, given by fixing $\bw^{\mathrm{S}}_{f,h} \in \mathbf{V}_{f,h}$ in the nonlinear form, such that
    \[
    \boldsymbol{c}(\bw^{\mathrm{S}}_{f,h}, \bu^{\mathrm{S}}_{f,h}, \bv^{\mathrm{S}}_{f,h}) = (\bw^{\mathrm{S}}_{f,h} \cdot \bnabla \bu^{\mathrm{S}}_{f,h}, \bv^{\mathrm{S}}_{f,h}),
    \]
    and then prove the well-posedness of the resulting Oseen/generalised poroelasticity problem.

    \item Next, we show that the discrete fixed-point operator $\mathcal{J}_h$ admits a fixed point.
\end{enumerate}

First, we show that the Oseen/generalised poroelasticity system is well-posed. To prove the existence of a solution, we employ the theory of differential-algebraic equations (DAEs) \cite{MR1101809}.

Let 
\(\{\boldsymbol\phi_{u_f,i}^{\mathrm{S}}\}_{i=1}^{k}\), 
\(\{\boldsymbol\phi_{u_r,i}^{\mathrm{P}}\}_{i=1}^{k}\), 
\(\{\boldsymbol\phi_{y_s,i}^{\mathrm{P}}\}_{i=1}^{k}\), 
\(\{\phi_{p_f,i}^{\mathrm{S}}\}_{i=1}^{k}\), 
\(\{\phi_{p_p,i}^{\mathrm{P}}\}_{i=1}^{k}\) 
and 
\(\{\boldsymbol\phi_{u_s,i}^{\mathrm{P}}\}_{i=1}^{k}\) 
be bases of $\mathbf{V}_{f,h}$, $\mathbf{V}_{r,h}$,  $\mathbf{V}_{s,h}$,  $\mathrm{W}_{f,h}$, $\mathrm{W}_{p,h}$, and 
$\mathbf{W}_{s,h}$, respectively. We write the linearized version of problem~\eqref{semi_weak_for} in matrix form. For this we introduce, for $1 \leq i, j \leq k$, the notation 
\begin{gather*}
\mathcal M_{\xi}   =   m_{\xi}(\boldsymbol\phi_{\star,j}^{\mathrm{P}},\boldsymbol\phi_{\star,i}^{\mathrm{P}}),  \quad 
\mathcal A_f^{\mathrm{S}}   =   a_f^{\mathrm{S}}(\boldsymbol\phi_{u_f,j}^{\mathrm{S}},\boldsymbol\phi_{u_f,i}^{\mathrm{S}}), \quad 
\mathcal A_f^{\mathrm{P}}   = a_{f}^{\mathrm{P}}(\boldsymbol\phi_{\star,j}^{\mathrm{P}},\boldsymbol\phi_{\star,i}^{\mathrm{P}}), \\
\mathcal A_s^{\mathrm{P}}   = a_s^{\mathrm{P}}(\boldsymbol\phi_{y_s,j}^{\mathrm{P}}, \boldsymbol\phi_{y_s,i}^{\mathrm{P}}), \quad 
\mathcal B^{\mathrm{S}}   =  b^{\mathrm{S}}(\boldsymbol\phi_{u_f,j}^{\mathrm{S}}, \phi_{p,i}^{\mathrm{S}}), \quad
 \mathcal A_{ss}^{\mathrm{BJS}} = a_{\mathrm{BJS}}\!\bigl(\cero,\boldsymbol\phi_{y_s,j}^{\mathrm{P}};\,\cero,\boldsymbol\phi_{y_s,i}^{\mathrm{P}}\bigr), \\
\mathcal B_s^{\mathrm{P}} =  b_s^{\mathrm{P}}(\boldsymbol\phi_{y_s,j}^{\mathrm{P}},\phi_{p,i}^{\mathrm{P}}), \quad \mathcal B_f^{\mathrm{P}} =  b_f^{\mathrm{P}}(\boldsymbol\phi_{u_r,j}^{\mathrm{P}},\phi_{p,i}^{\mathrm{P}}), \quad
\mathcal A_{fs}^{\mathrm{BJS}} =  a_{\mathrm{BJS}}\!\bigl(\boldsymbol\phi_{u_f,j}^{\mathrm{S}},\cero;\,\cero,\boldsymbol\phi_{y_s,i}^{\mathrm{P}}\bigr), \\
  \mathcal A_{ff}^{\mathrm{BJS}} = a_{\mathrm{BJS}}\!\bigl(\boldsymbol\phi_{u_f,j}^{\mathrm{S}},\cero;\,\boldsymbol\phi_{u_f,i}^{\mathrm{S}},\cero\bigr), \quad
  \mathcal A_{rr}^{\mathrm{BJS}} =  a_{\mathrm{BJS}}\!\bigl(\boldsymbol\phi_{u_r,j}^{\mathrm{P}},\cero;\,\boldsymbol\phi_{u_r,i}^{\mathrm{P}},\cero\bigr),\\
  \mathcal B_{f,\Gamma} =  b_{\Gamma}\!\bigl(\boldsymbol\phi_{u_f,j}^{\mathrm{S}},\cero,\cero;\,\boldsymbol\phi_{u_f,i}^{\mathrm{S}},\cero\bigr), \quad
  \mathcal B_{p,\Gamma} =  b_{\Gamma}\!\bigl(\cero,\boldsymbol\phi_{u_r,j}^{\mathrm{P}},\cero;\,\boldsymbol\phi_{u_f,i}^{\mathrm{S}},\cero\bigr),  \\
  \mathcal B_{s,\Gamma} = b_{\Gamma}\!\bigl(\cero,\cero,\boldsymbol\phi_{y_s,j}^{\mathrm{P}};\,\boldsymbol\phi_{u_f,i}^{\mathrm{S}},\cero\bigr),  \quad \mathcal C_{f,\Gamma} =   b_{\Gamma}\!\bigl(\boldsymbol\phi_{u_f,j}^{\mathrm{S}},\cero,\cero;\,\cero,\phi_{p,i}^{\mathrm{S}}\bigr), \\
  \mathcal C_{p,\Gamma} = b_{\Gamma}\!\bigl(\cero,\boldsymbol\phi_{u_r,j}^{\mathrm{P}},\cero;\,\cero,\phi_{p,i}^{\mathrm{S}}\bigr), \quad
  \mathcal C_{s,\Gamma} =  b_{\Gamma}\!\bigl(\cero,\cero,\boldsymbol\phi_{y_s,j}^{\mathrm{P}};\,\cero,\phi_{p,i}^{\mathrm{S}}\bigr), \\
  \mathcal N_{ff,\Gamma} = c_{\Gamma}\!\bigl(\boldsymbol\phi_{u_f,j}^{\mathrm{S}},\cero,\cero;\,\boldsymbol\phi_{u_f,i}^{\mathrm{S}},\cero,\cero\bigr), \quad
  \mathcal N_{rr,\Gamma} = c_{\Gamma}\!\bigl(\cero,\boldsymbol\phi_{u_r,j}^{\mathrm{P}},\cero;\,\cero,\boldsymbol\phi_{u_r,i}^{\mathrm{P}},\cero\bigr), \\
  \mathcal N_{ss,\Gamma} = c_{\Gamma}\!\bigl(\cero,\cero,\boldsymbol\phi_{y_s,j}^{\mathrm{P}};\,\cero,\cero,\boldsymbol\phi_{y_s,i}^{\mathrm{P}}\bigr),  \quad 
  \mathcal N_{fr,\Gamma} = c_{\Gamma}\!\bigl(\boldsymbol\phi_{u_f,j}^{\mathrm{S}},\cero,\cero;\,\cero,\boldsymbol\phi_{u_r,i}^{\mathrm{P}},\cero\bigr), \\
  \mathcal N_{fs,\Gamma} =  c_{\Gamma}\!\bigl(\boldsymbol\phi_{u_f,j}^{\mathrm{S}},\cero,\cero;\,\cero,\cero,\boldsymbol\phi_{y_s,i}^{\mathrm{P}}\bigr) , \quad
  \mathcal N_{rf,\Gamma} = c_{\Gamma}\!\bigl(\cero,\boldsymbol\phi_{u_r,j}^{\mathrm{P}},\cero;\,\boldsymbol\phi_{u_f,i}^{\mathrm{S}},\cero,\cero\bigr), \\
  \mathcal N_{rs,\Gamma} = c_{\Gamma}\!\bigl(\cero,\boldsymbol\phi_{u_r,j}^{\mathrm{P}},\cero;\,\cero,\cero,\boldsymbol\phi_{y_s,i}^{\mathrm{P}}\bigr), \quad
  \mathcal N_{sf,\Gamma} = c_{\Gamma}\!\bigl(\cero,\cero,\boldsymbol\phi_{y_s,j}^{\mathrm{P}};\,\boldsymbol\phi_{u_f,i}^{\mathrm{S}},\cero,\cero\bigr), \\
  \mathcal N_{sr,\Gamma} =  c_{\Gamma}\!\bigl(\cero,\cero,\boldsymbol\phi_{y_s,j}^{\mathrm{P}};\,\cero,\boldsymbol\phi_{u_r,i}^{\mathrm{P}},\cero\bigr),  \quad 
  \mathcal{C} =  \boldsymbol{c}(\bw^{\mathrm{S}}_{f,h},\boldsymbol\phi_{u_f,i}^{\mathrm{S}} , \boldsymbol\phi_{u_f,j}^{\mathrm{S}}). 
\end{gather*}

 For the sake of the forthcoming analysis, we use $\partial_t \by_{s}^{\mathrm{P}}$ instead of $\bu_{s}^{\mathrm{P}}$ in the poroelastic region. Therefore we can rewrite the weak formulation \eqref{semi_weak} in the DAE system form as  
\begin{equation}\label{mixed-primal}
\mathbf{M} \partial_t \vec{\bx}_h(t)+\mathbf{N} \vec{\bx}_h(t)=L(t),\end{equation}
where
\begin{gather*}
\resizebox{0.95\linewidth}{!}{$
L = \begin{pmatrix}
\mathcal{F}_{\bu_{f,h}^{\mathrm{S}}}\\
\mathcal{F}_{\bu_r}\\
\mathcal{F}_{\by_{s,h}^{\mathrm{P}}}\\
0\\
\mathcal{F}_{p_h^{\mathrm{S}}}\\
\mathcal{F}_{p_h^{\mathrm{P}}}
\end{pmatrix},\quad
\mathbf{M} = \begin{pmatrix}
\mathcal{M}_{\rho_f} & 0 & (\mathcal{A}_{fs}^{\mathrm{BJS}})^* + \mathcal{N}_{sf,\Gamma} + \mathcal{B}_{e,\Gamma}^* & 0 & 0 & 0 \\
0 & \mathcal{M}_{\rho_f\phi} & \mathcal{M}_{-\theta} + \mathcal{A}_f^{\mathrm{P}} + \mathcal{N}_{sr,\Gamma} & \mathcal{M}_{\rho_f\phi} & 0 & 0 \\
0 & \mathcal{M}_{\rho_f\phi} & A_{ss}^{\mathrm{BJS}} + \mathcal{A}_f^{\mathrm{P}} + \mathcal{M}_{-\theta} + \mathcal{N}_{ss,\Gamma} & \mathcal{M}_{\rho_p} & 0 & 0 \\
0 & 0 & \mathcal{M}_{-\rho_p} & 0 & 0 & 0 \\
0 & 0 & (\mathcal{C}_{s,\Gamma})^* & 0 & 0 & 0 \\
0 & 0 & -(\mathcal{B}_s^{\mathrm{P}})^* & 0 & 0 & \mathcal{M}_{\frac{(1-\phi)^2}{K}}
\end{pmatrix},
$} \\
\resizebox{0.95\linewidth}{!}{$
\mathbf{N} = 
\begin{pmatrix}
\mathcal{A}_f^{\mathrm{S}}
  +\mathcal{A}_{ff}^{\mathrm{BJS}}
  +\mathcal{B}_{f,\Gamma}^*
  +\mathcal{B}_{f,\Gamma}
  +\mathcal{N}_{ff,\Gamma} 
  +\mathcal{C}
  & \mathcal{N}_{rf,\Gamma}
  +\mathcal{B}_{p,\Gamma}^*
  & 0 & 0
  & \mathcal{B}^{\mathrm{S}}
  +\mathcal{C}_{f,\Gamma}
  & 0 \\[6pt]
\mathcal{B}_{p,\Gamma}
  +\mathcal{N}_{fr,\Gamma}
  & \mathcal{A}_f^{\mathrm{P}}
    +\mathcal{M}_{-\theta}
    +\mathcal{M}_{\phi^2/\kappa}
    +\mathcal{N}_{rr,\Gamma}
    +\mathcal{A}_{rr}^{\mathrm{BJS}}
  & 0 & 0
  & \mathcal{C}_{p,\Gamma}
  & \mathcal{B}_{f}^{\mathrm{P}} \\[6pt]
\mathcal{A}_{fs}^{\mathrm{BJS}}
  +\mathcal{B}_{e,\Gamma}
  +\mathcal{N}_{fs,\Gamma}
  & \mathcal{A}_f^{\mathrm{P}}
    +\mathcal{M}_{-\theta}
    +\mathcal{N}_{rs,\Gamma}
  & \mathcal{A}_s^{\mathrm{P}}
  & 0
  & \mathcal{C}_{s,\Gamma}
  & \mathcal{B}_{s}^{\mathrm{P}} \\[6pt]
0 & 0 & 0 & \mathcal{M}_{\rho_{p}} & 0 & 0 \\[6pt]
-(\mathcal{B}^{\mathrm{S}})^*
  +(\mathcal{C}_{f,\Gamma})^*
  & (\mathcal{C}_{p,\Gamma})^*
  & 0 & 0 & 0 & 0 \\[6pt]
0 & -(\mathcal{B}_{f}^{\mathrm{P}})^* & 0 & 0 & 0 & 0
\end{pmatrix}.
$}
\end{gather*}
We readily note that the  matrix $\mathbf{M}+\mathbf{N}$  
yields a generalized saddle-point structure 
of the form 
\[
\mathbf{M} + \mathbf{N}
=
\begin{pmatrix}
\mathbf{A} & \mathbf{B}^T\\
-\mathbf{B}    & \mathbf{C}
\end{pmatrix},
\]
with
\[
\mathbf{A} =
\resizebox{0.925\textwidth}{!}{$
\begin{pmatrix}
\mathcal{M}_{\rho_f}
  + \mathcal{A}_f^{\mathrm{S}}
  + \mathcal{A}_{f f}^{\mathrm{BJS}}
  + \mathcal{B}_{f,\Gamma}^*
  + \mathcal{B}_{f,\Gamma}
  + \mathcal{N}_{ff, \Gamma} + \mathcal{C}
&
\mathcal{N}_{rf, \Gamma}
  + \mathcal{B}_{p,\Gamma}^*
&
(\mathcal{A}_{fs}^{\mathrm{BJS}})^*
  + \mathcal{N}_{sf, \Gamma}
  + \mathcal{B}_{e, \Gamma}^*
& 0
\\[4pt]
\mathcal{B}_{p,\Gamma}
  + \mathcal{N}_{fr, \Gamma}
&
\!\!\!\!\!\!\!\!\!\!\!\!\!\!\!\!\!\!\!\!\!\mathcal{A}_f^{\mathrm{P}}
  + \mathcal{M}_{-\theta}
  + \mathcal{M}_{\phi^2/\kappa}
  + \mathcal{N}_{pp, \Gamma}
  + \mathcal{M}_{\rho_{f}\phi}
&
\mathcal{M}_{-\theta}
  + \mathcal{A}^{\mathrm{P}}_f
  + \mathcal{N}_{sr, \Gamma}
& 
\mathcal{M}_{\rho_{f}\phi}
\\[4pt]
\mathcal{A}_{fs}^{\mathrm{BJS}}
  + \mathcal{B}_{e, \Gamma}
  + \mathcal{N}_{fs, \Gamma}
&
\!\!\!\mathcal{A}_f^{\mathrm{P}}
  + \mathcal{M}_{- \theta}
  + \mathcal{N}_{rs, \Gamma}
  + \mathcal{M}_{\rho_{f}\phi}
&
\!\!\!\!\!\!\!\mathcal{A}_s^{\mathrm{P}}
  + A_{ss}^{\mathrm{BJS}}
  + \mathcal{A}^{\mathrm{P}}_f
  + \mathcal{M}_{-\theta}
  + \mathcal{N}_{ss, \Gamma}
& 
\mathcal{M}_{\rho_{p}}
\\[4pt]
0 & 0 
& \mathcal{M}_{-\rho_{p}}
& \mathcal{M}_{\rho_{p}}
\end{pmatrix},
$}
\]
\[
\mathbf{B}^T =
\begin{pmatrix}
\mathcal{B}^{\mathrm{S}} + \mathcal{C}_{f,\Gamma} & 0 \\
\mathcal{C}_{p,\Gamma}                            & \mathcal{B}_{f}^{\mathrm{P}} \\
\mathcal{C}_{s,\Gamma}                            & \mathcal{B}_{s}^{\mathrm{P}} \\
0                                                 & 0
\end{pmatrix},
\quad
\mathbf{C} =
\begin{pmatrix}
0 & 0\\
0 & \displaystyle \mathcal{M}_{\tfrac{(1-\phi)^2}{K}}
\end{pmatrix}.
\]

We now prove an auxiliary estimate that will be used in our subsequent analysis. For this we employ the trace, Hölder, and Young inequalities.  Let $\bu_h \in \mathbf{H}^1(\Omega), q_h \in L^2(\Omega)$. Then we have:
\begin{subequations}
\begin{align}
2 \mu_f(\boldsymbol{\beps}(\boldsymbol{u}_h) \boldsymbol{n} \cdot \boldsymbol{n}, \boldsymbol{u}_h \cdot \boldsymbol{n})_{\Sigma} & \leq 2 \mu_f\sum_{E \in \mathcal{E}_{\Sigma}}\|\boldsymbol{\beps}(\boldsymbol{u}_h) \boldsymbol{n}\|_{0, E}\|\boldsymbol{u}_h \cdot \boldsymbol{n}\|_{0, E} \nonumber \\
& \leq 2 \mu_f \sum_{E \in \mathcal{E}_{\Sigma}}\bigl(\frac{h_E \delta_1}{2}\|\boldsymbol{\beps}(\boldsymbol{u}_h)\|_{0, E}^2+\frac{h_E^{-1}}{2 \delta_1}\|\boldsymbol{u}_h \cdot \boldsymbol{n}\|_{0, E}^2\bigr) \nonumber \\&
 \leq \delta_1 C_{\mathrm{tr}}^2 \mu_f \|\boldsymbol{\beps}(\boldsymbol{u}_h)\|_{0, \Omega}^2+ \frac{\mu_f}{\delta_1} \sum_{E \in \mathcal{E}_{\Sigma}} \frac{1}{h_E}\|\boldsymbol{u}_h \cdot \boldsymbol{n}\|_{0, E}^2, \label{I1} \\
  (q_h,\boldsymbol{u}_h\cdot \boldsymbol{n})_{\Sigma}
  &\le 
    \sum_{E\in\mathcal{E}_{\Sigma}}
      \|q_h\|_{0,E}\,\|\boldsymbol{u}_h\cdot\boldsymbol{n}\|_{0,E}
\nonumber \\
& 
  \le 
    \sum_{K\in\mathcal{T}_h}
      \frac{C_{\mathrm{tr}}^2\,\delta_2}{2\,\mu_f}\,\|q_h\|_{0,K}^2
    +\frac{1}{2\,\delta_2}
      \sum_{E\in\mathcal{E}_{\Sigma}}
        \frac{\mu_f}{h_E}\,\|\boldsymbol{u}_h\cdot\boldsymbol{n}\|_{0,E}^2. \label{I2} 
\end{align}
\end{subequations}

\begin{lemma}\label{coercivity-continuity}
 Under Assumptions~\ref{(H1)}--\ref{(H3)} and $ \frac{4}{C_{\mathrm{I}} ^2} \| \boldsymbol{f}_{\mathrm{S}} \|_{0,\Omega_{\mathrm{S}}} < 1 $, the linear operator $\mathbf{A}$ is continuous and elliptic with constant $C_{\alpha}$ (defined below), for any given $ \bw_{f,h}^{\mathrm{S}} \in \mathbf{K}_h$.
 \end{lemma}
\begin{proof}
From trace, Cauchy--Schwarz, Young inequalities,  and the bound in \eqref{I1}, it follows that there exist  constants $C_f, C_s, C_r, C_{\mathrm{BJS}}>0$  
 such that 
\begin{align*}
a_f^{\mathrm{S}}(\bu_{f,h}^{\mathrm{S}},\bv_{f,h}^{\mathrm{S}}) &\leq C_f \|\bu_{f,h}^{\mathrm{S}}\|_{1,\Omega_{\mathrm{S}}} \|\bv_{f,h}^{\mathrm{S}}\|_{1,\Omega_{\mathrm{S}}} 
, \\  a_f^{\mathrm{P}}(\by_{s,h}^{\mathrm{P}},\bw_{s,h}^{\mathrm{P}}) -m_{\theta}( \by_{s,h}^{\mathrm{P}},\bw_{s,h}^{\mathrm{P}}) &\leq C_s \|\by_{s,h}^{\mathrm{P}}\|_{1,\Omega_{\mathrm{P}}} \|\bw_{s,h}^{\mathrm{P}}\|_{1,\Omega_{\mathrm{P}}}, \\ 
a_f^{\mathrm{P}}(\bu_{r,h}^{\mathrm{P}},\bv_{r,h}^{\mathrm{P}}) -m_{\theta}( \bu_{r,h}^{\mathrm{P}},\bv_{r,h}^{\mathrm{P}}) + m_{\frac{\phi^2}{\kappa}}(\bu_{r,h}^{\mathrm{P}},\bv_{r,h}^{\mathrm{P}}) &\leq C_r \|\bu_{r,h}^{\mathrm{P}}\|_{1,\Omega_{\mathrm{P}}} \|\bv_{r,h}^{\mathrm{P}}\|_{1,\Omega_{\mathrm{P}}}, 
\end{align*}
and
\begin{align*}
 a_{\mathrm{BJS}}(\bu_{f,h}^{\mathrm{S}}, \bu_{s,h}^{\mathrm{P}}; \bv_{f,h}^{\mathrm{S}}, \bv_{s,h}^{\mathrm{P}}) +  b_{\mathrm{BJS}}(\bu_r^{\mathrm{P}}; \bv_r^{\mathrm{P}}) 
&\leq C_{\mathrm{BJS}}(\|\bu_{f,h}^{\mathrm{S}}\|_{1,\Omega_{\mathrm{S}}}+\|\bu_{s,h}^{\mathrm{P}}\|_{1,\Omega_{\mathrm{P}}}+ \|\bu_r^{\mathrm{P}}\|_{1,\Omega_{\mathrm{P}}}) \\ & \quad  \qquad \times (\|\bv_{f,h}^{\mathrm{S}}\|_{1,\Omega_{\mathrm{S}}}+\|\bv_{s,h}^{\mathrm{P}}\|_{1,\Omega_{\mathrm{P}}}  +\|\bv_r^{\mathrm{P}}\|_{1,\Omega_{\mathrm{P}}}), \\
 b_{\Gamma}(\bv_{f,h}^{\mathrm{S}}, \bv_{r,h}^{\mathrm{P}}, \bw_{s,h}^{\mathrm{P}} ; \bu_{f,h}^{\mathrm{S}}, 0 )   
 & \leq C \| \bu_{f,h}^{\mathrm{S}} \|_{1,\Omega_{\mathrm{S}}}  \bigl(   \frac{\gamma \mu_f}{h_E}  \|\bv_{f,h}^{\mathrm{S}}\! \cdot \!\bn_{\mathrm{S}}\! + \bv_{r,h}^{\mathrm{P}}\! \cdot \!\bn_{\mathrm{P}} \!+ \bw_{s,h}^{\mathrm{P}}\! \cdot\! \bn_{\mathrm{P}}  \|_{0,E}  ^2 \bigr)^{\frac12}, \\ 
 c_{\Gamma}(\bu_{f,h}^{\mathrm{S}}, \bu_{r,h}^{\mathrm{P}}, \by_{s,h}^{\mathrm{P}} ; \bv_{f,h}^{\mathrm{S}}, \bv_{r,h}^{\mathrm{P}}, \bw_{s,h}^{\mathrm{P}} )  & \leq C \frac{\gamma \mu_f}{h_E}   \| \bu_{f,h}^{\mathrm{S}}  \cdot \bn_{\mathrm{P}} + \bu_{r,h}^{\mathrm{P}}  \cdot \bn_{\mathrm{P}} + \by_{s,h}^{\mathrm{P}} \cdot \bn_{\mathrm{P}}\|_{0,E} \\ &  \qquad  \quad  \times \| \bv_{f,h}^{\mathrm{S}}  \cdot \bn_{\mathrm{P}} + \bv_{r,h}^{\mathrm{P}}  \cdot \bn_{\mathrm{P}} + \bw_{s,h}^{\mathrm{P}} \cdot \bn_{\mathrm{P}}\|_{0,E}.
\end{align*}

For the nonlinear term, we apply the Cauchy--Schwarz, Sobolev, and Korn's inequalities to obtain
\begin{align*}
    \boldsymbol{c}(\bw_{f,h}^{\mathrm{S}}, \bu_{f,h}^{\mathrm{S}}, \bv_{f,h}^{\mathrm{S}}) 
    &\leq  S_f^2 K_f^3 \|\varepsilon (\bw_{f,h}^{\mathrm{S}})\|_{0, \Omega_{\mathrm{S}}} \, \|\varepsilon ( \bu_{f,h}^{\mathrm{S}})\|_{0, \Omega_{\mathrm{S}}} \|\varepsilon (\bv_{f,h}^{\mathrm{S}})\|_{0, \Omega_{\mathrm{S}}}
    \leq \frac{C_{\mathrm{I}}}{4} \, \|\varepsilon ( \bu_{f,h}^{\mathrm{S}})\|_{0, \Omega_{\mathrm{S}}}\|\varepsilon ( \bv_{f,h}^{\mathrm{S}})\|_{0, \Omega_{\mathrm{S}}},
\end{align*}
and thus, the operator $\mathbf{A}$ is continuous.
 
On the other hand, we use Sobolev, Poincar\'e, Korn's inequalities \cite{MR2373954} and the assumptions in Section \ref{assumptions}, to arrive at the bounds
\begin{align*}
 a_f^{\mathrm{S}}(\bv_{f,h}^{\mathrm{S}},\bv_{f,h}^{\mathrm{S}}) + \boldsymbol{c}(\bw_{f,h}^{\mathrm{S}}, \bv _{f,h}^{\mathrm{S}}, \bv_{f,h}^{\mathrm{S}})  
& \geq \left(2 \mu_f - \frac{C_{\mathrm{I}}}{4}\right) \| \varepsilon(\bv_{f,h}^{\mathrm{S}})\|_{0,\Omega_{\mathrm{S}}}^2
,\\ 
 a_f^{\mathrm{P}}(\bw_{s,h}^{\mathrm{P}},\bw_{s,h}^{\mathrm{P}}) -m_{\theta}( \bw_{s,h}^{\mathrm{P}},\bw_{s,h}^{\mathrm{P}}) & \geq \min \{ 2 \mu_f \phi,  C_P^2 \} \|\bw_{s,h}^{\mathrm{P}}\|_{1,\Omega_{\mathrm{P}}}^2,  \\ 
   a_f^{\mathrm{P}}(\bv_{r,h}^{\mathrm{P}},\bv_{r,h}^{\mathrm{P}}) -m_{\theta}( \bv_{r,h}^{\mathrm{P}},\bv_{r,h}^{\mathrm{P}}) + m_{\phi^2/\kappa}(\bv_{r,h}^{\mathrm{P}},\bv_{r,h}^{\mathrm{P}}) & \geq \min \{ 2 \mu_f \phi,  C_P^2, \frac{\phi^2 C_P^2}{\kappa} \} \|\bv_{r,h}^{\mathrm{P}}\|_{1,\Omega_{\mathrm{P}}}^2, \\ 
  a_{\mathrm{BJS}}(\bv_{f,h}^{\mathrm{S}}, \bw_{s,h}^{\mathrm{P}}; \bv_{f,h}^{\mathrm{S}}, \bw_{s,h}^{\mathrm{P}}) +b_{\mathrm{BJS}}( \bv_{r,h}^{\mathrm{P}} ;  \bv_{r,h}^{\mathrm{P}})   &  \geq \mu_f \alpha_{\mathrm{BJS}} Z_{\mathrm{max}}^{-1/2}
(|\bv_{f,h}^{\mathrm{S}} - \bw_{s,h}^{\mathrm{P}}|_{\mathrm{BJS}}^2 + |\bv_{r,h}^{\mathrm{P}}|_{\mathrm{BJS}}^2
), \\
 b_{\Gamma}(\bv_{f,h}^{\mathrm{S}}, \bv_{r,h}^{\mathrm{P}}, \bw_{s,h}^{\mathrm{P}} ; \bv_{f,h}^{\mathrm{S}}, 0 ) &  \geq - \delta_1 C_{\mathrm{tr}}^2 \mu_f \| \varepsilon(\bv_{f,h}^{\mathrm{S}})\|_{0,\Omega_{\mathrm{S}}}^2 \\
 & \quad - \frac{\mu_f}{\delta_1} \sum_{E \in \mathcal{E}_{\Sigma}} \frac{1}{h_E}  ( \|\boldsymbol{v}_{f,h}^{\mathrm{S}} \cdot \bn_{\mathrm{S}} + \boldsymbol{v}_{r,h}^{\mathrm{P}} \cdot \bn_{\mathrm{P}}   + \boldsymbol{w}_{s,h}^{\mathrm{P}} \cdot \bn_{\mathrm{P}}\|_{0, E}^2 ),  
\\ 
  c_{\Gamma}(\bv_{f,h}^{\mathrm{S}}, \bv_{r,h}^{\mathrm{P}}, \bw_{s,h}^{\mathrm{P}} ; \bv_{f,h}^{\mathrm{S}}, \bv_{r,h}^{\mathrm{P}}, \bw_{s,h}^{\mathrm{P}} ) & \geq  \frac{\gamma }{h_E} \| \bv_{f,h}^{\mathrm{S}} \cdot \bn_{\mathrm{S}}+ \bv_{r,h}^{\mathrm{P}} \cdot \bn_{\mathrm{P}} + \bw_{s,h}^{\mathrm{P}} \cdot \bn_{\mathrm{P}} \|_{0,E}^2.
\end{align*}
 \cblue{Then, combining the bilinear forms $a_f^{\mathrm{S}}, b_{\Gamma}$, and $c_{\Gamma}$, we arrive at}
\begin{align*}
 &  \cred{ a_f^{\mathrm{S}}(\bv_{f,h}^{\mathrm{S}},\bv_{f,h}^{\mathrm{S}})  +  b_{\Gamma}(\bv_{f,h}^{\mathrm{S}}, \bv_{r,h}^{\mathrm{P}}, \bw_{s,h}^{\mathrm{P}} ; \bv_{f,h}^{\mathrm{S}}, 0 )  +  c_{\Gamma}(\bv_{f,h}^{\mathrm{S}}, \bv_{r,h}^{\mathrm{P}}, \bw_{s,h}^{\mathrm{P}} ; \bv_{f,h}^{\mathrm{S}}, \bv_{r,h}^{\mathrm{P}}, \bw_{s,h}^{\mathrm{P}} ) } \\ & \qquad  \cred{  \geq ( 2 \mu_f - \delta_1 C_{\mathrm{tr}}^2 \mu_f) \| \varepsilon(\bv_{f,h}^{\mathrm{S}})\|_{0,\Omega_{\mathrm{S}}}^2 + (\gamma - \frac{\mu_f}{\delta_1}) \sum_{E \in \mathcal{E}_{\Sigma}} \frac{1}{h_E} ( \|\boldsymbol{v}_{f,h}^{\mathrm{S}} \cdot \bn_{\mathrm{S}} + \boldsymbol{v}_{r,h}^{\mathrm{P}} \cdot \bn_{\mathrm{P}}   + \boldsymbol{w}_{s,h}^{\mathrm{P}} \cdot \bn_{\mathrm{P}}\|_{0, E}^2 ).} 
\end{align*}
 \cred{By selecting the positive parameter $\delta_1$ in such a way that satisfies $\delta_1 < \frac{2 }{ C_{\mathrm{tr}}^2 }$, we ensure that $\left(2 \mu_f - \delta_1 C_{\mathrm{tr}}^2 \mu_f \right) >0 $. Additionally, we define }
\begin{align*}
    \cred{  C_S = \min \left\{ ( 2 \mu_f - \delta_1 C_{\mathrm{tr}}^2 \mu_f), (\gamma - \frac{\mu_f}{\delta_1}), \mu_f \alpha_{\mathrm{BJS}} Z_{\mathrm{max}}^{-1/2}, 2 \mu_f \phi,  C_P^2, \frac{\phi^2 C_P^2}{\kappa}  \right\},}
\end{align*}
  \cred{  with $\gamma \geq \gamma_0 > \frac{\mu_f}{\delta_1}$. Note that this is the coercivity constant for the operator without the bilinear form $\boldsymbol{c}$, and it arises when introducing the set $\boldsymbol{K}_h$. 
  Next, we 
  choose $\delta_1 = \frac{1}{C_{\mathrm{tr}}^2}$, from which  we obtain}
\begin{align*}
 & \cred{ a_f^{\mathrm{S}}(\bv_{f,h}^{\mathrm{S}},\bv_{f,h}^{\mathrm{S}}) + \boldsymbol{c}(\bw_{f,h}^{\mathrm{S}}, \bv _{f,h}^{\mathrm{S}}, \bv_{f,h}^{\mathrm{S}})  +  b_{\Gamma}(\bv_{f,h}^{\mathrm{S}}, \bv_{r,h}^{\mathrm{P}}, \bw_{s,h}^{\mathrm{P}} ; \bv_{f,h}^{\mathrm{S}}, 0 )  +  c_{\Gamma}(\bv_{f,h}^{\mathrm{S}}, \bv_{r,h}^{\mathrm{P}}, \bw_{s,h}^{\mathrm{P}} ; \bv_{f,h}^{\mathrm{S}}, \bv_{r,h}^{\mathrm{P}}, \bw_{s,h}^{\mathrm{P}} ) } \\ & \qquad    \cred{\geq ( \mu_f - \frac{C_{\mathrm{I}}}{4} ) \| \varepsilon(\bv_{f,h}^{\mathrm{S}})\|_{0,\Omega_{\mathrm{S}}}^2 + (\gamma - \mu_f C_{\mathrm{tr}}^2 ) \sum_{E \in \mathcal{E}_{\Sigma}} \frac{1}{h_E} ( \|\boldsymbol{v}_f^{\mathrm{S}} \cdot \bn_{\mathrm{S}} + \boldsymbol{v}_{r,h}^{\mathrm{P}} \cdot \bn_{\mathrm{P}}   + \boldsymbol{w}_{s,h}^{\mathrm{P}} \cdot \bn_{\mathrm{P}}\|_{0, E}^2 ). }
\end{align*}
\cred{This estimate implies that  $\mathbf{A}$ is coercive with constant }
\begin{align*}
    \cred{  C_{\alpha} = \min \left\{ (  \mu_f - \frac{C_{\mathrm{I}}}{4} ), (\gamma - \mu_f C_{\mathrm{tr}}^2 ), \mu_f \alpha_{\mathrm{BJS}} Z_{\mathrm{max}}^{-1/2}, 2 \mu_f \phi,  C_P^2, \frac{\phi^2 C_P^2}{\kappa}  \right\},}
\end{align*}
\cred{provided that $C_{\mathrm{I}} < 4 \mu_f$.}
\end{proof}
\begin{lemma}\label{continuity-b}
The operator $\mathbf{B}$ and its transpose $\mathbf{B}^{T}$ are bounded and continuous. 
\end{lemma}
\begin{proof}
    For all $\vec{\bv}_h=(\bv_{f,h}^{\mathrm{S}},\bv_{r,h}^{\mathrm{P}}, \bw_{s,h}^{\mathrm{P}}, ) \in \vec{\mathbf{V}}_h$ and $\vec{q}_h=( q_h^{\mathrm{S}}, q_h^{\mathrm{P}}) \in \vec{Q}_h$, we can apply trace, Cauchy--Schwarz, and Young's inequalities to have 
    \begin{align*}
        \langle \mathbf{B}(\vec{\bv}_h),\vec{q}_h\rangle &= (\nabla \cdot \bv_{f,h}^{\mathrm{S}}, q_h^{\mathrm{S}}) + (\nabla \cdot \bw_{s,h}^{\mathrm{P}}, q_h^{\mathrm{P}}) + (\nabla \cdot (\phi\bv_{r,h}^{\mathrm{P}}), q_h^{\mathrm{P}})  + \int_{\Sigma}q_h^{\mathrm{S}}(\bn_{\mathrm{S}} \cdot \bv_{f,h}^{\mathrm{S}} +\bn_{\mathrm{P}} \cdot \bv_{r,h}^{\mathrm{P}} + \bn_{\mathrm{P}} \cdot \bw_{s,h}^{\mathrm{P}}) \ds\\ &
        \lesssim \|\bv_{f,h}^{\mathrm{S}}\|_{1,\Omega_{\mathrm{S}}}^2 +\|\bv_{s,h}^{\mathrm{P}}\|_{1,\Omega_{\mathrm{P}}}^2 +\|\bv_{r,h}^{\mathrm{P}}\|_{1,\Omega_{\mathrm{P}}}^2  + \|q_h^{\mathrm{P}}\|_{0,\Omega_{\mathrm{P}}}^2  +  \!\!\sum_{K\in\mathcal{T}_h} \!\!
      \frac{C_{\mathrm{tr}}^2\,\delta_2}{2\,\mu_f}\,\|q_h^{\mathrm{S}}\|_{0,K}^2 \\ & \quad 
    +\frac{1}{2\,\delta_2}
      \sum_{E\in\mathcal{E}_{\Sigma}}
        \frac{\mu_f}{h_E} (\|\bn_{\mathrm{S}} \cdot \bv_{f,h}^{\mathrm{S}}\|_{0,E}^2 + \|\bn_{\mathrm{P}} \cdot \bv_{r,h}^{\mathrm{P}}\|_{0,E}^2+ \|\bn_{\mathrm{P}} \cdot \bw_{s,h}^{\mathrm{P}}\|_{0,E}^2), 
    \end{align*}
    where we have also used \eqref{I2}. 
\end{proof}

The next results help to establish the Ladyzhenskaya--Babu\v{s}ka--Brezzi (LBB) condition for the $\mathbf{B}$ block.
\begin{lemma}\label{usual_inf}
There exists a constant $\xi_1(\Omega)>0$ such that
$$
\inf _{\substack{\vec{q}_h \in \mathrm{W}_h }} \sup _{\substack{ \vec{\bv}_h \in \mathbf{V}_h }} \frac{b^{\mathrm{S}} (\bv_{f,h}^{\mathrm{S}}, q_h^{\mathrm{S}})+b_{f}^{\mathrm{P}}( \bv_{r,h}^{\mathrm{P}},q_h^{\mathrm{P}}) + b^{\mathrm{P}}_{s}(\bw_{s,h}^{\mathrm{P}}, q_h^{\mathrm{P}})}{\| \vec{\bv}_h\|_{V_h}\|\vec{q}_h\|_{W_h}} \geq \xi_1 >0 .
$$
\end{lemma}
\begin{proof}
This is proven by using the usual inf-sup condition for the Stokes problem \cite{MR3097958} and the weighted inf-sup condition in \cite[Lemma 14]{MR4253885}.
\end{proof}
\begin{lemma}\label{inf_sup}
There exists a constant $\xi_2 (\Omega) >0$, such that
\begin{align*}
   \inf _{\substack{\vec{q}_h \in \mathrm{W}_h }} \sup _{\substack{ \vec{\bv}_h \in  \cred{ \mathbf{V}_h }} } \frac{
 \langle \mathbf{B}(\vec{\bv}_h),\vec{q}_h\rangle}
{\left|\!\left|\!\left|  \vec{\bv}_h \right|\!\right|\!\right| \|\vec{q}_h\|_{W_h}}   \geq \xi_2,
\end{align*}
where 
\begin{align*}
    \langle \mathbf{B}(\vec{\bv}_h),\vec{q}_h\rangle & := b^{\mathrm{S}} (\bv_{f,h}^{\mathrm{S}}, q_h^{\mathrm{S}})+b_{f}^{\mathrm{P}}( \bv_{r,h}^{\mathrm{P}},q_h^{\mathrm{P}}) + b^{\mathrm{P}}_{s}(\bw_{s,h}^{\mathrm{P}}, q_h^{\mathrm{P}})  + \mathcal{C}_{f,\Gamma}(\bv_{f,h}^{\mathrm{S}},  q_h^{\mathrm{S}} ) + \mathcal{C}_{p,\Gamma} (\bv_{r,h}^{\mathrm{P}}, q_h^{\mathrm{S}} )+ \mathcal{C}_{s,\Gamma} ( \bw_{s,h}^{\mathrm{P}}, q_h^{\mathrm{S}} ),\end{align*}
    and 
    \begin{align*}
 \left|\!\left|\!\left| \vec{\bv}_h  \right|\!\right|\!\right|  &:= \| \bv_{f, h}^{\mathrm{S}} \|_{1,\Omega_{\mathrm{S}}} + \| \bv_{r, h}^{\mathrm{P}} \|_{1,\Omega_{\mathrm{P}}} + \|  \bw_{s,h}^{\mathrm{P}} \|_{1,\Omega_{\mathrm{P}}}   + \frac{1}{h_E} \sum_{E \in \mathcal{E}_{\Sigma}}  \| \bn_{\mathrm{S}}\cdot v_{f,h}^{\mathrm{S}}
    +\bn_{\mathrm{P}}\cdot v_{r,h}^{\mathrm{P}}
    +\bn_{\mathrm{P}}\cdot w_{s,h}^{\mathrm{P}} \|_{0,E}.
    \end{align*}
\end{lemma} 
\begin{proof}
The term $b_{\Gamma}(\bv_{f,h}^{\mathrm{S}}, \bv_{r,h}^{\mathrm{P}}, \bw_{s,h}^{\mathrm{P}} ; \cero , q_h^{\mathrm{S}} )$ is defined as 
\begin{align*}
       b_{\Gamma}(\bv_{f,h}^{\mathrm{S}}, \bv_{r,h}^{\mathrm{P}}, \bw_{s,h}^{\mathrm{P}} ; \cero, q_h^{\mathrm{S}} )  & \coloneqq  \mathcal{C}_{f,\Gamma}(\bv_{f,h}^{\mathrm{S}},  q_h^{\mathrm{S}} ) + \mathcal{C}_{p,\Gamma} (\bv_{r,h}^{\mathrm{P}}, q_h^{\mathrm{S}} )+ \mathcal{C}_{s,\Gamma} ( \bw_{s,h}^{\mathrm{P}}, q_h^{\mathrm{S}} ) \\ & = \int_{\Sigma} q_h^{\mathrm{S}} (\bn_{\mathrm{S}} \cdot \bv_{f,h}^{\mathrm{S}} +\bn_{\mathrm{P}} \cdot \bv_{r,h}^{\mathrm{P}} + \bn_{\mathrm{P}} \cdot \bw_{s,h}^{\mathrm{P}}) \ds. 
\end{align*}
\cred{We follow the same idea discussed in \cite{bansal2024nitsche}.} Consider the discrete space associated with the strong imposition of the mass balance across the interface
\begin{equation}\label{def:V0}\mathbf{V}_{h,0} = \{ \boldsymbol{v}_h \in \mathbf{V}_h : \bn_{\mathrm{S}} \cdot \bv_{f,h}^{\mathrm{S}} +\bn_{\mathrm{P}} \cdot \bv_{r,h}^{\mathrm{P}} + \bn_{\mathrm{P}} \cdot \bw_{s,h}^{\mathrm{P}} =0 \,\text{ on }\, \Sigma\}.\end{equation}
This space naturally yields that $b_{\Gamma}(\bv_{f,h}^{\mathrm{S}}, \bv_{r,h}^{\mathrm{P}}, \bw_{s,h}^{\mathrm{P}} ; \cero, q_h^{\mathrm{S}} ) = 0$ for all $\boldsymbol{v}_h \in \mathbf{V}_{h,0}$, so we obtain
\begin{align*}
  		\sup _{\cero \neq \vec{\bv}_h \in \mathbf{V}_{h,0}} \frac{ \langle \mathbf{B}(\vec{\bv}_h),\vec{q}_h\rangle }{\| \vec{\bv}_h\|_{\mathbf{V}_{h}}} & = \sup_{\boldsymbol{0} \neq \vec{\bv}_h \in \mathbf{V}_{h,0}} \frac{b^{\mathrm{S}} (\bv_{f,h}^{\mathrm{S}}, q_h^{\mathrm{S}})+b_{f}^{\mathrm{P}}( \bv_{r,h}^{\mathrm{P}},q_h^{\mathrm{P}}) + b^{\mathrm{P}}_{s}(\bw_{s,h}^{\mathrm{P}}, q_h^{\mathrm{P}}) }{\| \vec{\bv}_h\|_{\mathbf{V}_{h}}}   \geq \xi_1 \| \vec{q}_h \|_{W_h} \quad \forall \vec{q}_h \in \mathrm{W}_h,
\end{align*}
by virtue of Lemma \ref{usual_inf}.  Now, we can use this space to derive a lower bound as follows
\begin{align*} 
  		\sup _{\cero \neq \vec{\bv}_h \in  \cred{  \mathbf{V}_{h}} }\frac{\langle \mathbf{B}(\vec{\bv}_h),\vec{q}_h\rangle}{ \left|\!\left|\!\left| \vec{\bv}_h \right|\!\right|\!\right|} \geq \sup _{\cero \neq \vec{\bv}_h \in \mathbf{V}_{h,0}} \frac{\langle \mathbf{B}(\vec{\bv}_h),\vec{q}_h\rangle}{\| \vec{\bv}_h\|_{\mathbf{V}_{h}}} \geq \xi_2 \| \vec{q}_h \|_{W_h} \quad \forall \vec{q}_h \in \mathrm{W}_h, 
\end{align*}
which concludes the proof. 
\end{proof}

Before presenting the proof of existence and uniqueness of a solution to the discrete problem \eqref{mixed-primal}, we require some additional results.

We denote the bilinear forms associated with the matrices \(\mathbf{A}\), \(\mathbf{B}\), and \(\mathbf{C}\) by $\phi_{\mathbf{A}}(\bullet,\bullet)$, $\phi_{\mathbf{B}}(\bullet,\bullet)$, and $\phi_{\mathbf{C}}(\bullet,\bullet)$,  defined on  $(\mathbf{V}_h \times \mathbf{W}_{s,h}) \times (\mathbf{V}_h \times \mathbf{W}_{s,h})$, 
 $\mathbf{V}_h \times \mathrm{W}_h$, and 
$\mathrm{W}_h \times \mathrm{W}_h$, respectively. For a given $\bw_{f,h}^{\mathrm{S}} \in \mathbf{K}_h$, we write 
\begin{gather*}
\phi_{\mathbf{A}}(\bw_{f,h}^{\mathrm{S}}; (\vec{\bu}_h, \bu_{s, h}^{\mathrm{S}}), (\vec{\bv}_h, \bv_{s, h}^{\mathrm{S}})) 
= \begin{pmatrix} \vec{\bv}_h \\ \bv_{s, h}^{\mathrm{S}} \end{pmatrix}^T \!
\mathbf{A} 
\begin{pmatrix} \vec{\bu}_h \\ \bu_{s, h}^{\mathrm{S}} \end{pmatrix}, \quad 
\phi_{\mathbf{B}}(\vec{\bv}_h, \vec{p}_h) 
= \vec{\bv}_h^T \mathbf{B} \vec{p}_h, \quad 
\phi_{\mathbf{C}}(\vec{p}_h, \vec{q}_h) 
= \vec{q}_h^T \mathbf{C} \vec{p}_h.
\end{gather*}
By identifying functions in the $\mathrm{FE}$ spaces with algebraic vectors of their corresponding degrees of freedom, we note that $\operatorname{ker}(\phi_{\mathbf{A}})=\operatorname{ker}(\mathbf{A}), \; \operatorname{ker}(\phi_{\mathbf{B}})=\operatorname{ker}(\mathbf{B})$, and $\operatorname{ker}(\phi_{\mathbf{C}})=\operatorname{ker}(\mathbf{C})$. Also, for $\phi_{\mathbf{B}^T}(\vec{q}_h,\vec{\bv}_h)=$ $\phi_{\mathbf{B}}(  \vec{\bv}_h , \vec{q}_h )$, we have that $\operatorname{ker}(\phi_{\mathbf{B}^T})=\operatorname{ker}(\mathbf{B}^T)$. 
\begin{remark}
   \cred{ We introduce the bilinear forms $\phi_{\mathbf{A}}(\bullet,\bullet)$, $\phi_{\mathbf{B}}(\bullet,\bullet)$, and $\phi_{\mathbf{C}}(\bullet,\bullet)$ in order to simplify the presentation of the algebraic arguments in Lemma \ref{invertibilty}. Although these notations are only used locally, they allow us to express kernel properties and coercivity arguments in a more transparent functional framework.}
\end{remark}
\begin{lemma}\label{invertibilty}
Under Assumptions~\ref{(H1)}-\ref{(H3)} and $ \frac{4}{C_{\mathrm{I}} ^2} \| \boldsymbol{f}_{\mathrm{S}} \|_{0,\Omega_{\mathrm{S}}} < 1 $, the bilinear forms $\phi_{\mathbf{A}}(\bullet,\bullet), \phi_{\mathbf{B}}(\bullet,\bullet)$ and $\phi_{\mathbf{C}}(\bullet,\bullet)$ satisfy
$$
\operatorname{ker}(\phi_{\mathbf{A}}) \cap \operatorname{ker}(\phi_{\mathbf{B}}) =\{ \cero \}, \quad 
\operatorname{ker}(\phi_{\mathbf{C}}) \cap \operatorname{ker}(\phi_{\mathbf{B}^T})  =\{0 \},
$$
for any given $ \bw_{f,h}^{\mathrm{S}} \in \mathbf{K}_h$. Moreover, $\phi_{\mathbf{A}}(\bullet,\bullet)$ and $\phi_{\mathbf{C}}(\bullet,\bullet)$ are positive definite and semi-definite, respectively.
\end{lemma}
\begin{proof}
Lemma \ref{coercivity-continuity} implies the coercivity of $\phi_{\mathbf{A}}(\bullet,\bullet)$ and $\operatorname{ker}(\phi_{\mathbf{A}})=\{\cero\}$, hence the first statement of the lemma follows. We next note that $\operatorname{ker}(\phi_{\mathbf{B}^T})$ consists of $\vec{q}_h \in \mathrm{W}_h$ such that
$$
\phi_{\mathbf{B}^T}(\vec{q}_h , \vec{\bv}_h )=0, \quad \forall \vec{\bv}_h  \in \mathbf{V}_h.
$$
Therefore, the inf-sup condition from Lemma~\ref{inf_sup} implies that $\operatorname{ker}(\phi_{\mathbf{B}^T})=\{0\}$, which gives the second statement of the lemma. The positive semi-definiteness of $\phi_{\mathbf{C}}(\bullet,\bullet)$ is straightforward.
\end{proof}

Now, we are in position to establish the well-posedness of the fixed-point operator $\mathcal{J}_h$.
\begin{lemma}\label{invertibilty1}
Under Assumptions~\ref{(H1)}--\ref{(H3)} and $ \frac{4}{C_{\mathrm{I}}^2} \| \boldsymbol{f}_{\mathrm{S}} \|_{0,\Omega_{\mathrm{S}}} < 1 $, if the matrices $\mathbf{A}$ and $\mathbf{C}$ are positive semi-definite and $\operatorname{ker}(\mathbf{A}) \cap \operatorname{ker}(\mathbf{B}) = \operatorname{ker}(\mathbf{C}) \cap \operatorname{ker}(\mathbf{B}^T) = \{0\}$, then the matrix $\mathbf{M} + \mathbf{N}$ is invertible for any given $ \bw_{f,h}^{\mathrm{S}} \in \mathbf{K}_h$.
 \end{lemma}

\begin{theorem}\label{E+H}
For each $\boldsymbol{f}_{\mathrm{S}} \in L^{\infty}(\mathrm{I};L^2(\Omega_{\mathrm{S}})), 
\boldsymbol{f}_{\mathrm{P}} \in L^{\infty}(0,T;L^2(\Omega_{\mathrm{P}})),$  
$\theta \in L^{\infty}(0,T;L^2(\Omega_{\mathrm{P}}))$, and compatible initial data $ ( \bu_{r,h}^{\mathrm{P}}(0), \bu_{s,h}^{\mathrm{P}}(0), \bu_{f,h}^{\mathrm{S}}(0), p^{\mathrm{P}}_h(0), \bsigma^{\mathrm{P}}_{h}(0) ) \in \mathbf{V}_{r,h}  \times \mathbf{V}_{s,h} \times \mathbf{V}_{f,h} \times \mathrm{W}_{p,h} \times \bZ_h$ under   assumptions \ref{(H1)}-\ref{(H3)}. Moreover, assume that $ \frac{4}{C_{\mathrm{I}}^2} \| \boldsymbol{f}_{\mathrm{S}} \|_{0,\Omega_{\mathrm{S}}} < 1 $, for any $\widehat{\epsilon}_f^{\prime}, \breve{\epsilon}_f^{\prime}$ that satisfy
$$
\frac{3- 2 (\varsigma+1) \widehat{\epsilon}_f^{\prime} C_{\mathrm{tr}}}{4} >0,
$$
where $\varsigma \in\{-1,0,1\}$ provided that $\gamma> (\varsigma + 1) (\widehat{\epsilon}^{\prime}_f)^{-1}$, there exists a unique solution $(\bu^{\mathrm{S}}_{f,h}, p^{\mathrm{S}}_h,  \bu_{r,h}^{\mathrm{P}}, p^{\mathrm{P}}_h,  \by_{s,h}^{\mathrm{P}} ,\bu_{s,h}^{\mathrm{P}})   \in  W^{1,\infty}(0, T ; \mathbf{V}_{f,h}) \times L^{\infty}(0, T ; \mathrm{W}_{f,h}) \times W^{1,\infty}(0, T ; \mathbf{V}_{r,h})$ $\times  W^{1,\infty}(0, T ; \mathrm{W}_{p,h}) \times W^{1,\infty}(0, T ; \mathbf{V}_{s,h}) \times $  $W^{1,\infty}(0, T ; \mathbf{W}_{s,h})$ of the weak formulation \eqref{mixed-primal}, for any given $ \bw_{f,h}^{\mathrm{S}} \in \mathbf{K}_h$.
\end{theorem}
\begin{proof}
The matrix \(\mathbf{M}\) has no zero rows, which implies that the system has no algebraic constraints. Hence, the initial data can be chosen to satisfy the prescribed boundary conditions. In particular, the initial values
\[
\bu_{f,h}^{\mathrm{S}}(0) = \bu_{f,0},\quad
\bu_{r,h}^{\mathrm{P}}(0) = \bu_{r,0},\quad
\by_{s,h}^{\mathrm{P}}(0) = \by_{s,0},\quad
p_h^{\mathrm{P}}(0) = p_h^{\mathrm{P},0},\quad
\bu_{s,h}^{\mathrm{P}}(0) = \bu_{s,0}
\]
are consistent and do not lead to any incompatibility issues. Furthermore, owing to Lemma~\ref{invertibilty1}, the matrix \(\mathbf{M} + \mathbf{N}\) with \(s = 1\) is invertible. According to the DAE theory (see \cite[Theorem~2.3.1]{MR1101809}), if the matrix pencil \(s\,\mathbf{M} + \mathbf{N}\) is nonsingular for some \(s\neq 0\) and the initial data are consistent, then the system \eqref{mixed-primal} admits a solution. Consequently, \cite[Theorem~2.3.1]{MR1101809} guarantees the existence of a solution to the weak semi-discrete formulation~\eqref{mixed-primal}.

To show uniqueness, we assume that there are two solutions satisfying these equations with the same initial conditions. Then, we readily have that their difference $(\tilde{\bu}_{f,h}^{\mathrm{S}} , \tilde{p}_h^{\mathrm{S}} ,\tilde{\bu}_{r,h}^{\mathrm{P}}, \tilde{p}_h^{\mathrm{P}}, \tilde{\by}_{s,h}^{\mathrm{P}}, \tilde{\bu}_{s,h}^{\mathrm{P}})$ satisfies \eqref{mixed-primal} with zero data. By setting $$(\bv_{f,h}^{\mathrm{S}}, q_h^{\mathrm{S}}, \bv_{r,h}^{\mathrm{P}}, q_h^{\mathrm{P}}, \bw_{s,h}^{\mathrm{P}}, \bv_{s,h}^{\mathrm{P}})
= (\tilde{\bu}_{f,h}^{\mathrm{S}}, \tilde{p}_h^{\mathrm{S}}, \tilde{\bu}_{r,h}^{\mathrm{P}}, \tilde{p}_h^{\mathrm{P}}, \partial_t \tilde{\by}_{s,h}^{\mathrm{P}},\tilde{\bu}_{s,h}^{\mathrm{P}}),$$ in \eqref{mixed-primal} and using the well-known inequality 
\begin{equation}\label{ineq:aux} 
-(\xi a, a) - (\xi b, b) \leq 2 (\xi a, b) \leq (\xi a, a) + (\xi b, b),
\end{equation}
we derive the following weak form of the energy balance
\begin{align*}
   & \frac{1}{2} \partial_t \Big( ( \rho_f \tilde{\bu}_{f,h}^{\mathrm{S}}, \tilde{\bu}_{f,h}^{\mathrm{S}}  ) + (\rho_s (1-\phi) \tilde{\bu}_{s,h}^{\mathrm{P}}, \tilde{\bu}_{s,h}^{\mathrm{P}})_{\Omega_{\mathrm{P}}}
   + ({(1-\phi)^2}{K}^{-1} \tilde{p}_h^{\mathrm{P}}, \tilde{p}_h^{\mathrm{P}} )_{\Omega_{\mathrm{P}}} \\
   & 
   + (\sqrt{\rho_f \phi} (\tilde{\bu}_{r,h}^{\mathrm{P}} + \tilde{\bu}_{s,h}^{\mathrm{P}}), \sqrt{\rho_f \phi} (\tilde{\bu}_{r,h}^{\mathrm{P}} + \tilde{\bu}_{s,h}^{\mathrm{P}}))_{\Omega_{\mathrm{P}}} 
    + (2 \mu_p \beps(\tilde{\by}_{s,h}^{\mathrm{P}}), \beps(\tilde{\by}_{s,h}^{\mathrm{P}}))_{\Omega_{\mathrm{P}}} \\
    & 
    + (\lambda_p \nabla \cdot \tilde{\by}_{s,h}^{\mathrm{P}}, \nabla \cdot \tilde{\by}_{s,h}^{\mathrm{P}})_{\Omega_{\mathrm{P}}} 
    \Big)  + \left| \tilde{\bu}_{f,h}^{\mathrm{S}} - \partial_t \tilde{\by}_{s,h}^{\mathrm{P}} \right|_{\mathrm{BJS}} ^2   + (2 \mu_f \beps(\tilde{\bu}_{f,h}^{\mathrm{S}}), \beps(\tilde{\bu}_{f,h}^{\mathrm{S}}))_{\Omega_{\mathrm{S}}}   \\ &   
  - \sum_{E \in \mathcal{E}_{\Sigma}} \int_{\Sigma} (1+ \varsigma) (2 \mu_f \beps (\tilde{\bu}_{f,h}^{\mathrm{S}}) \bn_{\mathrm{S}} \cdot \bn_{\mathrm{S}} )  ( \bn_{\mathrm{S}} \cdot \tilde{\bu}_{f,h}^{\mathrm{S}} + \bn_{\mathrm{P}} \cdot \tilde{\bu}_{r,h}^{\mathrm{P}} + \bn_{\mathrm{P}} \cdot d_{\tau} \tilde{\by}_{s,h}^{\mathrm{P}} ) \, \ds   \\ &
   + ({\phi^2}{\kappa}^{-1} \tilde{\bu}_{r,h}^{\mathrm{P}}, \tilde{\bu}_{r,h}^{\mathrm{P}} )_{\Omega_{\mathrm{P}}} + \sum_{E \in \mathcal{E}_{\Sigma}} \int_{\Sigma} \frac{\gamma \mu_f}{h_E} 
    \bigl( \bn_{\mathrm{S}} \cdot \tilde{\bu}_{f,h}^{\mathrm{S}} + \bn_{\mathrm{P}} \cdot \tilde{\bu}_{r,h}^{\mathrm{P}} + \bn_{\mathrm{P}} \cdot d_{\tau} \tilde{\by}_{s,h}^{\mathrm{P}} \bigr)^2 \, \ds \nonumber \\ & 
    + \left| \tilde{\bu}_{r,h}^{\mathrm{P}} \right|_{\mathrm{BJS}} ^2  + (\bw_{f,h}^{\mathrm{S}} \cdot \nabla \bu_{f,h}^{\mathrm{S}}, \bu_{f,h}^{\mathrm{S}})= 0.
\end{align*}
 \cred{  By invoking assumptions~\ref{(H1)}–\ref{(H3)}, applying Lemma~\ref{coercivity-continuity} together with the trace, Young’s, and Cauchy--Schwarz inequalities, employing the estimate
 $$ \rho_f\,\phi\,\bigl\lVert \bu_r^{\mathrm{P}} + \bu_s^{\mathrm{P}}\bigr\rVert_{0,\Omega_{\mathrm{P}}}^2 \geq 
  \rho_f\,\phi\,\Bigl(\tfrac12\bigl\lVert \bu_r^{\mathrm{P}}\bigr\rVert_{0,\Omega_{\mathrm{P}}}^2
  -\;\bigl\lVert \bu_s^{\mathrm{P}}\bigr\rVert_{0,\Omega_{\mathrm{P}}}^2\Bigr),$$ and then integrating in time over $(0,t]$ for arbitrary $t \in (0, T]$}, \cblue{we can assert that}
\begin{align*}
0 &  \cred{ \gtrsim 
  \|\tilde \bu_{f,h}^{\mathrm S}(t)\|_{0,\Omega_{\mathrm S}}^2 +  \|\tilde \bu_{s,h}^{\mathrm P}(t)\|_{0,\Omega_{\mathrm P}}^2
  + \|\tilde \bu_{r,h}^{\mathrm P}(t)\|_{0,\Omega_{\mathrm P}}^2
  + \|\tilde p_h^{\mathrm P}(t)\|_{L^2(\Omega)}^2
  + \|\tilde \by_{s,h}^{\mathrm P}(t)\|_{1,\Omega_{\mathrm P}}^2} \\
  & \quad 
 \cred{ + \int_0^t \bigl|\tilde \bu_{f,h}^{\mathrm S} - \partial_t \tilde \by_{s,h}^{\mathrm P}\bigr|_{\mathrm{BJS}}^2 \,\ds 
 + \int_0^t \bigl|\tilde \bu_{r,h}^{\mathrm P}\bigr|_{\mathrm{BJS}}^2
+ \int_0^t \|\tilde \bu_{r,h}^{\mathrm P}(s)\|_{0,\Omega_{\mathrm P}}^2 \,\ds} \\
& \quad  \cred{  + \sum_{E\in\mathcal E_\Sigma}
  \Bigl(\gamma - \tfrac{1+\varsigma}{\widehat\epsilon_f'}\Bigr)\,
  h_E^{-1}
  \bigl\|\,
    \bn_{\mathrm S}\!\cdot\!\tilde\bu_{f,h}^{\mathrm S}
    + \bn_{\mathrm P}\!\cdot\!\tilde\bu_{r,h}^{\mathrm P}
    + \bn_{\mathrm P}\!\cdot\!d_\tau \tilde \by_{s,h}^{\mathrm P}
  \bigr\|_{0,E}^2}   \cred{  + \frac{\bigl(3 - 2(\varsigma + 1)\,\widehat\epsilon_f' C_{\mathrm{tr}} \bigr)}{4}
  \int_0^t \|\tilde \bu_{f,h}^{\mathrm S}(s)\|_{1,\Omega_{\mathrm S}}^2 \,\ds
}.
\end{align*}
\cblue{Since the right-hand side is non-negative}, this readily gives 
\begin{gather*}
\|\tilde{\bu}_{s,h}^{\mathrm{P}}\|_{L^{\infty}(0, T ; \mathbf{L}^2(\Omega_{\mathrm{P}}))}=\|\tilde{\bu}_{f,h}^{\mathrm{S}}\|_{L^2(0, T ; \mathbf{H}^1(\Omega_{\mathrm{S}}))}=\|\tilde{p}_h^{\mathrm{P}}\|_{L^{\infty}(0, T ; L^2(\Omega_{\mathrm{P}}))}=0, \\ 
\|\tilde{\by}_{s,h}^{\mathrm{P}}\|_{L^{\infty}(0, T ; \mathbf{H}^1(\Omega_{\mathrm{P}}))}=\|\tilde{\bu}_{r,h}^{\mathrm{P}}\|_{L^2(0, T ; \mathbf{L}^2(\Omega_{\mathrm{P}}))} =0.
\end{gather*}
Finally, we use the inf-sup conditions \cblue{from Lemmas~\ref{usual_inf}-\ref{inf_sup}} for $\tilde{p}_h^{\mathrm{S}}, \tilde{p}_h^{\mathrm{P}}$ to obtain $\|(\tilde{p}_{h}^{\mathrm{S}}, \tilde{p}_{h}^{\mathrm{P}})\|_{W_h} \lesssim  0$. This implies $\|\tilde{p}_h^{\mathrm{S}}\|_{L^{\infty}(0, T ; L^2(\Omega_{\mathrm{S}}))} =0$. Hence, the solution of \eqref{mixed-primal} is unique, and equivalently, the fixed-point operator $\mathcal{J}_h$ is well-posed.
\end{proof}

The subsequent theorem establishes the well-posedness of Nitsche’s scheme \eqref{semi_weak_for}. 
\begin{theorem}
For each $\boldsymbol{f}_{\mathrm{S}} \in L^{\infty}(\mathrm{I};L^2(\Omega_{\mathrm{S}})), 
\boldsymbol{f}_{\mathrm{P}} \in L^{\infty}(0,T;L^2(\Omega_{\mathrm{P}})),$  
$\theta \in L^{\infty}(0,T;L^2(\Omega_{\mathrm{P}}))$, and compatible initial data $ ( \bu_{r,h}^{\mathrm{P}}(0), \bu_{s,h}^{\mathrm{P}}(0), \bu_{f,h}^{\mathrm{S}}(0), p^{\mathrm{P}}_h(0), \bsigma^{\mathrm{P}}_{h}(0) ) \in \mathbf{V}_{r,h}  \times \mathbf{V}_{s,h} \times \mathbf{V}_{f,h} \times \mathrm{W}_{p,h} \times \bZ_h$ under   assumptions \ref{(H1)}-\ref{(H3)}. Moreover, assume that  $ \frac{4}{C_{\mathrm{I}}^2} \| \boldsymbol{f}_{\mathrm{S}} \|_{0,\Omega_{\mathrm{S}}} < 1 $, and for any $\widehat{\epsilon}_f^{\prime}, \breve{\epsilon}_f^{\prime}$ that satisfy
$$
\frac{3- 2 (\varsigma+1) \widehat{\epsilon}_f^{\prime} C_{\mathrm{tr}}}{4} >0,
$$
where $\varsigma \in\{-1,0,1\}$ provided that $\gamma> (\varsigma + 1) (\widehat{\epsilon}^{\prime}_f)^{-1}$, there exists a unique solution $(\bu^{\mathrm{S}}_{f,h}, p^{\mathrm{S}}_h,  \bu_{r,h}^{\mathrm{P}}, p^{\mathrm{P}}_h,  \by_{s,h}^{\mathrm{P}},  \bu_{s,h}^{\mathrm{P}})   \in  W^{1,\infty}(0, T ; \mathbf{V}_{f,h}) \times L^{\infty}(0, T ; \mathrm{W}_{f,h}) \times W^{1,\infty}(0, T ; \mathbf{V}_{r,h})$ $\times  W^{1,\infty}(0, T ; \mathrm{W}_{p,h}) \times W^{1,\infty}(0, T ; \mathbf{V}_{s,h}) \times W^{1,\infty}(0, T ; \mathrm{W}_{s,h})$ of the weak formulation \eqref{semi_weak_for}.
\end{theorem}
\begin{proof}
According to the relations provided in \eqref{fixed_map}, we aim to establish the well-posedness of the variational formulation \eqref{semi_weak_for}. To this end, we employ Banach's fixed-point theorem by demonstrating that the operator $\mathcal{J}_h$ admits a unique fixed point in the set $\mathbf{K}_h$. The validity of the assumption 
$\frac{4}{C_{\mathrm{I}} ^2} \left\| \boldsymbol{f}_{\mathrm{S}} \right\|_{0,\Omega_{\mathrm{S}}} < 1$
as shown in \cblue{Theorem}~\ref{E+H}, ensures that $\mathcal{J}_h$ is well-defined.

Let $\bw_{f1,h}^{\mathrm{S}}, \bw_{f2,h}^{\mathrm{S}}, \bu_{f1,h}^{\mathrm{S}}, \bu_{f2,h}^{\mathrm{S}} \in \mathbf{K}_h$ be given, such that $\bu_{f1,h}^{\mathrm{S}} = \mathcal{J}_h(\bw_{f1,h}^{\mathrm{S}})$ and $\bu_{f2,h}^{\mathrm{S}} = \mathcal{J}_h(\bw_{f2,h}^{\mathrm{S}})$. Then, by the definition of $\mathcal{J}_h$, the following identities hold for  $\vec{\bx}_{h1}  =(\bu_{f1,h}^{\mathrm{S}}, {p}_{h1}^{\mathrm{S}},\bu_{r1,h}^{\mathrm{P}}, p_{h1}^{\mathrm{P}}, \by_{s1,h}^{\mathrm{P}},\bu_{s1,h}^{\mathrm{P}})$, $\vec{\bx}_{h2}  =(\bu_{f2,h}^{\mathrm{S}}, {p}_{h2}^{\mathrm{S}},\bu_{r2,h}^{\mathrm{P}}, p_{h2}^{\mathrm{P}}, \by_{s2,h}^{\mathrm{P}},\bu_{s2,h}^{\mathrm{P}}) \in \mathbf{X}_h$:  
    \begin{align*}
  \mathcal{A}_{\bw_{f1,h}^{\mathrm{S}}}\left( \vec{\bx}_{h1},\vec{\by}_{h} \right) &:= \bar{E}\bigl(\partial_t \vec{\bx}_{h1},\vec{\by}_{h}\bigr)
  +\bar{H}\bigl(\vec{\bx}_{h1},\vec{\by}_{h}\bigr) + \boldsymbol{c}(\bw_{f1,h}^{\mathrm{S}}, \bu_{f1,h}^{\mathrm{S}}, \bv_{f,h}^{\mathrm{S}})
  =F(\vec{\by}_{h}), \\
   \mathcal{A}_{\bw_{f2,h}^{\mathrm{S}}}\left( \vec{\bx}_{h2},\vec{\by}_{h} \right) &:= \bar{E}\bigl(\partial_t \vec{\bx}_{h2},\vec{\by}_{h}\bigr)
  +\bar{H}\bigl(\vec{\bx}_{h2},\vec{\by}_{h}\bigr) + \boldsymbol{c}(\bw_{f2,h}^{\mathrm{S}}, \bu_{f2,h}^{\mathrm{S}}, \bv_{f,h}^{\mathrm{S}})
  =F(\vec{\by}_{h}),
    \end{align*}
    for all $\vec{\by}_{h}  =(
\bv_{f,h}^{\mathrm{S}}, {q}_{h}^{\mathrm{S}}, \bv_{r,h}^{\mathrm{P}}, 
q_{h}^{\mathrm{P}}, \bw_{s,h}^{\mathrm{P}},\bv_{s,h}^{\mathrm{P}}) \in \mathbf{X}_h$.
    
  By adding and subtracting appropriate terms, we can derive the following:
  			\begin{align}\label{40}
            \mathcal{A}_{\bw_{f1,h}^{\mathrm{S}}}\left( \vec{\bx}_{h1}-\vec{\bx}_{h2},\vec{\by}_{h} \right) = -\boldsymbol{c}\left(\bw_{f1,h}^{\mathrm{S}}-\bw_{f2,h}^{\mathrm{S}} ; \bu_{f2,h}^{\mathrm{S}}, \bv_{f,h}^{\mathrm{S}} \right) \quad \forall \vec{\by}_{h} \in \mathbf{X}_h.
  			\end{align}
  			Given that $\boldsymbol{w}_{h1} \in \mathbf{K}_h$, and using \eqref{40} along with the coercivity of the bilinear forms established in Lemma~\ref{coercivity-continuity}, we can deduce: 
  			\begin{align*}
  			C_{\alpha} \left\|\bu_{f1,h}^{\mathrm{S}}-\bu_{f2,h}^{\mathrm{S}}\right\|_{1,\Omega_{\mathrm{S}}} & \leq \sup _{ \boldsymbol{0} \neq  \vec{\by}_{h}  \in \mathbf{X}_h} \frac{\mathcal{A}_{\bw_{f1,h}^{\mathrm{S}}}\left( \vec{\bx}_{h1}-\vec{\bx}_{h2},\vec{\by}_{h} \right)}{\|\vec{\by}_{h} \|} \\
            & =\sup _{\boldsymbol{0} \neq \vec{\by}_{h}  \in \mathbf{X}_h} \frac{-\boldsymbol{c}\left(\bw_{f1,h}^{\mathrm{S}}-\bw_{f2,h}^{\mathrm{S}} ; \bu_{f2,h}^{\mathrm{S}}, \bv_{f,h}^{\mathrm{S}} \right)}{\|\vec{\by}_{h} \|} \\
  			& \leq S_f^2 K_f^3 \left\|\bw_{f1,h}^{\mathrm{S}}-\bw_{f2,h}^{\mathrm{S}} \right\|_{1,\Omega_{\mathrm{S}}}\left\|\bu_{f2,h}^{\mathrm{S}}\right\|_{1,\Omega_{\mathrm{S}}},
  			\end{align*}
  			which together with the fact that $\bu_{f2,h}^{\mathrm{S}} \in \mathbf{K}_h$, allows us to assert the bounds 
            \begin{align*}
                C_{\alpha} \left\|\bu_{f1,h}^{\mathrm{S}}-\bu_{f2,h}^{\mathrm{S}}\right\|_{1,\Omega_{\mathrm{S}}} & \leq \frac{1}{C_{\mathrm{I}} } \|\boldsymbol{f}\|_{0,\Omega_{\mathrm{S}}} \left\|\bw_{f1,h}^{\mathrm{S}}-\bw_{f2,h}^{\mathrm{S}} \right\|_{1,\Omega_{\mathrm{S}}}, \\
                \left\|\bu_{f1,h}^{\mathrm{S}}-\bu_{f2,h}^{\mathrm{S}}\right\|_{1,\Omega_{\mathrm{S}}} & \leq \frac{C_{\mathrm{I}}}{4 C_{\alpha} } \left\|\bw_{f1,h}^{\mathrm{S}}-\bw_{f2,h}^{\mathrm{S}} \right\|_{1,\Omega_{\mathrm{S}}}.
            \end{align*}
  			In turn, these steps establish that $\mathcal{J}_h$ is a contraction mapping.
	Hence, $\bu_{f,h}^{\mathrm{S}} \in \mathbf{K}_h$ is the unique fixed point of $\mathcal{J}_h$ and  $\vec{\bx}_{h1} \in W^{1,\infty}(0, T ; \mathbf{V}_{f,h}) \times L^{\infty}(0, T ; \mathrm{W}_{f,h}) \times W^{1,\infty}(0, T ; \mathbf{V}_{r,h}) \times  W^{1,\infty}(0, T ; \mathrm{W}_{p,h}) \times W^{1,\infty}(0, T ; \mathbf{V}_{s,h}) \times W^{1,\infty}(0, T ;\mathrm{W}_{s,h})$ is the unique solution of \eqref{semi_weak_for}. 
\end{proof}

\section{Fully discrete formulation}\label{section5}
For the time discretization we employ the backward Euler method with constant time-step $\tau$, $T = N \tau$, and let $t_n = n\tau$, $1 \leq n \leq N$. Let $d_{\tau} u^n := \tau^{-1}(u^n - u^{n-1})$ be the
first order (backward) discrete time derivative, where $u^n \approx u(t_n)$. 
The fully discrete problem reads:  given $ \bu_{f,h}^0=\bu_{f,h}^{\mathrm{S}}(0)$,  $\bu_{r,h}^0=\bu_{r,h}^{\mathrm{P}}(0)$,  $\by_{s,h}^0=\by_{s,h}^{\mathrm{P}}(0)$,   $p^{\mathrm{P},0}_h={p}_{h}^{\mathrm{P}}(0)$, and  $\bu_{s,h}^0=\bu_{s,h}^{\mathrm{P}}(0)$, find  $\vec{\bx}_h^n \in \vec{\bX}_h$,  such that for $1 \leq n \leq N $, there holds 
    \begin{equation}\label{fully}
        \bar{E}\bigl(\frac{1}{\tau}\vec{\bx}_h^n, \vec{\by}_h \bigr) + \bar{H}(\vec{\bx}_h^n, \vec{\by}_h) + (\bu_{f,h}^{\mathrm{S},n-1} \cdot \bnabla \bu_{f,h}^{\mathrm{S},n}, \bv_{f,h}^{\mathrm{S}}) = F^n(\vec{\by}_h) +  \bar{E}(\frac{1}{\tau}\vec{\bx}_h^{n-1}, \vec{\by}_h) ,
    \end{equation}
for all $\vec{\by}_h\in \vec{\bX}_h$ and where $F^n$ stands for the evaluation of $F$ at time $t_n$.  
\begin{theorem}
Under Assumptions \ref{(H1)}-\ref{(H3)}, and provided that
$\| \bu_{f,h}^{\mathrm{S},n} \|_{1,\Omega_{\mathrm{S}}} < \frac{\mu_f}{2 S_f^2 K_f^3}$ and  $\gamma > \mu_f (1+\varsigma) C_{\mathrm{tr}}^2,$
for $\varsigma \in \{-1,0,1\}$ and $1 \le n \le N$, the fully discrete scheme~\eqref{fully} admits a unique solution.
\end{theorem}
\begin{proof}
\cred{The aim is to show that the following term 
\[
\bar{E}\bigl(\tfrac{1}{\tau}\vec{\bx}_h^n, \vec{\by}_h^n \bigr) 
+ \bar{H}(\vec{\bx}_h^n, \vec{\by}_h^n) 
+ (\bu_{f,h}^{\mathrm{S},n-1} \cdot \bnabla \bu_{f,h}^{\mathrm{S},n}, \bu_{f,h}^{\mathrm{S},n})
\]
is coercive, where 
$\vec{\bx}_h^n =
(
\bu_{f,h}^{\mathrm{S},n},
p_h^{\mathrm{S},n},
\bu_{r,h}^{\mathrm{P},n},
p_h^{\mathrm{P},n},
\by_{s,h}^{\mathrm{P},n},
\bu_{s,h}^{\mathrm{P},n}
)$
and
$\vec{\by}_h^n =
(
\bu_{f,h}^{\mathrm{S},n},
p_h^{\mathrm{S},n},
\bu_{r,h}^{\mathrm{P},n},
p_h^{\mathrm{P},n},
\tfrac{1}{\tau}\by_{s,h}^{\mathrm{P},n},
\bu_{s,h}^{\mathrm{P},n}
)$ denote trial and test functions, respectively. We proceed to choose   $\vec{\by}_h^n = T\vec{\bx}_h^n$, where the operator $T$ scales the $\by_{s,h}^{\mathrm{P},n}$ component by $\frac{1}{\tau}$. Since $T$ is bounded and satisfies the below estimate \eqref{b_estimate}, implies that the associated discrete operator is non-singular, and hence the solution is unique. Let us then consider} 
\begin{align*}
 &\cred{\bar{E}\bigl(\frac{1}{\tau}\vec{\bx}_h^n, \vec{\by}_h \bigr) + \bar{H}(\vec{\bx}_h^n, \vec{\by}_h) + (\bu_{f,h}^{\mathrm{S},n-1} \cdot \bnabla \bu_{f,h}^{\mathrm{S},n}, \bu_{f,h}^{\mathrm{S}})   = a_f^{\mathrm{S}}(\bu^{\mathrm{S},n}_{f,h}, \bu^{\mathrm{S},n}_{f,h})+ \frac{1}{\tau } a_{f}^{\mathrm{P}}(\bu_{r,h}^{\mathrm{P},n}, \by_{s,h}^{\mathrm{P},n})+ \frac{1}{\tau } a_s^{\mathrm{P}}(\by_{s,h}^{\mathrm{P},n}, \by_{s,h}^{\mathrm{P},n}) } \\& \quad \cred{+ a_{f}^{\mathrm{P}}(\bu_{r,h}^{\mathrm{P},n},\bu_{r,h}^{\mathrm{P},n}) + \frac{1}{\tau } a_{f}^{\mathrm{P}}(\by_{s,h}^{\mathrm{P},n}, \bu_{r,h}^{\mathrm{P},n})  + \frac{1}{\tau^2} a_{f}^{\mathrm{P}}(\by_{s,h}^{\mathrm{P},n}, \by_{s,h}^{\mathrm{P},n}) + \frac{1}{\tau^2} a_{\mathrm{BJS}}(\bu^{\mathrm{S},n}_{f,h},  \by_{s,h}^{\mathrm{P},n} ; \bu^{\mathrm{S},n}_{f,h}, \by_{s,h}^{\mathrm{P},n})   - \frac{1}{\tau } m_{\theta}(\bu_{r,h}^{\mathrm{P},n}, \by_{s,h}^{\mathrm{P},n})} \\& \quad  \cred{- \frac{1}{\tau^2} m_{\theta}( \by_{s,h}^{\mathrm{P},n}, \by_{s,h}^{\mathrm{P},n}) - m_{\theta}(\bu_{r,h}^{\mathrm{P},n}, \bu_{r,h}^{\mathrm{P},n}) - \frac{1}{\tau } m_{\theta}( \by_{s,h}^{\mathrm{P},n}, \bu_{r,h}^{\mathrm{P},n})  + m_{\phi^2/\kappa}(\bu_{r,h}^{\mathrm{P},n}, \bu_{r,h}^{\mathrm{P},n}) +\frac{1}{\tau^2} m_{\rho_f \phi}(\bu_{r,h}^{\mathrm{P},n}, \by_{s,h}^{\mathrm{P},n}) } \\&\quad \cred{+ \frac{1}{\tau^2} m_{\rho_p}(\bu_{s,h}^{\mathrm{P},n}, \by_{s,h}^{\mathrm{P},n}) +\frac{1}{\tau } m_{\rho_f \phi }(\bu_{r,h}^{\mathrm{P},n}, \bu_{r,h}^{\mathrm{P},n})  +\frac{1}{\tau } m_{\rho_f }(\bu_{f,h}^{\mathrm{S},n}, \bu_{f,h}^{\mathrm{S},n})   +\frac{1}{\tau } m_{\rho_f \phi}(\bu_{s,h}^{\mathrm{P},n}, \bu_{r,h}^{\mathrm{P},n})  + b_{\mathrm{BJS}}(\bu^{\mathrm{P},n}_{r,h};  \bu_{r,h}^{\mathrm{P},n})} \\&\quad \cred{+ \frac{1}{\tau }
((1-\phi)^2 K^{-1} p^{\mathrm{P},n}_h, p^{\mathrm{P}}_h)_{\Omega_{\mathrm{P}}}  + b_{\Gamma}( \bu_{f,h}^{\mathrm{S},n}, \bu_{r,h}^{\mathrm{P},n}, \frac{1}{\tau} \by_{s,h}^{\mathrm{P},n} ; (1+\varsigma )\bu^{\mathrm{S}}_{f})   + (\bu_{f,h}^{\mathrm{S},n-1} \cdot \bnabla \bu_{f,h}^{\mathrm{S},n}, \bu_{f,h}^{\mathrm{S}}) }\\& \quad \cred{+ c_{\Gamma}\bigl(\bu_{f,h}^{\mathrm{S},n},\bu_{r,h}^{\mathrm{P},n},\frac{1}{\tau} \by_{s,h}^{\mathrm{P},n}
    ;\,\bu_{f,h}^{\mathrm{S},n},\bu_{r,h}^{\mathrm{P},n},\frac{1}{\tau} \by_{s,h}^{\mathrm{P},n}\bigr).}
    \end{align*}
\cred{We bound some of the terms above using the inequality \eqref{ineq:aux}. This step gives }
\begin{align*}
& \cred{\frac{1}{\tau } a_{f}^{\mathrm{P}}(\bu_{r,h}^{\mathrm{P},n}, \by_{s,h}^{\mathrm{P},n})+ \frac{1}{\tau } a_{f}^{\mathrm{P}}(\by_{s,h}^{\mathrm{P},n}, \bu_{r,h}^{\mathrm{P},n}) - \frac{1}{\tau } m_{\theta}(\bu_{r,h}^{\mathrm{P},n}, \by_{s,h}^{\mathrm{P},n}) - \frac{1}{\tau } m_{\theta}( \by_{s,h}^{\mathrm{P},n}, \bu_{r,h}^{\mathrm{P},n}) +\frac{1}{\tau^2} m_{\rho_f \phi}(\bu_{r,h}^{\mathrm{P},n}, \by_{s,h}^{\mathrm{P},n})}\\& \cred{+\frac{1}{\tau } m_{\rho_f \phi}(\bu_{s,h}^{\mathrm{P},n}, \bu_{r,h}^{\mathrm{P},n}) + \frac{1}{\tau^2} m_{\rho_p}(\bu_{s,h}^{\mathrm{P},n}, \by_{s,h}^{\mathrm{P},n}) \geq - \frac{1}{\tau } a_{f}^{\mathrm{P}}\left(\bu_{r,h}^{\mathrm{P},n}, \bu_{r,h}^{\mathrm{P},n}\right) - \frac{1}{\tau } a_{f}^{\mathrm{P}}(\by_{s,h}^{\mathrm{P},n}, \by_{s,h}^{\mathrm{P},n}) + \frac{1}{\tau } m_{\theta}(\bu_{r,h}^{\mathrm{P},n}, \bu_{r,h}^{\mathrm{P},n})} \\ &  \cred{+ \frac{1}{\tau } m_{\theta}(\by_{s,h}^{\mathrm{P},n}, \by_{s,h}^{\mathrm{P},n}) - \frac{1}{\tau} m_{\rho_f \phi}(\bu_{r,h}^{\mathrm{P},n}, \bu_{r,h}^{\mathrm{P},n}) - \frac{1}{\tau} m_{\rho_f \phi}(\bu_{s,h}^{\mathrm{P},n}, \bu_{s,h}^{\mathrm{P},n}) + \frac{1}{\tau } m_{\rho_p}(\bu_{s,h}^{\mathrm{P},n}, \bu_{s,h}^{\mathrm{P},n}). }
\end{align*} 
\cred{By combining both estimates above and using Lemma \ref{coercivity-continuity}, we arrive at} 
\begin{align*}
&  \cred{\mathrm{E}(\frac{1}{\tau}\vec{\bx}_h^n, \vec{\by}_h) + \mathrm{H}(\vec{\bx}_h^n, \vec{\by}_h) + (\bu_{f,h}^{\mathrm{S},n-1} \cdot \bnabla \bu_{f,h}^{\mathrm{S},n}, \bu_{f,h}^{\mathrm{S}}) \geq  a_f^{\mathrm{S}}(\bu^{\mathrm{S},n}_{f,h}, \bu^{\mathrm{S},n}_{f,h})+ \frac{1}{\tau } a_s^{\mathrm{P}}(\by_{s,h}^{\mathrm{P},n}, \by_{s,h}^{\mathrm{P},n})  + b_{\mathrm{BJS}}(\bu^{\mathrm{P},n}_{r,h};  \bu_{r,h}^{\mathrm{P},n})} \\&\quad \cred{+ a_{\mathrm{BJS}}(\bu^{\mathrm{S},n}_{f,h},\frac{1}{\tau}  \by_{s,h}^{\mathrm{P},n} ; \bu^{\mathrm{S},n}_{f,h}, \frac{1}{\tau} \by_{s,h}^{\mathrm{P},n}) + m_{\phi^2/\kappa}(\bu_{r,h}^{\mathrm{P},n}, \bu_{r,h}^{\mathrm{P},n}) + \frac{1}{\tau }
((1-\phi)^2 K^{-1} p^{\mathrm{P},n}_h, p^{\mathrm{P}}_h)_{\Omega_{\mathrm{P}}} + (\bu_{f,h}^{\mathrm{S},n-1} \cdot \bnabla \bu_{f,h}^{\mathrm{S},n}, \bu_{f,h}^{\mathrm{S}})} \\ & \quad \cred{+ b_{\Gamma}( \bu_{f,h}^{\mathrm{S},n}, \bu_{r,h}^{\mathrm{P},n}, \frac{1}{\tau} \by_{s,h}^{\mathrm{P},n} ; (1+\varsigma )\bu^{\mathrm{S}}_{f}) + c_{\Gamma}\bigl(\bu_{f,h}^{\mathrm{S},n},\bu_{r,h}^{\mathrm{P},n},\frac{1}{\tau} \by_{s,h}^{\mathrm{P},n} ;\,\bu_{f,h}^{\mathrm{S},n},\bu_{r,h}^{\mathrm{P},n},\frac{1}{\tau} \by_{s,h}^{\mathrm{P},n}\bigr)  \geq 2 \mu_f  \| \varepsilon(\bu_{f,h}^{\mathrm{S},n})\|_{0,\Omega_{\mathrm{S}}}^2} \\ & \quad  \cred{+ \mu_f \alpha_{\mathrm{BJS}} Z_{\mathrm{max}}^{-1/2}(
|\bu_{f,h}^{\mathrm{S},n} - \bw_{s,h}^{\mathrm{P},h}|_{\mathrm{BJS}}^2 + |\bu_{r,h}^{\mathrm{P},n}|_{\mathrm{BJS}}^2
) + \frac{\phi^2 C_P^2}{\kappa}  \|\bu_{r,h}^{\mathrm{P},n}\|_{1,\Omega_{\mathrm{P}}}^2 + 2 \mu_p \| \varepsilon (\by_{s,h}^{\mathrm{P},n}) \|_{0,\Omega_{\mathrm{P}}} + \lambda_p \| \nabla \cdot \by_{s,h}^{\mathrm{P},n} \|_{0, \Omega_{\mathrm{P}}} } \\& \quad
\cred{- \delta_1 C_{\mathrm{tr}}^2 \mu_f \| \varepsilon(\bu_{f,h}^{\mathrm{S},n})\|_{0,\Omega_{\mathrm{S}}}^2 - \frac{\mu_f}{\delta_1} (1+\varsigma )\sum_{E \in \mathcal{E}_{\Sigma}} \frac{1}{h_E}  ( \|\boldsymbol{u}_{f,h}^{\mathrm{S},n} \cdot \bn_{\mathrm{S}} + \boldsymbol{u}_{r,h}^{\mathrm{P},n} \cdot \bn_{\mathrm{P}}   + \frac{1}{\tau} \boldsymbol{y}_{s,h}^{\mathrm{P},n} \cdot \bn_{\mathrm{P}}\|_{0, E}^2 ) + \frac{(1-\phi)^2K^{-1}}{\tau} \|p_h^{\mathrm{P},n}\|_{0,\Omega_{\mathrm{P}}} } \\& \quad \cred{+ \sum_{E \in \mathcal{E}_{\Sigma}} \frac{\gamma }{h_E} \| \bu_{f,h}^{\mathrm{S},n} \cdot \bn_{\mathrm{S}}+ \bu_{r,h}^{\mathrm{P},n} \cdot \bn_{\mathrm{P}} + \frac{1}{\tau} \by_{s,h}^{\mathrm{P},n} \cdot \bn_{\mathrm{P}} \|_{0,E}^2 - S_f^2 K_f^3 \|\bu_{f,h}^{\mathrm{S},{n-1}}\|_{1, \Omega_{\mathrm{S}}} \, \| \varepsilon (\bu_{f,h}^{\mathrm{S},n}) \|_{0, \Omega_{\mathrm{S}}}^2  .}
\end{align*}
\cred{Then, choosing $\delta_1 = \frac{1}{C_{\mathrm{tr}}^2}$, we can assert that}
\begin{align}\label{b_estimate}
    & \cred{ \mathrm{E}(\frac{1}{\tau}\vec{\bx}_h^n, \vec{\by}_h) + \mathrm{H}(\vec{\bx}_h^n, \vec{\by}_h) + (\bu_{f,h}^{\mathrm{S},n-1} \cdot \bnabla \bu_{f,h}^{\mathrm{S},n}, \bu_{f,h}^{\mathrm{S}}) \geq \frac{\mu_f}{2}  \| \varepsilon(\bu_{f,h}^{\mathrm{S},n})\|_{0,\Omega_{\mathrm{S}}}^2  + \mu_f \alpha_{\mathrm{BJS}} Z_{\mathrm{max}}^{-1/2}(
|\bu_{f,h}^{\mathrm{S},n} - \bw_{s,h}^{\mathrm{P},h}|_{\mathrm{BJS}}^2 + |\bu_{r,h}^{\mathrm{P},n}|_{\mathrm{BJS}}^2
)} \nonumber \\ & \quad  \cred{+ \frac{\phi^2 C_P^2}{\kappa}  \|\bu_{r,h}^{\mathrm{P},n}\|_{1,\Omega_{\mathrm{P}}}^2  + (\gamma -\mu_f C_{\mathrm{tr}}^2 (1+\varsigma )) \sum_{E \in \mathcal{E}_{\Sigma}} \frac{1}{h_E}  ( \|\boldsymbol{u}_{f,h}^{\mathrm{S},n} \cdot \bn_{\mathrm{S}} + \boldsymbol{u}_{r,h}^{\mathrm{P},n} \cdot \bn_{\mathrm{P}}   + \frac{1}{\tau} \boldsymbol{y}_{s,h}^{\mathrm{P},n} \cdot \bn_{\mathrm{P}}\|_{0, E}^2 )  + 2 \mu_p \| \varepsilon (\by_{s,h}^{\mathrm{P},n}) \|_{0,\Omega_{\mathrm{P}}}} \nonumber  \\& \quad \cred{+ \lambda_p \| \nabla \cdot \by_{s,h}^{\mathrm{P},n} \|_{0, \Omega_{\mathrm{P}}} + \frac{(1-\phi)^2K^{-1}}{\tau} \|p_h^{\mathrm{P},n}\|_{0,\Omega_{\mathrm{P}}} .}
\end{align} 
\cred{It is clear that all terms on the right-hand side are positive. Hence, the bilinear form on the left-hand side is positive-definite, and consequently the matrix obtained from the system \eqref{fully} is non-singular. The uniqueness follows from the fact that a linear system with a non-singular matrix admits a unique solution.}
\end{proof}
\subsection{Stability analysis of the fully discrete scheme} 
In the following lemma, we discuss the stability analysis of fully discrete problem.
We will make use of the discrete space–time
norms, where $I=(0,T)$
$$
\| \phi\|_{l^2(I;X)}^2 := \tau \sum_{n=1}^N \| \phi^n \|_X^2,~ \| \phi\|_{l^{\infty}(I;X)}^2 := \max_{0 \leq n \leq N}\| \phi^n \|_X^2,~ \left| \bphi\right|_{l^2(I;{\mathrm{BJS}})} = \tau \sum_{n=1}^N \left| \bphi\right|^2_{\mathrm{BJS}}.
$$ 

\begin{lemma}\label{stability_fully}
Under Assumptions~\ref{(H1)}-\ref{(H3)} and 
$\| \bu_{f,h}^{\mathrm{S},n} \|_{1,\Omega_{\mathrm{S}}} < \frac{\mu_f}{2 S_f^2 K_f^3}, \ 1 \leq n \leq N,$  for any $\widehat{\epsilon}_f^{\prime}, \breve{\epsilon}_f^{\prime}$ such that  
$$\frac14\bigl( 3 - 2(1 + \varsigma)\widehat{\epsilon_f}^{\prime} C_{\mathrm{tr}} - \breve{\epsilon}_f^{\prime} \bigr)>0,$$
where $\varsigma \in\{-1,0,1\}$ provided that $\gamma> (\varsigma + 1) (\widehat{\epsilon}^{\prime}_f)^{-1}$, there exist constants $0<c<1$ and $C>1$, uniformly independent of the mesh  size $h$, such that
\begin{align*}
    & \|\bu_{f,h}^{\mathrm{S},N}\|_{0,\Omega_{\mathrm{S}}}^2 + \|\bu_{r,h}^{\mathrm{P},N}\|_{0,\Omega_{\mathrm{P}}}^2 + \|\by_{s,h}^{\mathrm{P},N} \|_{1,\Omega_{\mathrm{P}}}^2 + \|p^{\mathrm{P},N}_h\|_{0,\Omega_{\mathrm{P}}}^2  +  \|\bu_{s,h}^{\mathrm{P},N}\|_{0,\Omega_{\mathrm{P}}}^2 + \left| \bu^{\mathrm{P},n}_{r,h} \right|_{\mathrm{BJS}}^2   \\&\quad   + \tau \sum_{n=1}^N ( \left| \bu^{\mathrm{S},n}_{f,h} - d_{\tau} \by_{s,h}^{\mathrm{P},n} \right|_{\mathrm{BJS}}^2   + \|\bu_{r,h}^{\mathrm{P},n}\|_{0,\Omega_{\mathrm{P}}}^2 + \|p_h^{\mathrm{S},n}\|_{0,\Omega_{\mathrm{S}}}^2 + \|p_h^{\mathrm{P},n}\|_{0,\Omega_{\mathrm{P}}}^2 )  \\&\quad  + \tau^2 \sum_{n=1}^N(  \|d_{\tau} \bu_{f,h}^{\mathrm{S},n}\|_{0,\Omega_{\mathrm{S}}}^2 
     + \|d_{\tau} \bu_{r,h}^{\mathrm{P},n}\|_{0,\Omega_{\mathrm{P}}}^2   + \|d_{\tau} \by_{s,h}^{\mathrm{P},n}\|_{1,\Omega_{\mathrm{P}}}^2 
     + \|d_{\tau} p^{\mathrm{P},n}_h\|_{0,\Omega_{\mathrm{P}}}^2
    \|d_{\tau} \bu_{s,h}^{\mathrm{P},n}\|_{0,\Omega_{\mathrm{P}}}^2 )  \\ & \quad + c \tau \sum_{n=1}^N \bigl(  \|\bu^{\mathrm{S},n}_{f,h}\|_{1,\Omega_{\mathrm{S}}}^2  + \sum_{E \in \mathcal{E}_{\Sigma}}  h_E^{-1} \|\bn_{\mathrm{S}} \cdot \bu_{f,h}^{\mathrm{S},n}  + \bn_{\mathrm{P}} \cdot \bu_{r,h}^{\mathrm{P},n}  + \bn_{\mathrm{P}} \cdot d_{\tau} \by_{s,h}^{\mathrm{P},n}\|_{0,E}^2 \bigr)\\
    & \lesssim \exp(T) \bigg[ \|\bu_{f,h}^{0}\|_{0,\Omega_{\mathrm{S}}}^2  + \|\bu_{r,h}^{0}\|_{0,\Omega_{\mathrm{P}}}^2  + \|\by_{s,h}^{0}\|_{1,\Omega_{\mathrm{P}}}^2 + \|p^{\mathrm{P},0}_h\|_{0,\Omega_{\mathrm{P}}}^2
      + \|\bu_{s,h}^{0}\|_{0,\Omega_{\mathrm{P}}}^2    +\left\|\ff_{\mathrm{P}}(0)\right\|_{0,\Omega_{\mathrm{P}}} \\& \quad  +  
     C \tau  \sum_{n=1}^N  \|\boldsymbol{f}_{\mathrm{S}}(t_n)\|_{0,\Omega_{\mathrm{S}}}^2   + \epsilon_1^{-1} \tau  \sum_{n=1}^N  \bigl( \|\ff_{\mathrm{P}}(t_n)\|_{0,\Omega_{\mathrm{P}}}^2 + \|\theta(t_n)\|_{0,\Omega_{\mathrm{P}}}^2  + \|r_{\mathrm{S}}(t_n)\|_{0,\Omega_{\mathrm{S}}}^2 \bigr)   +\tau \sum_{n=1}^{N}\left\|d_\tau \ff_{\mathrm{P}}^n\right\|_{0,\Omega_{\mathrm{P}}}^2 \bigg].
\end{align*} More precisely, we have
$$
c<\min \left\{ \frac 14\bigl( 3 - 2(1 + \varsigma)\widehat{\epsilon_f}^{\prime} C_{\mathrm{tr}} - \breve{\epsilon}_f^{\prime} \bigr) , \gamma - (\varsigma + 1) (\widehat{\epsilon}^{\prime}_f)^{-1} \right\}, \qquad 
 C>( \breve{\epsilon}_f^{\prime})^{-1} .
$$
\end{lemma}
 \begin{proof}
         We choose $ (\bv_{f,h}^{\mathrm{S}}, q^{\mathrm{S}}_h, \bv_{r,h}^{\mathrm{P}}, q^{\mathrm{P}}_h, \bw_{s,h}^{\mathrm{P}}) = (\bu^{\mathrm{S},n}_{f,h}, p^{\mathrm{S},n}_h, \bu_{r,h}^{\mathrm{P},n}, p^{\mathrm{P},n}_h, d_{\tau} \by_{s,h}^{\mathrm{P},n})$ in \eqref{fully}, apply \eqref{ineq:aux}, and use the following identity 
    $\int_{\Omega_{\mathrm{P}}} {\varsigma}^n d_{\tau} \varsigma^n = \frac{1}{2} d_{\tau} \| \varsigma^n \|^2_{0,\Omega_{\mathrm{P}}} + \frac{1}{2} \tau \| d_{\tau} \varsigma^n \|^2_{0,\Omega_{\mathrm{P}}}.$
With this we readily  obtain the energy inequality 
\begin{align*}
&  \frac{d_{\tau}}{2}  \bigl( \rho_f \| \bu_{f,h}^{\mathrm{S},n}\|_{0,\Omega_{\mathrm{S}}}^2  + \rho_s (1-\phi) \| \bu_{s,h}^{\mathrm{P},n}\|_{0,\Omega_{\mathrm{P}}}^2 + {(1-\phi)^2}{K}^{-1} \|p^{\mathrm{P},n}_h\|^2_{0,\Omega_{\mathrm{P}}} +  2 \mu_p \|\beps(\by_{s,h}^{\mathrm{P},n})\|^2_{0,\Omega_{\mathrm{P}}} + \lambda_p \|\nabla \cdot \by_{s,h}^{\mathrm{P},n}\|^2_{0,\Omega_{\mathrm{P}}} \\ &   \quad  +\rho_f \phi \|(\bu_{s,h}^{\mathrm{P},n}+\bu_{r,h}^{\mathrm{P},n})\|^2_{0,\Omega_{\mathrm{P}}} \bigr)  + \frac{\tau}{2} \bigl( \rho_f \|d_{\tau}  \bu_{f,h}^{\mathrm{S},n}\|_{0,\Omega_{\mathrm{S}}}^2  + \rho_s (1-\phi)   \| d_{\tau}\bu_{s,h}^{\mathrm{P},n}\|_{0,\Omega_{\mathrm{P}}}^2  +(1-\phi)^2{K}^{-1} \|d_{\tau} p^{\mathrm{P},n}_h\|_{0,\Omega_{\mathrm{P}}}^2 
 \\
& \quad  +   2 \mu_p\| d_{\tau} \beps( \by_{s,h}^{\mathrm{P},n})\|_{0,\Omega_{\mathrm{P}}}^2    + \lambda_p \|d_{\tau} \nabla \cdot \by_{s,h}^{\mathrm{P},n}\|_{0,\Omega_{\mathrm{P}}}^2   + \rho_f \phi \|d_{\tau} (\bu_{s,h}^{\mathrm{P},n}+\bu_{r,h}^{\mathrm{P},n} )\|_{0,\Omega_{\mathrm{P}}}^2 \bigr)  \\ &\quad   + \left| \bu^{\mathrm{S},n}_{f,h} - d_{\tau} \by_{s,h}^{\mathrm{P},n}\right|_{\mathrm{BJS}}^2    + \left| \bu^{\mathrm{P},n}_{r,h} \right|_{\mathrm{BJS}}^2  + ({\phi^2}{\kappa}^{-1} \bu_{r,h}^{\mathrm{P},n},\bu_{r,h}^{\mathrm{P},n} )_{\Omega_{\mathrm{P}}}    +   (2 \mu_f  \beps(\bu^{\mathrm{S},n}_{f,h}), \beps(\bu^{\mathrm{S},n}_{f,h}) )_{\Omega_{\mathrm{S}}}  \\& \quad  - \sum_{E \in \mathcal{E}_{\Sigma}} \int_{\Sigma} (1+ \varsigma) (2 \mu_f \beps (\bu_{f,h}^{\mathrm{S},n}) \bn_{\mathrm{S}} \cdot \bn_{\mathrm{S}} ) ( \bn_{\mathrm{S}} \cdot \bu_{f,h}^{\mathrm{S},n}  + \bn_{\mathrm{P}} \cdot \bu_{r,h}^{\mathrm{P},n}  + \bn_{\mathrm{P}} \cdot d_{\tau} \by_{s,h}^{\mathrm{P},n} ) \ds \\ & \quad 
 +  \sum_{E \in \mathcal{E}_{\Sigma}} \int_{\Sigma} \frac{\gamma \mu_f}{h_E} ( \bn_{\mathrm{S}} \cdot \bu_{f,h}^{\mathrm{S},n} + \bn_{\mathrm{P}} \cdot \bu_{r,h}^{\mathrm{P},n} + \bn_{\mathrm{P}} \cdot d_{\tau} \by_{s,h}^{\mathrm{P},n} )^2  \ds   + (\bu_{f,h}^{\mathrm{S},n-1} \cdot \bnabla \bu_{f,h}^{\mathrm{S},n}, \bu_{f,h}^{\mathrm{S}})_{\Omega_{\mathrm{S}}} \\
 & \leq  (\ff_{\mathrm{S}}(t_n), \bu^{\mathrm{S},n}_{f,h})_{\Omega_{\mathrm{S}}}     + ( \rho_p \ff_{\mathrm{P}}(t_n), d_{\tau} \by_{s,h}^{\mathrm{P},n} )_{\Omega_{\mathrm{P}}}    + ( r_{\mathrm{S}}(t_n), p^{\mathrm{S},n}_h)_{\Omega_{\mathrm{S}}}   + (\rho_f \phi  \ff_{\mathrm{P}}(t_n), \bu_{r,h}^{\mathrm{P},n} )_{\Omega_{\mathrm{P}}}    + (\rho_f^{-1} \theta(t_n), p^{\mathrm{P},n}_h )_{\Omega_{\mathrm{P}}}. 
\end{align*} 
Using the trace, Cauchy--Schwarz, Young’s, Korn's and Sobolev inequalities, 
together with Lemma~\ref{coercivity-continuity} and the estimate
 $ \rho_f \,\phi\,\| \bu_r^{\mathrm{P}} + \bu_s^{\mathrm{P}}\|_{0,\Omega_{\mathrm{P}}}^2
\geq   \rho_f \,\phi\,\Bigl(\tfrac12\|\bu_r^{\mathrm{P}}\|_{0,\Omega_{\mathrm{P}}}^2
  - \|\bu_s^{\mathrm{P}}\|_{0,\Omega_{\mathrm{P}}}^2\Bigr),
  $ we then sum over \(n=1,\ldots,N\) and multiply by \(\tau\) to obtain
\begin{align}\label{stability_1}
    & \|\bu_{f,h}^{\mathrm{S},N}\|_{0,\Omega_{\mathrm{P}}}^2 + \|\bu_{r,h}^{\mathrm{P},N}\|_{0,\Omega_{\mathrm{P}}}^2 + \|\by_{s,h}^{\mathrm{P},N} \|_{1,\Omega_{\mathrm{P}}}^2  + \|p^{\mathrm{P},N}_h\|_{0,\Omega_{\mathrm{P}}}^2  + \|\bu_{s,h}^{\mathrm{P},N}\|_{0,\Omega_{\mathrm{P}}}^2 \nonumber \\ 
    & \quad + \tau \sum_{n=1}^N \bigl( \left| \bu^{\mathrm{P},n}_{r,h}  \right|_{\mathrm{BJS}}^2       + \left| \bu^{\mathrm{S},n}_{f,h} - d_{\tau} \by_{s,h}^{\mathrm{P},n} \right|_{\mathrm{BJS}}^2     + \|\bu_{r,h}^{\mathrm{P},n}\|_{0,\Omega_{\mathrm{P}}}^2 \bigr)+ \tau \sum_{n=1}^N \frac{\bigl( 3 - 2(1 + \varsigma)\widehat{\epsilon_f}^{\prime} C_{\mathrm{tr}} - \breve{\epsilon}_f^{\prime} \bigr)}{4}  \|\bu^{\mathrm{S},n}_{f,h}\|_{1,\Omega_{\mathrm{S}}}^2   \nonumber  \\ 
    & \quad + \tau^2 \sum_{n=1}^N \bigl(\|d_{\tau} \bu_{f,h}^{\mathrm{S},n}\|_{0,\Omega_{\mathrm{S}}}^2  + \|d_{\tau} \bu_{r,h}^{\mathrm{P},n}\|_{0,\Omega_{\mathrm{P}}}^2   + \|d_{\tau} \by_{s,h}^{\mathrm{P},n}\|_{1,\Omega_{\mathrm{P}}}^2  + \|d_{\tau} p^{\mathrm{P},n}_h\|_{0,\Omega_{\mathrm{P}}}^2  + \|d_{\tau} \bu_{s,h}^{\mathrm{P},n}\|_{0,\Omega_{\mathrm{P}}}^2 \bigr) \nonumber \\ 
    & \quad + \tau \sum_{n=1}^N \sum_{E \in \mathcal{E}_{\Sigma}}  \bigl( \gamma - \frac{1 + \varsigma}{\widehat{\epsilon_f}^{\prime}}\bigr) h_E^{-1} \|\bn_{\mathrm{S}} \cdot \bu_{f,h}^{\mathrm{S},n}    + \bn_{\mathrm{P}} \cdot \bu_{r,h}^{\mathrm{P},n} + \bn_{\mathrm{P}} \cdot d_{\tau} \by_{s,h}^{\mathrm{P},n}\|_{0,E}^2  \nonumber \\
    & \lesssim      \|\bu_{f,h}^{0}\|_{0,\Omega_{\mathrm{S}}}^2   + \|\bu_{r,h}^{0}\|_{0,\Omega_{\mathrm{P}}}^2  + \|\by_{s,h}^{0}\|_{1,\Omega_{\mathrm{P}}}^2  + \|p^{\mathrm{P},0}_h\|_{0,\Omega_{\mathrm{P}}}^2 + \|\bu_{s,h}^{0}\|_{0,\Omega_{\mathrm{P}}}^2 
     + (\breve{\epsilon}_f^{\prime})^{-1} \tau \sum_{n=1}^N  \|\boldsymbol{f}_{\mathrm{S}}(t_n)\|_{0,\Omega_{\mathrm{S}}}^2  \nonumber   \\ 
     & \quad + \epsilon_1^{-1} \tau \sum_{n=1}^N  ( \|\ff_{\mathrm{P}}(t_n)\|_{0,\Omega_{\mathrm{P}}}^2  + \|\theta(t_n)\|_{0,\Omega_{\mathrm{P}}}^2  + \|r_{\mathrm{S}}(t_n)\|_{0,\Omega_{\mathrm{S}}}^2 ) 
     + \tau \sum_{n=1}^N  (\rho_p \ff_{\mathrm{P}}(t_n), d_{\tau} \by_{s,h}^{\mathrm{P},n})_{\Omega_{\mathrm{P}}} \nonumber \\
    & \quad  + \epsilon_1 \tau \sum_{n=1}^N  ( \|p^{\mathrm{P},n}_h\|_{0,\Omega_{\mathrm{P}}}^2    + \|p^{\mathrm{S},n}_h\|_{L^2(\Omega_{\mathrm{S}})}^2  + \|\bu_{r,h}^{\mathrm{P},n}\|_{0,\Omega_{\mathrm{P}}}^2 ) .
\end{align}
Next, to bound the second-last term on the right-hand side of \eqref{stability_1}, we use summation 
by parts as follows 
\begin{align}\label{stability_2}
& \tau \sum_{n=1}^N(\ff_{\mathrm{P}}(t_n), d_\tau \by_{s,h}^{\mathrm{P},n})=(\ff_{\mathrm{P}}(t_N), \by_{s,h}^{\mathrm{P},N})-(\ff_{\mathrm{P}}(0), \by_{s,h}^0)-\tau \sum_{n=1}^{N-1}(d_\tau \ff_{\mathrm{P}}^n, \by_{s,h}^{\mathrm{P},n})  \\&  \nonumber  \leq \frac{\epsilon_1}{2}\|\by_{s,h}^{\mathrm{P},N}\|_{0,\Omega_{\mathrm{P}}}^2   +\frac{1}{2 \epsilon_1}\|\ff_{\mathrm{P}}(t_N)\|_{0,\Omega_{\mathrm{P}}}^2 +\frac{\tau}{2} \sum_{n=1}^{N-1}\|\by_{s,h}^{\mathrm{P},n}\|_{0,\Omega_{\mathrm{P}}}^2 +\frac{1}{2}(\|\by_{s,h}^{\mathrm{P},0}\|_{0,\Omega_{\mathrm{P}}}^2    +\|\ff_{\mathrm{P}}(0)\|_{0,\Omega_{\mathrm{P}}}^2   +\tau \sum_{n=1}^{N-1}\|d_\tau \ff_{\mathrm{P}}^n\|_{0,\Omega_{\mathrm{P}}}^2 ).
\end{align}
\cred{In turn, we apply the inf–sup condition \cblue{in Lemma \ref{inf_sup}} to \(p_h^{\mathrm{S},n}\) and \(p_h^{\mathrm{P},n}\), combine this step with \eqref{fully} and the continuity bounds from Lemma~\ref{coercivity-continuity}, and then proceed to sum over \(n=1,\dots,N\) and to multiply by \(\tau\), to finally obtain:}
\begin{align}\label{stability_3}
   & \cred{\epsilon_1 \tau \sum_{n=1}^N ( \|p_h^{\mathrm{S},n}\|^2_{W_{f,h}} + \|p_h^{\mathrm{P},n}\|^2_{W_{p,h}}  ) } \nonumber \\
    & \cred{ \lesssim \epsilon_1 \tau \sum_{n=1}^N \bigl( \|\bu_{f,h}^{\mathrm{S},n}\|^2_{1,\Omega_{\mathrm{S}}} + \| \bu_{r,h}^{\mathrm{P},n}\|^2_{0,\Omega_{\mathrm{P}}}  
    + \left|\bu_{f,h}^{\mathrm{S},n}- \partial_t \by_{s,h}^{\mathrm{P},n}\right|_{\mathrm{BJS}}^2  + \|\ff_{\mathrm{S}}(t_n)\|^2_{0,\Omega_{\mathrm{S}}} + \|\ff_{\mathrm{P}}(t_n)\|^2_{0,\Omega_{\mathrm{P}}}} \nonumber  \\
    &  \quad \cred{+ \sum_{E \in \mathcal{E}_{\Sigma}} \frac{\gamma \mu_f}{h_E} \| \bn_{\mathrm{S}} \cdot \bu_{f,h}^{\mathrm{S},n} + \bn_{\mathrm{P}} \cdot \bu_{r,h}^{\mathrm{P},n} + \bn_{\mathrm{P}} \cdot d_{\tau} \by_{s,h}^{\mathrm{P},n} \|^2_{0,E} 
    \bigr).}
\end{align}
\cred{By combining \eqref{stability_1}, \eqref{stability_2}, and \eqref{stability_3}, and by choosing $\epsilon_2$ sufficiently small and then $\epsilon_1$ sufficiently small, we can assert that}
\begin{align*}
   & \cred{ \|\bu_{f,h}^{\mathrm{S},N}\|_{0,\Omega_{\mathrm{S}}}^2  + \|\bu_{r,h}^{\mathrm{P},N}\|_{0,\Omega_{\mathrm{P}}}^2 + \|\by_{s,h}^{\mathrm{P},N} \|_{1,\Omega_{\mathrm{P}}}^2 + \|p^{\mathrm{P},N}_h\|_{0,\Omega_{\mathrm{P}}}^2
  + \|\bu_{s,h}^{\mathrm{P},N}\|_{0,\Omega_{\mathrm{P}}}^2   + \tau \sum_{n=1}^N \bigl( \left| \bu^{\mathrm{S},n}_{f,h} - d_{\tau} \by_{s,h}^{\mathrm{P},n} \right|_{\mathrm{BJS}}^2 } \\ & \quad  \cred{+  \left| \bu^{\mathrm{P},n}_{r,h}  \right|_{\mathrm{BJS}}^2     + \|\bu_{r,h}^{\mathrm{P},n}\|_{0,\Omega_{\mathrm{P}}}^2 + \|p_h^{\mathrm{S},n}\|_{W_{f,h}}^2 + \|p_h^{\mathrm{P},n}\|_{W_{p,h}}^2 \bigr) + \tau^2 \sum_{n=1}^N \bigl( \|d_{\tau} \bu_{f,h}^{\mathrm{S},n}\|_{0,\Omega_{\mathrm{S}}}^2 + \|d_{\tau} \bu_{r,h}^{\mathrm{P},n}\|_{0,\Omega_{\mathrm{P}}}^2 } \\ & 
  \quad \cred{ + \|d_{\tau} \by_{s,h}^{\mathrm{P},n}\|_{1,\Omega_{\mathrm{P}}}^2 + \|d_{\tau} p^{\mathrm{P},n}_h\|_{0,\Omega_{\mathrm{P}}}^2  + \|d_{\tau} \bu_{s,h}^{\mathrm{P},n}\|_{0,\Omega_{\mathrm{P}}}^2  \bigr) + \tau \sum_{n=1}^N \biggl( \frac{3 - 2(1 + \varsigma)\widehat{\epsilon_f}^{\prime} C_{\mathrm{tr}} - \breve{\epsilon}_f^{\prime}}{4} \biggr) \|\bu^{\mathrm{S},n}_{f,h}\|_{1,\Omega_{\mathrm{S}}}^2 } \\ 
  & \quad   \cred{+ \tau \sum_{n=1}^N \sum_{E \in \mathcal{E}_{\Sigma}} \biggl( \gamma - \frac{1 + \varsigma}{\widehat{\epsilon_f}^{\prime}} \biggr) h_E^{-1}\|\bn_{\mathrm{S}} \cdot \bu_{f,h}^{\mathrm{S},n} + \bn_{\mathrm{P}} \cdot \bu_{r,h}^{\mathrm{P},n} + \bn_{\mathrm{P}} \cdot d_{\tau} \by_{s,h}^{\mathrm{P},n}\|_{0,E}^2 } \\
  & \cred{ \lesssim \|\bu_{f,h}^{0}\|_{0,\Omega_{\mathrm{S}}}^2  + \|\bu_{r,h}^{0}\|_{0,\Omega_{\mathrm{P}}}^2 + \|\by_{s,h}^{0}\|_{1,\Omega_{\mathrm{P}}}^2  + \|p^{\mathrm{P},0}_h\|_{0,\Omega_{\mathrm{P}}}^2  + \|\bu_{s,h}^{0}\|_{0,\Omega_{\mathrm{P}}}^2  + \left\|\ff_{\mathrm{P}}(0)\right\|_{0,\Omega_{\mathrm{P}}} 
    + (\breve{\epsilon}_f^{\prime})^{-1} \tau \sum_{n=1}^N  \|\boldsymbol{f}_S(t_n)\|_{0,\Omega_{\mathrm{S}}}^2 }  \\& \quad \cred{+ \epsilon_1^{-1} \tau \sum_{n=1}^N  \bigl(  \|\ff_{\mathrm{P}}(t_n)\|_{0,\Omega_{\mathrm{P}}}^2 + \|\theta(t_n)\|_{0,\Omega_{\mathrm{P}}}^2  + \|r_{\mathrm{S}}(t_n)\|_{0,\Omega_{\mathrm{S}}}^2 \bigr)    +\frac{\tau}{2} \sum_{n=1}^{N-1}\left\|\by_{s,h}^{\mathrm{P},n}\right\|_{0,\Omega_{\mathrm{P}}}  +\tau \sum_{n=1}^{N}\left\|d_\tau \ff_{\mathrm{P}}^n\right\|_{0,\Omega_{\mathrm{P}}}.}
\end{align*}
\cred{Hence, applying the discrete Gronwall inequality \cite{MR1299729}, we obtain the desired result.}
\end{proof}

\subsection{Error analysis}
We now turn to analyzing the spatial discretization error. Let $k_f$ and $s_f$ be the degrees of polynomials in $\mathbf{V}_{f,h}$ and $\mathrm{W}_{f,h}$, let $k_p$ and $s_p$ be the degrees of polynomials in $\mathbf{V}_{r,h}$ and $\mathrm{W}_{p,h}$ respectively, and let $k_s$
and $s_s$ be the polynomial degree in $\mathbf{V}_{s,h}$ and $\mathbf{W}_{s,h}$.

Let $Q_{f, h}, Q_{p, h}$, and $\bQ_{s, h}$ be the $L^2$-projection operators onto $\mathrm{W}_{f,h},$$ \mathrm{W}_{p,h}$, and $\bW_{s,h}$ respectively, satisfying:
\begin{subequations}
\begin{align}
(p^{\mathrm{S}}-Q_{f, h} p^{\mathrm{S}}, q^{\mathrm{S}}_h)_{\Omega_{\mathrm{S}}}&=0 & \forall q^{\mathrm{S}}_h \in \mathrm{W}_{f,h}, \label{ees1}\\
(p^{\mathrm{P}}-Q_{p, h} p^{\mathrm{P}}, q^{\mathrm{P}}_h)_{\Omega_{\mathrm{P}}}& =0 & \forall q^{\mathrm{P}}_h \in \mathrm{W}_{p,h}, \label{ees2} \\
(\bu_{s}^{\mathrm{P}}-\bQ_{s, h} \bu_s^{\mathrm{P}}, \bv_{s,h}^{\mathrm{P}})_{\Omega_{\mathrm{P}}}&=0 & \forall \bv_{s,h}^{\mathrm{P}} \in \mathbf{W}_{s,h}. \label{ees3}
\end{align}
\end{subequations}
These operators satisfy the approximation properties \cite{MR1299729}:
\begin{subequations}
\begin{align}
\|p^{\mathrm{S}}-Q_{f, h} p^{\mathrm{S}}\|_{0,\Omega_{\mathrm{S}}} & \leq C_1^{\star} h^{r_{s f}}\|p^{\mathrm{S}}\|_{r_{s f},\Omega_{\mathrm{S}}} & 0 \leq r_{s_f} \leq s_f + 1, \label{ps}\\
\|p^{\mathrm{P}}-Q_{p, h} p^{\mathrm{P}}\|_{0,\Omega_{\mathrm{P}}} &\leq  C_1^{\star} h^{r_{s_p}}\|p^{\mathrm{P}}\|_{{r_{s_p}},\Omega_{\mathrm{P}}} & 0 \leq r_{s_p} \leq s_p + 1, \label{ph}\\
\|\bu_{s}^{\mathrm{P}}-\bQ_{s, h} \bu_s^{\mathrm{P}}\|_{0,\Omega_{\mathrm{P}}} &\leq C_1^{\star} h^{{r}_{s_s}}\|\bu_s^{\mathrm{P}}\|_{{r_{s_s}},\Omega_{\mathrm{P}}} & 0 \leq {r}_{s_s} \leq s_s + 1 \label{vh}.
\end{align}
\end{subequations}
Next, we consider a Stokes-type projection operator $(\bS_{f, h}, R_{f, h}): \mathbf{V}_f \rightarrow \mathbf{V}_{f, h} \times \mathrm{W}_{f,h}$, defined for all $\bv_f^{\mathrm{S}}\in \mathbf{V}_f$ as
\begin{subequations}
\begin{align}
a_f^{\mathrm{S}}(\bS_{f, h} \bv_f^{\mathrm{S}}, \bv_{f,h}^{\mathrm{S}})-b_f^{\mathrm{S}}(\bv_{f,h}^{\mathrm{S}}, R_{f, h} \bv_f^{\mathrm{S}})&=a_f^{\mathrm{S}}(\bv_f^{\mathrm{S}}, \bv_{f,h}^{\mathrm{S}}) & \forall \bv_{f,h}^{\mathrm{S}} \in \mathbf{V}_{f, h}, \label{sp1} \\
b_f^{\mathrm{S}}(\bS_{f, h} \bv_f^{\mathrm{S}}, q^{\mathrm{S}}_h)&=b_f(\bv_f^{\mathrm{S}}, q^{\mathrm{S}}_h) & \forall q^{\mathrm{S}}_h \in \mathrm{W}_{f,h} \label{sp2}.
\end{align}
\end{subequations}
The operator $\bS_{f, h}$ satisfies the following approximation property \cite{MR3851065,MR2788393}:
\begin{align}\label{sp3}
\|\bv_f^{\mathrm{S}}-\bS_{f, h} \bv_f^{\mathrm{S}}\|_{1,\Omega_{\mathrm{S}}} \leq C_1^{\star} h^{r_{k_f}-1}\|\bv_f^{\mathrm{S}}\|_{{r_{k_f}},\Omega_{\mathrm{S}}}, \quad 1 \leq r_{k_f} \leq k_f + 1 .
\end{align}
Similarly, let $\bPi_{r, h}$ be the Stokes projection onto $\mathbf{V}_{r, h}$ satisfying  for all $\bv_r^{\mathrm{P}} \in \mathbf{V}_r$,
\begin{subequations}
\begin{align}\label{mfe1}
(\nabla \cdot \bPi_{r, h} \bv_r^{\mathrm{P}}, q^{\mathrm{P}}_h)& =(\nabla \cdot \bv_r^{\mathrm{P}}, q^{\mathrm{P}}_h) \quad \forall q^{\mathrm{P}}_h \in \mathrm{W}_{p,h}, \\
 \|\bv_r^{\mathrm{P}}-\bPi_{r, h} \bv_r^{\mathrm{P}}\|_{0,\Omega_{\mathrm{P}}} &\leq C_1^{\star} h^{r_{k_p}-1}\|\bv_r^{\mathrm{P}}\|_{H^{r_{k_p}}(\Omega_{\mathrm{P}})}, \quad 1 \leq r_{k_p} \leq k_p + 1. \label{mfe3}
\end{align}\end{subequations}
Finally, let $\bS_{s, h}$ be the Scott--Zhang interpolant from $\mathbf{V}_s$ onto $\mathbf{V}_{s, h}$, satisfying \cite{MR1011446}:
\begin{align}\label{sz1}
& \|\by_s^{\mathrm{P}}-\bS_{s, h} \by_s^{\mathrm{P}}\|_{0,\Omega_{\mathrm{P}}}+h\left|\by_s^{\mathrm{P}}-\bS_{s, h} \by_s^{\mathrm{P}}\right|_{1,\Omega_{\mathrm{P}}} \leq C_1^{\star} h^{r_{k_s}}\|\by_s^{\mathrm{P}}\|_{{r_{k s}},\Omega_{\mathrm{P}}}, 
 \quad 1 \leq r_{k_s} \leq k_s + 1 .
\end{align}
\begin{theorem}\label{appendix_fully}
Assume~\ref{(H1)}-\ref{(H3)}  and suppose that 
$\| \bu_{f,h}^{\mathrm{S},n} \|_{1,\Omega_{\mathrm{S}}} < \frac{\mu_f}{2 S_f^2 K_f^3}$ for $1 \leq n \leq N$ and for any \(\widehat{\epsilon}_f^{\prime},\breve{\epsilon}_f^{\prime}\) that satisfies 
\[
\frac{1 - (\varsigma + 1)\widehat{\epsilon}_f^{\prime} C_{\mathrm{tr}}}{2} > 0,
\]
where \(\varsigma \in \{-1, 0, 1\}\), provided that \(\gamma > (\varsigma + 1)(\widehat{\epsilon}_f^{\prime})^{-1}\). Assume further that the solution of \eqref{mixed-primal} is sufficiently smooth. Then,  the solution of  
\eqref{fully} with initial conditions $\bu_{f,h}^{\mathrm{S}}(0) = \bI_{f,h} \bu_{f,0}, \; \bu_{r,h}^{\mathrm{P}}(0) = \bI_{r,h} \bu_{r,0}, \; \by_{s,h}^{\mathrm{P}}(0) = \bI_{s,h} \by_{s,0},  \; p_{h}^{\mathrm{P}}(0) = Q_{r,h} p^{p,0}$, and $\bu_{s,h}^{\mathrm{P}}(0) = \bQ_{s,h} \bu_{s,0}$, satisfies
\begin{align*}
& \|\bu_f^{\mathrm{S}}-\bu_{f,h}^{\mathrm{S}}\|^2_{l^{\infty}(I;\mathbf{L}^2(\Omega_{\mathrm{S}}))}  + \|\bu_r^{\mathrm{P}}-\bu_{r,h}^{\mathrm{P}}\|^2_{l^{\infty}(I;\mathbf{L}^2(\Omega_{\mathrm{P}}))} 
+ \|p^{\mathrm{P}}-p^{\mathrm{P}}_h\|^2_{l^{\infty}(I;L^2(\Omega_{\mathrm{P}}))}  + \|\by_s^{\mathrm{P}}-\by_{s,h}^{\mathrm{P}}\|^2_{l^{\infty}(I;\mathbf{H}^1(\Omega_{\mathrm{P}}))}  \\ 
& \quad + \|\bu_{s}^{\mathrm{P}} - \bu_{s,h}^{\mathrm{P}}\|^2_{l^{\infty}(I;\mathbf{L}^2(\Omega_{\mathrm{P}}))}   + | \bu_r^{\mathrm{P}}- \bu_{r,h}^{\mathrm{P}} |_{l^2(I;\mathrm{BJS})}^2  + | (\bu_f^{\mathrm{S}}- d_{\tau} \by_s^{\mathrm{P}} ) - (\bu^{\mathrm{S}}_{f,h} - d_{\tau} \by_{s,h}^{\mathrm{P}} )|^2_{l^2(I;\mathrm{BJS})}  \\ 
&\quad  
+ \|\bu_r^{\mathrm{P}}-\bu_{r,h}^{\mathrm{P}}\|^2_{l^2(I;{\mathbf{L}^2(\Omega_{\mathrm{P}})})}  + \| p^{\mathrm{S}} -  p^{\mathrm{S}}_h \|^2_{l^2(I;L^2(\Omega_{\mathrm{S}}))}  
+ \| p^{\mathrm{P}} - p^{\mathrm{P}}_h \|^2_{l^2(I;L^2(\Omega_{\mathrm{P}}))} 
\\& \quad  + \left( \gamma -  \frac{1+ \varsigma} {\widehat{\epsilon_f}^{\prime}} \right) 
\sum_{E \in \mathcal{E}_{\Sigma}}  \frac{\mu_f}{h_E}  
\|(\bu_f^{\mathrm{S}}- \bu^{\mathrm{S}}_{f,h}) \cdot \bn_{\mathrm{S}}  + (\bu_r^{\mathrm{P}}- \bu^{\mathrm{P}}_{r,h}) \cdot \bn_{\mathrm{P}} 
+ d_{\tau} (\by_s^{\mathrm{P}}- \by^{\mathrm{P}}_{s,h}) \cdot \bn_{\mathrm{S}} \|_{0, E}^2 \\
&  \lesssim \exp(T) \bigg[ h^{2 r_{k_p}-2} 
\bigl( \| \bu_r^{\mathrm{P}} \|_{l^2(I;\mathbf{H}^{r_{k_p}}(\Omega_{\mathrm{P}}))}^2 + \| \bu_r^{\mathrm{P}} \|_{l^2(I;\mathbf{H}^{r_{k_p}+1}(\Omega_{\mathrm{P}}))}^2 
+ \| \bu_r^{\mathrm{P}} \|_{l^{\infty}(I;\mathbf{H}^{r_{k_p}+1}(\Omega_{\mathrm{P}}))}^2 + \| \bu_r^{\mathrm{P}} \|_{l^{\infty}(I;\mathbf{H}^{r_{k_p}}(\Omega_{\mathrm{P}}))}^2 \\
& \quad    + \|\partial_t \bu_r^{\mathrm{P}} \|_{L^2(I;\mathbf{H}^{r_{k_p}}(\Omega_{\mathrm{P}}))}^2 
+ \|\partial_t \bu_r^{\mathrm{P}} \|_{L^2(I;\mathbf{H}^{r_{k_p}+1}(\Omega_{\mathrm{P}}))}^2 + \|\partial_t \bu_r^{\mathrm{P}} \|_{L^{\infty}(I;\mathbf{H}^{r_{k_p}}(\Omega_{\mathrm{P}}))}^2 + \|\partial_{tt} \bu_r^{\mathrm{P}} \|_{L^{\infty}(I;\mathbf{H}^{r_{k_p}}(\Omega_{\mathrm{P}}))}^2 \bigr)  \\
& \quad   + h^{2 r_{s_s}} \bigl( \| \bu_s^{\mathrm{P}} \|_{l^2(I;\mathbf{H}^{r_{s_s}}(\Omega_{\mathrm{P}}))}^2 + \| \bu_s^{\mathrm{P}} \|_{l^{\infty}(I;\mathbf{H}^{r_{s_s}}(\Omega_{\mathrm{P}}))}^2 + \|\partial_t \bu_s^{\mathrm{P}} \|_{L^2(I;\mathbf{H}^{r_{s_s}}(\Omega_{\mathrm{P}}))}^2 \bigr) \\
& \quad    + h^{2 r_{k_s}-2} \bigl( \| \by_s^{\mathrm{P}} \|_{l^{\infty}(I;\mathbf{H}^{r_{k_s}+1}(\Omega_{\mathrm{P}}))}^2    + \|\partial_t \by_s^{\mathrm{P}} \|_{L^{\infty}(I;\mathbf{H}^{r_{k_s}+1}(\Omega_{\mathrm{P}}))}^2 + \|\partial_t \by_s^{\mathrm{P}} \|_{L^{\infty}(I;\mathbf{H}^{r_{k_s}}(\Omega_{\mathrm{P}}))}^2  \\
& \quad  + \|\partial_t \by_s^{\mathrm{P}} \|_{L^2(I;\mathbf{H}^{r_{k_s}+1}(\Omega_{\mathrm{P}}))}^2  
+ \|\partial_t \by_s^{\mathrm{P}} \|_{L^2(I;\mathbf{H}^{r_{k_s}+1}(\Omega_{\mathrm{P}}))}^2  + \|\partial_{tt} \by_s^{\mathrm{P}} \|_{L^{\infty}(I;\mathbf{H}^{r_{k_s}+1}(\Omega_{\mathrm{P}}))}^2 
+ \|\partial_{tt} \by_s^{\mathrm{P}} \|_{L^{\infty}(I;\mathbf{H}^{r_{k_s}}(\Omega_{\mathrm{P}}))}^2 \bigr) \\
& \quad  + h^{2 r_{s_p}} \bigl( \|p^{\mathrm{P}}\|^2_{l^{\infty}(I;H^{r_{s_p}}(\Omega_{\mathrm{P}}))} 
+ \| \partial_t p^{\mathrm{P}}\|^2_{L^2(I;H^{r_{s_p}}(\Omega_{\mathrm{P}}))} \bigr) 
+ h^{2 r_{s_f}}\|p^{\mathrm{S}}\|^2_{l^2(I;H^{r_{s_f}}(\Omega_{\mathrm{S}}))} \\
& \quad + h^{2 r_{k_f}}  \bigl(\| \bu_f^{\mathrm{S}} \|_{l^2(I;\mathbf{H}^{r_{k_f} +1}(\Omega_{\mathrm{S}}))}^2  + \|\partial_t \bu_f^{\mathrm{S}}\|^2_{L^2(I;\mathbf{H}^{r_{k_f} +1}(\Omega_{\mathrm{S}}))} \bigr)   \\
& \quad  + \tau^2 \bigl( \|\partial_{tt} \bu_r^{\mathrm{P}} \|_{L^2(I;\mathbf{L}^2(\Omega_{\mathrm{P}}))}^2  
+ \|\partial_{tt} \bu_r^{\mathrm{P}} \|_{L^{\infty}(I;\mathbf{L}^2(\Omega_{\mathrm{P}}))}^2   + \|\partial_{ttt} \bu_r^{\mathrm{P}} \|_{L^{2}(I;\mathbf{L}^2(\Omega_{\mathrm{P}}))}^2 
+ \|\partial_{tt} \bu_s^{\mathrm{P}} \|_{L^2(I;\mathbf{L}^2(\Omega_{\mathrm{P}}))}^2   \\ 
& \quad + \|\partial_{tt} \bu_s^{\mathrm{P}} \|_{L^{\infty}(I;\mathbf{L}^2(\Omega_{\mathrm{P}}))}^2
+ \|\partial_{ttt} \bu_s^{\mathrm{P}} \|_{L^{2}(I;\mathbf{L}^2(\Omega_{\mathrm{P}}))}^2 + \|\partial_{tt} \by_s^{\mathrm{P}} \|_{L^{\infty}(I;\mathbf{H}^1(\Omega_{\mathrm{P}}))}^2 + \|\partial_{tt} \by_s^{\mathrm{P}} \|_{L^2(I;\mathbf{H}^1(\Omega_{\mathrm{P}}))}^2\\
& \quad   + \|\partial_{ttt} \by_s^{\mathrm{P}} \|_{L^{2}(I;\mathbf{H}^1(\Omega_{\mathrm{P}}))}^2 
+ \| \partial_{tt} p^{\mathrm{P}}\|^2_{L^2(I;L^2(\Omega_{\mathrm{P}}))}  +  \|\partial_{tt} \bu_f^{\mathrm{S}}\|^2_{L^2(I;\mathbf{L}^2(\Omega_{\mathrm{P}}))} + \| \partial_t \bu_f^{\mathrm{S}} \|_{L^2(0, T; H^1(\Omega_{\mathrm{S}}))}^2 \bigr) \bigg], 
\end{align*}
where $0 \leq r_{k_f} \leq k_f$, $0 \leq r_{s_f} \leq s_f + 1$, $0 \leq r_{k_p}  \leq k_p$, $0 \leq r_{s_p} \leq s_p + 1$,  
$ 0 \leq r_{k_s} \leq k_s$,  $0 \leq r_{s_s} \leq s_s + 1$.
\end{theorem}
The proof of this result is postponed to Appendix \ref{error_analysis_fully_appendix}. 
\section{Numerical tests}\label{section6}
We now present several numerical examples that validate the accuracy of the derived error estimates. All simulations were performed using  the open-source finite element library Gridap (version 0.17.12) \cite{badia2020gridap}, whose 
high-level API  supports all components necessary for defining the problem, including integration over facets and separate domains, e.g., \(\Sigma\), \(\Omega_{\mathrm{S}}\), and \(\Omega_{\mathrm{P}}\), as required by ~\eqref{fully}.
\begin{remark}[Choice of the penalty parameter $\gamma$]
\cred{The stability analysis in Theorem \ref{stability_fully} shows that the Nitsche parameter 
$\gamma$ must be sufficiently large to guarantee coercivity of the 
interface contribution and, consequently, well-posedness of the 
discrete problem. In particular, $\gamma$ must dominate the 
negative consistency terms arising from the application of the trace 
and Young inequalities. }

\cred{Importantly, the required lower bound for $\gamma$ is independent of 
the mesh size $h$ and depends only on physical parameters and trace 
constants. In practice, moderate values of $\gamma$ are sufficient to 
ensure stability without deteriorating the conditioning of the linear 
system. In all numerical experiments, values $\gamma = 30$ or $40$ 
were found to provide stable and accurate results.}
\end{remark}
\subsection{Convergence tests against manufactured solutions}
The accuracy of the   discretization is verified using the following closed-form solutions defined on the domains $\Omega_{\mathrm{P}}=(0,1) \times(0.5,1), \Omega_{\mathrm{S}}=(0,1) \times(0,0.5)$, separated by the interface $\Sigma=(0,1) \times\{0.5\}$.

The synthetic model parameters are taken as
$\lambda_p=\mu_p=\mu_f=10$, $\alpha_{\mathrm{BJS}}=1$,  $\phi=0.1$, $\kappa=\rho_p= \rho_f= K =1$, $\theta =0.0$, $\gamma = 40$,
all regarded non-dimensional as we will be simply testing the convergence of the FE approximations. 
The model problem is then complemented with the appropriate Dirichlet boundary conditions and initial data. These functions do not necessarily fulfill the interface conditions, so additional terms are required giving modified relations on $\Sigma$:
\begin{gather*}
\bu_f^{\mathrm{S}}\cdot \bn_{\mathrm{S}}+\left(\partial_t \by_s^{\mathrm{P}}+\bu_r^{\mathrm{P}}\right) \cdot \bn_{\mathrm{P}}=m_{\Sigma,\mathrm{ex}}^1 , \quad 
-\left(\bsigma^{\mathrm{S}} \bn_{\mathrm{S}}\right) \cdot \bn_{\mathrm{S}}= -\left(\bsigma_f^{\mathrm{P}} \bn_{\mathrm{P}}\right) \cdot \bn_{\mathrm{P}} + m_{\Sigma,\mathrm{ex}}^2  , \\
-\left(\bsigma_f^{\mathrm{S}} \bn_{\mathrm{S}}\right) \cdot \btau_{f, j}=\mu_{f}
  \alpha_{\mathrm{BJS}} \sqrt{Z_j^{-1}}\left(\bu_f^{\mathrm{S}}-{\partial_t \by_s^{\mathrm{P}}}\right) \cdot \btau_{f, j} + m_{\Sigma,\mathrm{ex}}^4,\\ \bsigma_f^{\mathrm{S}} \bn_{\mathrm{S}}+\bsigma_f^{\mathrm{P}} \bn_{\mathrm{P}} +\bsigma_s^{\mathrm{P}} \bn_{\mathrm{P}}= m_{\Sigma,\mathrm{ex}}^3  , \quad
-(\bsigma_f^{\mathrm{P}} \bn_{\mathrm{P}}) \cdot \tau_{f,j} = \mu_{f}
  \alpha_{\mathrm{BJS}} \sqrt{Z_j^{-1}}\bu_r^{\mathrm{P}} \cdot \btau_{f, j} + m_{\Sigma,\mathrm{ex}}^5,
\end{gather*}
and the additional scalar and vector terms $m_{\Sigma,\mathrm{ex}}^i$ (computed with the exact solutions \eqref{num} entail the following changes in the linear functional $F$ (cf. \eqref{eq:linear-funct}))
\begin{align*}
\cblue{F(\vec{\by})} &:= \cblue{ F_1(\bv_f^{\mathrm{S}}) +  F_2(\bv_r^{\mathrm{P}}) +  F_3(\bw_s^{\mathrm{P}}) + F_4(q^{\mathrm{S}})}, \quad \cblue{\text{with}} \\
 F_1(\bv_f^{\mathrm{S}}) &:= \int_{\Omega_{\mathrm{S}}}  \ff_{\mathrm{S}} \bv_f^{\mathrm{S}} + \int_{{\Sigma}} \frac{\gamma \mu_f}{h_E} m_{\Sigma,\mathrm{ex}}^1 (\bv_f^{\mathrm{S}} \cdot \bn_{\mathrm{S}}) - \int_{{\Sigma}} \frac{\gamma \mu_f}{h_E} m_{\Sigma,\mathrm{ex}}^1  (2 \mu_f \beps(\bu_{f}^{\mathrm{S}}) )\bn_{\mathrm{S}}\bn_{\mathrm{S}}  - \langle  m_{\Sigma,\mathrm{ex}}^4 , \bv_f^{\mathrm{S}} \cdot \tau_{f,j} \rangle_{\Sigma},
 \\ 
 F_2(\bv_r^{\mathrm{P}}) &:= \int_{\Omega_{\mathrm{P}}} \rho_f \phi \ff_{\mathrm{P}} \bv_f^{\mathrm{S}} + \int_{{\Sigma}} \frac{\gamma \mu_f}{h_E} m_{\Sigma,\mathrm{ex}}^1 (\bv_r^{\mathrm{P}} \cdot \bn_{\mathrm{P}}) + \langle  m_{\Sigma,\mathrm{ex}}^2 , \bv_r^{\mathrm{P}} \cdot \bn_{\mathrm{P}} \rangle_{\Sigma}   - \langle  m_{\Sigma,\mathrm{ex}}^5 , \bv_r^{\mathrm{P}} \cdot \tau_{f,j} \rangle_{\Sigma}, \\
 F_3(\bw_s^{\mathrm{P}})&:=  \int_{\Omega_{\mathrm{P}}} \rho_p \ff_{\mathrm{P}} \bw_s^{\mathrm{P}} + \int_{{\Sigma}} \frac{\gamma \mu_f}{h_E} m_{\Sigma,\mathrm{ex}}^1 (\bw_s^{\mathrm{P}} \cdot \bn_{\mathrm{P}}) +  \langle  m_{\Sigma,\mathrm{ex}}^3 , \bw_s^{\mathrm{P}}  \rangle_{\Sigma}   + \langle  m_{\Sigma,\mathrm{ex}}^4 , \bw_s^{\mathrm{P}} \cdot \tau_{f,j} \rangle_{\Sigma}, \\ 
  F_4(q^{\mathrm{S}}) & := - \int_{{\Sigma}} \frac{\gamma \mu_f}{h_E} m_{\Sigma,\mathrm{ex}}^1  q^{\mathrm{S}}.
\end{align*}
\begin{remark}
   The spatial and temporal convergence studies use separate manufactured solutions to evaluate the corresponding discretization errors independently. For the spatial convergence test, smooth non-polynomial exact solutions involving trigonometric functions are considered, and the time step is chosen sufficiently small so that temporal errors are negligible. For the temporal convergence test, sufficiently refined fixed mesh is used so that spatial errors do not affect the observed rates. Moreover, the exact solutions are chosen such that their spatial components belong to the corresponding finite element spaces, allowing the optimal convergence rates of the backward Euler scheme to be clearly observed.
\end{remark}
$\mathbf{Space~ convergence:}$
The exact solution for space convergence is defined as:
\begin{align*}\label{num}
\nonumber
\bu_f^{\mathrm{S}} &= \left(\begin{array}{c}
 t x^3 \cos(4 \pi y)\\ 
- 2 t x^3 \sin(4 \pi y)
\end{array}\right), \quad  
p^{\mathrm{S}} = t^2(1-\sin(4 \pi x) \sin(4 \pi y)), \\
\bu_r^{\mathrm{P}} &= \left(\begin{array}{c}
t^2 \sin^2(4\pi y) - t x^3 \cos(4 \pi y)\\
t^2 \sin^2(4\pi y) + 2 t x^3 \sin(4 \pi y)
\end{array}\right),  \quad
\bu_s^{\mathrm{P}} = \left(\begin{array}{c}
 t x^3 \cos(4 \pi y)\\
-2 t x^3 \sin(4 \pi y)
\end{array}\right), \\
\by_s^{\mathrm{P}} &= \left(\begin{array}{c}
0.5 t^2 x^3 \cos(4 \pi y)\\  
- t^2 x^3 \sin(4 \pi y)
\end{array}\right), \quad
p^{\mathrm{P}} = t^2(1-\sin(4 \pi x) \sin(4 \pi y)).
\end{align*}
We generate successively refined simplicial grids and use a sufficiently small (non dimensional) time step $\tau = h \times 10^{-3}$ and final time $T=0.001$, to guarantee that the error produced by the time discretization does not dominate. Errors between the approximate and exact solutions are shown in Table \ref{conv_test_space}.  Theoretically, we observe that both
the relative velocity and solid displacement exhibit sub-optimal convergence, but in practice only the relative velocity exhibits this behavior in our tests. 
\newline
\begin{table}[t!]
\setlength{\tabcolsep}{2pt}
\centering
\small
\begin{tabular}{r|c|c|c|c|c|c|c|c|c}
\toprule
DoFs & $h$ & $\|e_{\bu_f^{\mathrm{S}}}\|_{l^2(\mathbf{H}^1)}$ & rate & $\|e_{p^{\mathrm{S}}}\|_{l^2(L^2)}$ & rate & $\|e_{\bu_r}\|_{l^2(\mathbf{L}^2)}$ & rate & $\|e_{p^{\mathrm{P}}}\|_{l^2(L^2)}$ & rate \\
\midrule
111   & 0.5000 & $1.79\times10^{-4}$ & -- & $3.91\times10^{-3}$ & -- & $6.76\times10^{-3}$ & -- & $1.87\times10^{-6}$ & -- \\
439   & 0.2500 & $2.90\times10^{-5}$ & 2.627 & $2.06\times10^{-3}$ & 0.929 & $8.86\times10^{-4}$ & 2.932 & $3.45\times10^{-7}$ & 2.440 \\
1767  & 0.1250 & $4.67\times10^{-6}$ & 2.632 & $4.23\times10^{-4}$ & 2.281 & $1.07\times10^{-4}$ & 3.054 & $1.55\times10^{-8}$ & 4.475 \\
7111  & 0.0625 & $9.41\times10^{-7}$ & 2.312 & $8.84\times10^{-5}$ & 2.258 & $1.80\times10^{-5}$ & 2.571 & $1.16\times10^{-9}$ & 3.741 \\
28551 & 0.0312 & $2.52\times10^{-7}$ & 1.901 & $1.73\times10^{-5}$ & 2.353 & $4.85\times10^{-6}$ & 1.887 & $2.41\times10^{-10}$& 2.269 \\
\bottomrule 
\end{tabular}
\bigskip 
\begin{tabular}{r|c|c|c|c|c}
\toprule 
DoFs & $h$ & $\|e_{\by_s^{\mathrm{P}}}\|_{l^2(\mathbf{H}^1)}$ & rate & $\|e_{\bu_s^{\mathrm{P}}}\|_{l^2(\mathbf{L}^2)}$ & rate   \\
\midrule 
111   & 0.5000 & $2.98\times10^{-6}$ & -- & $1.08\times10^{-5}$ & -- \\
439   & 0.2500 & $1.17\times10^{-6}$ & 1.350 & $1.67\times10^{-6}$ & 2.700 \\
1767  & 0.1250 & $1.65\times10^{-7}$ & 2.825 & $3.10\times10^{-7}$ & 2.424 \\
7111  & 0.0625 & $2.40\times10^{-8}$ & 2.779 & $6.78\times10^{-8}$ & 2.196 \\
28551 & 0.0312 & $4.15\times10^{-9}$ & 2.532 & $1.64\times10^{-8}$ & 2.045 \\
\bottomrule 
\end{tabular}
\caption{Experimental errors and convergence rates computed at the final time step for variables $\bu_f^{\mathrm{S}}, \bu_r^{\mathrm{P}}, p^{\mathrm{S}}, p^{\mathrm{P}}, \by_s^{\mathrm{P}}$, and $\bu_s^{\mathrm{P}}$ using the finite element spaces $\mathbb{P}_2^2$–$\mathbb{P}_1$–$\mathbb{P}_2^2$–$\mathbb{P}_1$–$\mathbb{P}_2^2$–$\mathbb{P}_1^2$.}
\label{conv_test_space}
\end{table}
\vspace{0.4mm}
$\mathbf{Time ~convergence:}$
The exact solution for temporal convergence is defined as:
\begin{align*}
\bu_f^{\mathrm{S}} &= \left(\begin{array}{c}
t^2(y^2-2xy)\\
t^2(y^2+2xy)
\end{array}\right), \quad
p_f^{\mathrm{S}} = t^2(x-y)-\frac{t^2}{4}, \quad
\bu_r^{\mathrm{P}} = \left(\begin{array}{c}
t^2(x^2+xy)\\
t^2(-y^2+xy)
\end{array}\right), \quad
p_h^{\mathrm{P}} = t^2(2x-y)-\frac{t^2}{4}, \\
\by_s^{\mathrm{P}} &= \left(\begin{array}{c}
\frac{1}{2}t^2(x^2+xy)\\
\frac{1}{2}t^2(y^2-xy)
\end{array}\right), \quad
\bu_s^{\mathrm{P}} = \left(\begin{array}{c}
t(x^2+xy)\\
t(y^2-xy)
\end{array}\right).
\end{align*}
The backward Euler method is assessed for time convergence and verified by partitioning the time interval $(0, 0.1)$ into successively refined uniform discretizations and computing cumulative errors $\hat{e}_s = \bigl( \sum_{n=1}^N \tau  \| s(t_{n+1}) - s_h^{n+1}\|^2_{\star}\bigr)^{1/2}$,
where $\| \cdot \|_{\star}$ is the appropriate space norm for the generic vector or scalar field $s$. For this test we use a fixed mesh $h=0.03125$. The results are shown in Table \ref{conv_test_time_M}, confirming the expected first-order convergence.
\begin{table}[t!]
\setlength{\tabcolsep}{2pt}
\centering
\begin{tabular}{r|c|c|c|c|c|c|c|c}
\toprule
$\tau$ & $\hat{e}_{\bu_f^{\mathrm{S}}}$ & rate &
$\hat{e}_{\bu_r^{\mathrm{P}}}$ & rate &
$\hat{e}_{p^{\mathrm{S}}}$ & rate &
$\hat{e}_{p_h}$ & rate \\
\midrule
$0.0500$ & $3.19\times10^{-4}$ & $-$
& $1.20\times10^{-1}$ & $-$
& $5.04\times10^{-3}$ & $-$
& $8.08\times10^{-3}$ & $-$ \\
$0.0250$ & $1.48\times10^{-4}$ & $1.105$
& $6.02\times10^{-2}$ & $0.996$
& $2.32\times10^{-3}$ & $1.117$
& $3.50\times10^{-3}$ & $1.207$ \\
$0.0125$ & $7.19\times10^{-5}$ & $1.046$
& $3.02\times10^{-2}$ & $0.994$
& $1.11\times10^{-3}$ & $1.071$
& $1.61\times10^{-3}$ & $1.116$ \\
$0.0063$ & $3.65\times10^{-5}$ & $0.978$
& $1.52\times10^{-2}$ & $0.991$
& $5.43\times10^{-4}$ & $1.024$
& $7.74\times10^{-4}$ & $1.062$ \\
\bottomrule
\end{tabular}
\bigskip
\begin{tabular}{r|c|c|c|c}
\toprule
$\tau$ & $\hat{e}_{\by_s^{\mathrm{P}}}$ & rate &
$\hat{e}_{\bu_s^{\mathrm{P}}}$ & rate \\
\midrule
$0.0500$ & $7.45\times10^{-4}$ & $-$
& $4.42\times10^{-3}$ & $-$ \\
$0.0250$ & $3.66\times10^{-4}$ & $1.027$
& $2.27\times10^{-3}$ & $0.959$ \\
$0.0125$ & $1.86\times10^{-4}$ & $0.972$
& $1.17\times10^{-3}$ & $0.953$ \\
$0.0063$ & $9.64\times10^{-5}$ & $0.952$
& $6.05\times10^{-4}$ & $0.955$ \\
\bottomrule
\end{tabular}
\caption{Experimental cumulative errors associated with the temporal discretization and convergence rates for the approximate solutions $\bu_f^{\mathrm{S}}, p^{\mathrm{S}}, \bu_r^{\mathrm{P}}, p_h, \by_s^{\mathrm{P}}, \text{ and } \bu_s^{\mathrm{P}}$, using a backward Euler scheme.}
\label{conv_test_time_M}
\end{table}
\subsection{2D flow in a channel with rigid obstacles between porous layers}
The computational domain is a two-dimensional channel that contains a free-flow region bounded above and below by porous layers. Three rigid, eye-shaped obstacles are embedded in the central free-flow region. At the inlet, the free-flow velocity \( \bu_f^{\mathrm{S}} \), pore velocity \( \bu_r^{\mathrm{P}} \), and solid displacement \( \by_s^{\mathrm{P}} \) are prescribed by the inflow profile
$
\bu_{\mathrm{in}} = \bigl( \frac{1}{0.49}(0.1\,(x_2 + 0.2)\,(1.2 - x_2)){0.49}, \; 0 \bigr)$.
No-slip conditions are applied on the surfaces of the rigid obstacles, i.e., \( \bu_f^{\mathrm{S}} = \mathbf{0} \). On the exterior walls, \( \bu_r^{\mathrm{P}} = \mathbf{0} \) and \( \by_s^{\mathrm{P}} = \mathbf{0} \). At the outlet, a do-nothing (zero-traction) boundary condition is imposed on both the velocity and displacement fields.

The physical parameters are defined as follows: $\mu_f = 0.01$, $\rho_f = 1.0$, $\rho_p = 1.0$, $\mu_p = 1.0336 \times 10^3$, $\lambda_p = 4.9364 \times 10^4$, $\kappa = 1 \times 10^{-3}$, $K = 1 \times 10^6$, $\theta = 0$, $\alpha_{\mathrm{BJS}} = 1$, $\phi = 0.3$, and $\gamma = 30$. The simulation is performed with a time step $\tau = 10^{-3}$ until the final time $T = 0.1$.

\begin{figure}[t!]
  \centering
   \subfloat[Fluid and porous velocity]{%
    \includegraphics[width=0.49\textwidth]{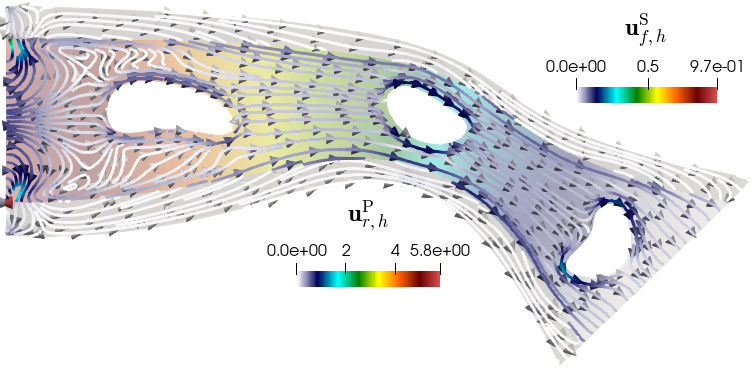}
  } 
  \subfloat[Fluid and porous pressure]{%
    \includegraphics[width=0.49\textwidth]{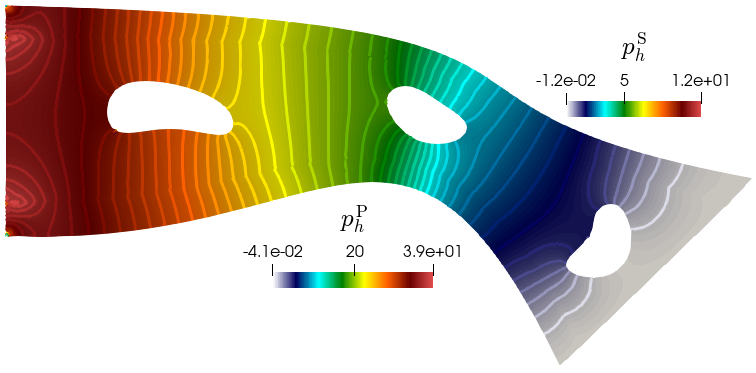}
  }\\
  \subfloat[Solid displacement]{%
    \includegraphics[width=0.49\textwidth]{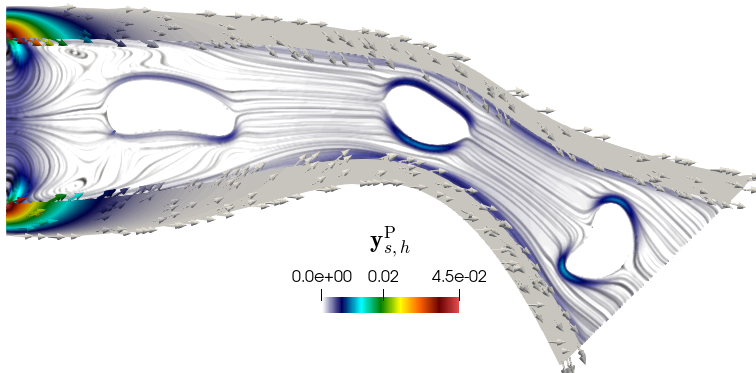}
  }
  \subfloat[Solid velocity]{%
    \includegraphics[width=0.49\textwidth]{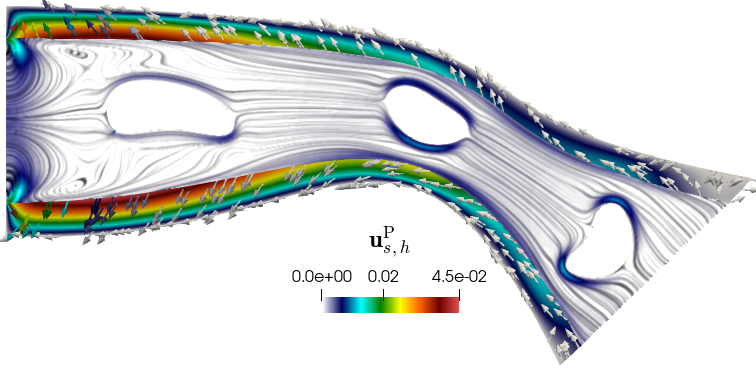}}
\caption{Coupled fluid–poroelastic simulation results:  
(a) free-flow velocity $\bu_{f,h}^\mathrm{S}$ and porous velocity $\bu_{r,h}^\mathrm{P}$ (streamlines and vectors);  
(b) profiles and iso-contours of pressures $p_{h}^\mathrm{S}$ (free flow) and $p_{h}^\mathrm{P}$ (poroelastic medium);  
(c) magnitude of solid displacement $|\by_{s,h}^\mathrm{P}|$;  
(d) magnitude of solid velocity $|\bu_{s,h}^\mathrm{P}|$.  
Colorbars indicate magnitudes; streamlines and vectors depict flow or solid motion.}
  \label{fig:coupled_flow_results}
\end{figure}

Figure~\ref{fig:coupled_flow_results}(a) shows velocity streamlines and vectors across the entire domain for both the fluid velocity $\boldsymbol{u}_{f,h}^\mathrm{S}$ and the porous velocity $\bu_{r,h}^\mathrm{P}$. The free flow accelerates along the main channel and forms recirculation zones behind the obstacles. Higher pore-fluid velocities occur near obstacle boundaries and along the interface, indicating regions of enhanced exchange between the free fluid and the porous medium. Figure~\ref{fig:coupled_flow_results}(b) shows the pressures $p_{h}^\mathrm{S}$ in the free-flow region and $p_{h}^\mathrm{P}$ in the poroelastic medium. The pressure varies along the channel, with a clear difference across the fluid–poroelastic interface caused by the resistance of the porous matrix. This pressure difference drives fluid exchange across the interface and influences the motion of the solid skeleton.  Figure~\ref{fig:coupled_flow_results}(c) shows  the solid displacement magnitude $|\boldsymbol{y}_{s,h}^\mathrm{P}|$ in the poroelastic layer. The largest displacements occur near the fluid–solid interface and behind the obstacles, indicating how the solid deforms in response to the flow. Figure~\ref{fig:coupled_flow_results}(d) shows the poroelastic solid velocity magnitude $|\bu_{s,h}^\mathrm{P}|$. Most solid motion occurs near the interface and along obstacle boundaries, corresponding to areas of high pressure gradient and strong pore-fluid flow. This indicates strong coupling between the fluid and the solid.
\subsection{3D simulation of the blood flow through a microfluidic chip with cylindrical poroelastic obstacles}
To demonstrate that our numerical method is both stable and accurate, we perform a 3D simulation of fluid flow inside a microfluidic chip \cite{MR4942764}, with a slight modification to the radius of the obstacles. The chip dimensions are $2\,\text{cm} \times 4.2 \, \text{cm} \times 0.5 \,\text{cm} $. On the left side, there is one inlet for blood to enter, and on the right side, there are two outlets where blood leaves. Similarly, there is a single inlet on the left for water to flow in and two outlets on the right for water to flow out. Inside the chip, ten pillars made of poroelastic materials (such as hydrogels or polymer composites) are arranged in specific positions. Six of these pillars have a radius of 0.2 cm, and the remaining four have a radius of 0.15 cm. The parameters are defined as $\mu_f = 0.01$, $\rho_{f} = 1.0$, $\rho_{p} = 1.2$, $\mu_{p} = 1.0336 \times 10^{3}$, $\lambda_{p} = 4.9364 \times 10^{4}$, $\kappa = 1 \times 10^{-3}$, $K = 1 \times 10^{6}$, $\theta = 0$, $\alpha_{\mathrm{BJS}} = 1$, $\phi = 0.3$, and $\gamma = 30$, and their units are in the CGS system. This problem was solved with a time step of $\tau = 10^{-3}$ and a final time of $T = 1$. The boundary conditions are as follows: on the inlet,  
\[
\bu_f^{\mathrm{S}} \;=\; \bu_{\mathrm{in}}
\;=\;
\biggl(\frac{20\,(z - 0.5)\,(1.5 - z)\,z\,(0.7 - z)}{0.49},\,0,\,0\biggr).
\]
On the lateral walls, the channel top, and the channel bottom (excluding the pillar hole), 
\(\bu_f^{\mathrm{S}} = \mathbf{0}\). On the outlet, 
\(\bsigma_{f}^\mathrm{S}\,\mathbf{n}_{\mathrm{S}} = \mathbf{0}\). Finally, on the top and bottom of the cylinder, 
\(\by_s^{\mathrm{P}} = \bu_r^{\mathrm{P}} = \mathbf{0}\). We report in Figure~\ref{fig:3d} the numerical results, showing the expected behavior of flow through deformable obstacles.

\begin{figure}[t!]
    \centering
    \includegraphics[width=0.55\linewidth]{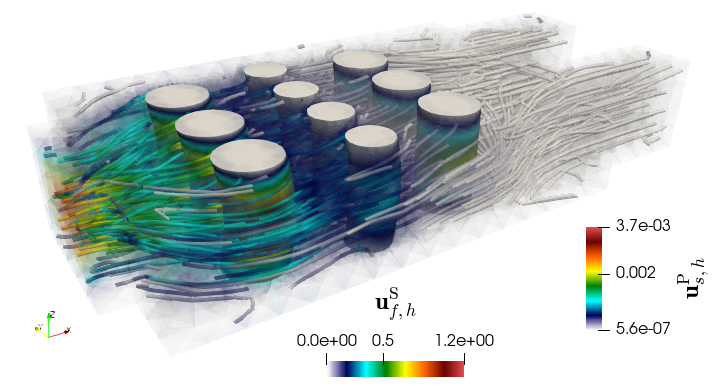} \\
    \includegraphics[width=0.55\linewidth]{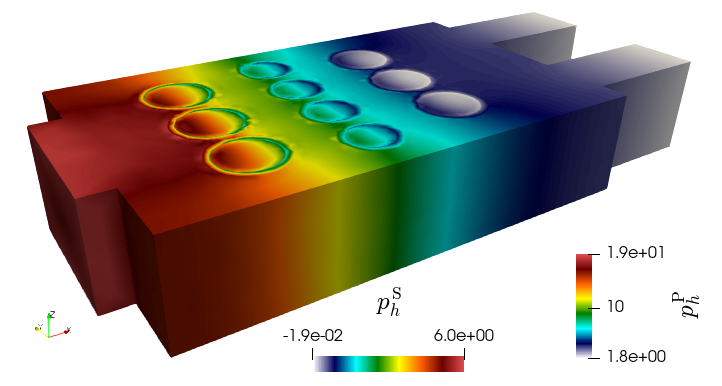} \\
    \includegraphics[width=0.55\linewidth]{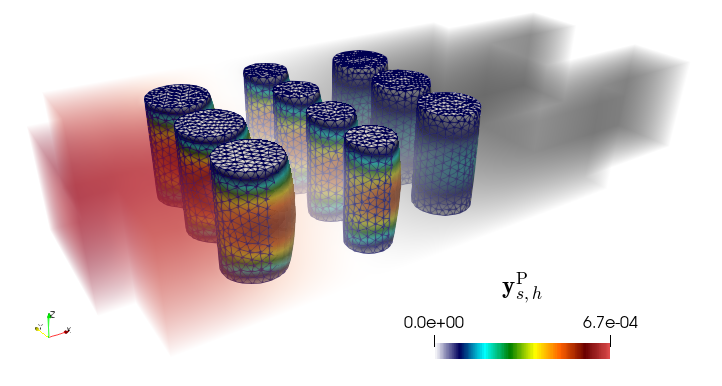}
    \caption{Velocity streamlines (left), pressure profiles (center), and displacement magnitude associated with flow through a microfluidic chip.}
    \label{fig:3d}
\end{figure}
Figure~\ref{fig:3d}(a) displays the blood-flow streamlines around the poroelastic pillars. Higher fluid velocities, represented by yellow/red colors, are observed near the inlet and in less obstructed regions, while lower velocities occur near the pillar surfaces due to flow resistance. The contours on the pillars represent the structural velocity, $\bu^\mathrm{P}_{s,h}$. Figure~\ref{fig:3d}(b) shows the pressure distribution in the fluid and porous regions. The fluid pressure, $p_h^\mathrm{S}$, decreases significantly from the inlet to the outlets, with large pressure gradients near the pillar boundaries. The pressure inside the pillars corresponds to the porous pressure, $p_h^\mathrm{P}$. Figure~\ref{fig:3d}(c) illustrates the magnitude of the solid displacement, $\by^\mathrm{P}_{s,h}$. The largest deformation, approximately $6.7 \times 10^{-4}\,\mathrm{m}$, occurs in the smaller pillars located in high-velocity regions, highlighting the coupling between the fluid flow and the poroelastic response.

\section{Conclusion}~\label{sec:concl}
A Nitsche-based formulation is proposed for the Navier--Stokes/generalized poroelasticity model, and the well-posedness of the discrete formulations is proved using DAE theory and Banach fixed point theorem. A priori error estimates for the fully discrete schemes are derived. Finally, we conduct a series of numerical experiments to validate the theoretical findings on spatio-temporal convergence. Specifically, we present a two-dimensional simulation of flow in a channel containing obstacles between a porous substrate, and a three-dimensional simulation of blood flow through a microfluidic device featuring cylindrical poroelastic obstacles. Further perspectives of this work include the extension to the fully nonlinear regime, as well as other types of transmission conditions that would allow for greater generality in the types of poromechanical problems we can tackle. 

\section*{Declarations}
\noindent\textbf{Acknowledgments.} 
AB was supported by the Ministry of Education, Government of India - MHRD. NB received support from Centro de Modelamiento Matematico (CMM), Proyecto Basal FB210005. RRB received  support from the Australian Research Council through the \textsc{Future Fellowship} grant FT220100496 and \textsc{Discovery Project} grant DP22010316, and from the Center of Advanced Study (CAS) at the Norwegian Academy of Science and Letters under the program \textsc{Mathematical Challenges in Brain Mechanics}.
\par\smallskip
\noindent
{\bf Data availability.} Data generated during the research discussed in the paper will be made available upon reasonable request.
\par\smallskip
\noindent
{\bf Conflict of interest statement}.  The Authors declare that they have no conflict of interest.
\bibliographystyle{siam}
\bibliography{references}

\appendix 
\section{Proof of Theorem \ref{appendix_fully}}\label{error_analysis_fully_appendix} 
\begin{proof}
We introduce the errors for all variables and split them into approximation and discretization errors:
\begin{align*}
\boldsymbol{e}_f^n & :=\bu_f^{{\mathrm{S},n}}-\bu_{f, h}^{{\mathrm{S}},n}=(\bu_f^{{\mathrm{S}},n}-\bI_{f, h} \bu_f^{{\mathrm{S}},n})+(\bI_{f, h} \bu_f^{{\mathrm{S}},n}-\bu^{\mathrm{S},n}_{f,h}):=\bchi_{f,h}^n+\bphi_{f,h}^n, \\
\boldsymbol{e}_r^n & :=\bu_r^{\mathrm{P},n}-\bu_{r,h}^{\mathrm{P},n}=(\bu_r^{\mathrm{P},n}-\bPi_{r, h} \bu_r^{\mathrm{P},n})+(\bPi_{r, h} \bu_r^{\mathrm{P},n}-\bu_{r,h}^{\mathrm{P},n}):=\bchi_{r,h}^n+\bphi_{r,h}^n, \\
\boldsymbol{e}_y^n & :=\by_s^{{\mathrm{P}},n}-\by_{s, h}^{\mathrm{P},n}=(\by_s^{{\mathrm{P}},n}-\bS_{s, h} \by_s^{{\mathrm{P}},n})+(\bS_{s, h} \by_s^{{\mathrm{P}},n}-\by_{s, h}^{\mathrm{P},n}):=\bchi_{y,h}^n+\bphi_{y,h}^n, \\
\boldsymbol{e}_{s}^n & :=\bu_s^{\mathrm{P},n}-\bu_{s,h}^{\mathrm{P},n}=(\bu_s^{\mathrm{P},n}-\bQ_{s, h} \bu_s^{\mathrm{P},n})+(\bQ_{s, h} \bu_s^{\mathrm{P},n}-\bu_{s,h}^{\mathrm{P},n}):=\bchi_{s,h}^n+\bphi_{s,h}^n, \\
e_{f p}^n & :=p^{\mathrm{S},n}-p^{\mathrm{S}}_h=(p^{\mathrm{S},n}-Q_{f, h} p^{\mathrm{S},n})+(Q_{f, h} p^{\mathrm{S},n}-p^{\mathrm{S}, n}_h):=\chi_{f p,h}^n+\phi_{f p, h}^n, \\
e_{p p}^n & :=p^{\mathrm{P},n}-p^{\mathrm{P}, n}_h=(p^{\mathrm{P},n}-Q_{p, h} p^{\mathrm{P},n})+(Q_{p, h} p^{\mathrm{P},n}-p^{\mathrm{P},n}_h):=\chi_{p p,h}^n+\phi_{p p, h}^n .
\end{align*} 
Denote the time discretisation errors as $r_n(\phi) = d_{\tau} \phi - \partial_t \phi $, for $\phi \in \{ \bu_f^{\mathrm{S},n}, \bu_r^{\mathrm{P},n},\by_s^{\mathrm{P},n}, p^{\mathrm{P},n}, \bu_s^{\mathrm{P},n} \}$. Subtracting \eqref{fully} from \eqref{n_weak} and adding the resulting equations, we obtain the following error equation
\begin{align*}
& m_{\rho_f }(d_{\tau} \boldsymbol{e}_f^n, \bv_{f,h}^{\mathrm{S}}) + m_{\rho_f \phi}(d_{\tau} \boldsymbol{e}_r^n, \bw_{s,h}^{\mathrm{P}})  + m_{\rho_p}(d_{\tau}\boldsymbol{e}_s^n, \bw_{s,h}^{\mathrm{P}}) + m_{\rho_f \phi}(d_{\tau}\boldsymbol{e}_r^n, \bv_{r,h}^{\mathrm{P}})  + m_{\rho_f \phi}(d_{\tau}\boldsymbol{e}_s^n, \bv_{r,h}^{\mathrm{P}})  -m_{\rho_p}(d_{\tau} \boldsymbol{e}_y^n, \bv_{s,h}^{\mathrm{P}}) \\
&\quad  + m_{\rho_p}(\boldsymbol{e}_s^n, \bv_{s,h}^{\mathrm{P}})  
+ a_f^{\mathrm{S}}(\boldsymbol{e}^n_f, \bv_{f,h}^{\mathrm{S}}) + a_{f}^{\mathrm{P}}(\boldsymbol{e}_r^n, \bw_{s,h}^{\mathrm{P}})  + a_s^{\mathrm{P}}(\boldsymbol{e}_y^n, \bw_{s,h}^{\mathrm{P}}) + a_{f}^{\mathrm{P}}(\boldsymbol{e}_r^n, \bv_{r,h}^{\mathrm{P}}) + a_{f}^{\mathrm{P}}(d_{\tau} \boldsymbol{e}_y^n, \bv_{r,h}^{\mathrm{P}})  + a_{f}^{\mathrm{P}}(d_{\tau} \boldsymbol{e}_{y}^n, \bw_{s,h}^{\mathrm{P}})   \\
&\quad + b^{\mathrm{S}}(\bv_{f,h}^{\mathrm{S}}, e_{fp}^n)  + b_s^{\mathrm{P}}(\bw_{s,h}^{\mathrm{P}}, e^{pp}_h)  + b_f^{\mathrm{P}}(\bv_{r,h}^{\mathrm{P}}, e_{pp}^n) - m_{\theta}(\boldsymbol{e}_r^n, \bw_{s,h}^{\mathrm{P}}) - m_{\theta}( \boldsymbol{e}_s^n, \bw_{s,h}^{\mathrm{P}})  - m_{\theta}(\boldsymbol{e}_r^n, \bv_{r,h}^{\mathrm{P}}) - m_{\theta}( \boldsymbol{e}_{s}^n, \bv_{r,h}^{\mathrm{P}})\\ 
&\quad   + m_{\phi^2/\kappa}(\boldsymbol{e}_r^n, \bv_{r,h}^{\mathrm{P}})
+ b_{\Gamma}(\bv_{f,h}^{\mathrm{S}}, \bv_{r,h}^{\mathrm{P}}, \bw_{s,h}^{\mathrm{P}} ; \boldsymbol{e}^n_f, e_{fp}^{\mathrm{S}} ) - b_s^{\mathrm{P}}(d_{\tau} \boldsymbol{e}_{y}^n, q^{\mathrm{P}}_h)  - b_f^{\mathrm{P}}(\boldsymbol{e}_r^n, q^{\mathrm{P}}_h)  - b^{\mathrm{S}}(\boldsymbol{e}^n_f, q^{\mathrm{S}}_h)  
+ b_{\Gamma}( \boldsymbol{e}_f^n, \boldsymbol{e}_{r}^n, d_{\tau} \boldsymbol{e}_y^n ; \varsigma  \bv^{\mathrm{S}}_f, -q_h^{\mathrm{S}})   \\
& \quad   + a_{\mathrm{BJS}}(\boldsymbol{e}_{f}^{n}, d_{\tau} \boldsymbol{e}_y^{n} ; \bv_{f,h}^{\mathrm{S}}, \bw_{s,h}^{\mathrm{P}})  +c_{\Gamma}(\boldsymbol{e}_f^n, \boldsymbol{e}_r^n, d_{\tau} \boldsymbol{e}_y^n;  \bv_{f,h}^{\mathrm{S}}, \bv_{r,h}^{\mathrm{P}},  \bw_{s,h}^{\mathrm{P}}  )  +  b_{\mathrm{BJS}}(\boldsymbol{e}_{r}^{n}, \bv_{r,h}^{\mathrm{P}}) + ((1-\phi)^2 K^{-1} d_{\tau} e^n_{pp}, q^{\mathrm{P}}_h)_{\Omega_{\mathrm{P}}} \\ 
& \quad  + ( \bu_f^{\mathrm{S},n} \cdot \nabla  \bu_f^{\mathrm{S},n} -  \bu_{f,h}^{\mathrm{S},n-1} \cdot \nabla  \bu_{f,h}^{\mathrm{S},n},  \bv_{f,h}^{\mathrm{S},n} ) + \mathcal{E} = 0,
\end{align*}
where the discretisation error is given by 
\begin{align*}
    \mathcal{E} & = m_{\rho_f }(r_n(\bu_f^{\mathrm{S}}), \bv_{f,h}^{\mathrm{S}}) + m_{\rho_f \phi}(r_n(\bu_r^{\mathrm{P}}), \bw_{s,h}^{\mathrm{P}})  + m_{\rho_p}(r_n(\bu_{s}^{\mathrm{P}}), \bw_{s,h}^{\mathrm{P}})  + m_{\rho_f \phi}(r_n(\bu_r^{\mathrm{P}}),  \bv_{r,h}^{\mathrm{P}})  + m_{\rho_f \phi}(r_n(\bu_{s}^{\mathrm{P}}), \bv_{r,h}^{\mathrm{P}}) \\ 
    & \quad  + a_{f}^{\mathrm{P}}(r_n(\by_{s}^{\mathrm{P}}), \bv_{r,h}^{\mathrm{P}}) + ((1 - \phi)^2 K^{-1} r_n(p^{\mathrm{P}}), q_{h}^{\mathrm{P}})_{\Omega_{\mathrm{P}}} 
    - b_s^{\mathrm{P}}(r_n(\by_{s}^{\mathrm{P}}), q_h^{\mathrm{P}})   + a_{f}^{\mathrm{P}}(r_n(\by_{s}^{\mathrm{P}}),  \bw_{s,h}^{\mathrm{P}}) + a_{\mathrm{BJS}}(\boldsymbol{0}, r_n(\by_{s}^{\mathrm{P}}); \bv_{f,h}^{\mathrm{S}},  \bw_{s,h}^{\mathrm{P}})   \\ & \quad   + b_{\Gamma}(\cero, \cero, r_n(\by_{s}^{\mathrm{P}})  ; \bv_{f,h}^{\mathrm{S}}, 0 )  - b_{\Gamma}(\cero, \cero, r_n(\by_{s}^{\mathrm{P}}) ; \cero, q_h^{\mathrm{S}}  ) + c_{\Gamma}(\cero, \cero, r_n(\by_{s}^{\mathrm{P}});  \bv_{f,h}^{\mathrm{S}}, \bv_{r,h}^{\mathrm{P}}, \bw_{y,h}^{\mathrm{P}} ).
\end{align*}
Setting $\bv_{f,h}^{\mathrm{S}}= \bphi_{f,h}^n, \bv_{r,h}^{\mathrm{P}} = \bphi_{r,h}^n, \bw_{s,h}^{\mathrm{P}} = d_{\tau} \bphi_{y,h}^n, \bv_{s,h}^{\mathrm{P}}= \bphi_{s,h}^n, q^{\mathrm{S}}_h  = \phi_{fp,h}^n, \text{ and } q^{\mathrm{P}}_h = \phi_{pp,h}^n$, we conclude that the following terms simplify due to the properties of the projection operators \eqref{ees2}, \eqref{sp2} and \eqref{mfe1}:
\begin{gather*}
b^{\mathrm{S}}(\bchi_{f,h}^n, \phi_{f p, h}^n) =b_f^{\mathrm{P}}(\bchi_{r,h}^n, \phi_{p p, h}^n)= b_f^{\mathrm{P}}(\bphi_{r,h}^n, \chi_{pp,h}^n) =0, \\ 
((1-\phi)^2 K^{-1} d_{\tau} \chi_{p p,h}^n, \phi_{p p, h}^n) = 0.
\end{gather*}
Rearranging terms and using the results above, the error equation 
becomes 
\begin{align}\label{ne1}
&a_f^{\mathrm{S}}(\bphi_{f,h}^n,\bphi_{f,h}^n)+ a_{f}^{\mathrm{P}}(\bphi_{r,h}^n, d_{\tau} \bphi_{y,h}^n) +a_{f}^{\mathrm{P}}(\bphi_{r,h}^n, \bphi_{r,h}^n) + a_{f}^{\mathrm{P}}(d_{\tau} \bphi_{y,h}^n, \bphi_{r,h}^n) + a_{f}^{\mathrm{P}}(d_{\tau} \bphi_{y,h}^n, d_{\tau} \bphi_{y,h}^n) +a_s^{\mathrm{P}}(\bphi_{y,h}^n,d_{\tau} \bphi_{y,h}^n)  \nonumber \\
&\quad +a_{\mathrm{BJS}}(\bphi_{f,h}^n, d_{\tau} \bphi_{y,h}^n ; \bphi_{f,h}^n, d_{\tau} \bphi_{y,h}^n)   + b_{\Gamma}(\bphi_{f,h}^n, \bphi_{r,h}^n, d_{\tau} \bphi_{y,h}^n ; \bphi_{f,h}^n, 0 ) + b_{\Gamma}(\bphi_{f,h}^n, \bphi_{r,h}^n, d_{\tau} \bphi_{y,h}^n ; \varsigma  \bphi_{f,h}^n, 0)  \nonumber \\
&\quad+c_{\Gamma}(\bphi_{f,h}^n, \bphi_{r,h}^n, d_{\tau} \bphi_{y,h}^n ;  \bphi_{f,h}^n, \bphi_{r,h}^n, d_{\tau} \bphi_{y,h}^n  ) + m_{\rho_f \phi}(d_{\tau}\bphi_{r,h}^n, \bphi_{s,h}^n) + m_{\rho_p}(d_{\tau}\bphi_{s,h}^n,  \bphi_{s,h}^n) + m_{\rho_f \phi }(d_{\tau}\bphi_{r,h}^n, \bphi_{r,h}^n)\nonumber \\
&\quad  + m_{\rho_f \phi}(d_{\tau}\bphi_{s,h}^n, \bphi_{r,h}^n) - m_{\theta}(\bphi_{r,h}^n, \bphi_{r,h}^n)   + ((1-\phi)^2 K^{-1} d_{\tau} \phi_{pp,h}^n, \phi_{pp,h}^n)_{\Omega_{\mathrm{P}}} - m_{\theta}(\bphi_{r,h}^n, \bphi_{s,h}^n) - m_{\theta}(\bphi_{s,h}^n, \bphi_{s,h}^n) \nonumber \\
&\quad - m_{\theta}( \bphi_{s,h}^n, \bphi_{r,h}^n) + m_{\phi^2/\kappa}(\bphi_{r,h}^n, \bphi_{r,h}^n) + m_{\rho_f}(d_{\tau} \bphi_{f,h}^n, \bphi_{f,h}^n)  + b_{\mathrm{BJS}}(\bphi_{r,h}^n, \bphi_{r,h}^n)\nonumber \\
&= \mathcal{J}_1 + \mathcal{J}_2 + \mathcal{J}_3 + \mathcal{J}_4 + \mathcal{J}_5 +\mathcal{J}_6 +\mathcal{J}_7 +\mathcal{J}_8 + \mathcal{J}_9,
\end{align}
where the right-hand side terms are defined as follows
\begin{align*}
    \mathcal{J}_1  := & -a_f^{\mathrm{S}}(\bchi_{f,h}^n, \bphi_{f,h}^n)  + m_{\theta}(\bchi_{r,h}^n, \bphi_{r,h}^n) 
    + m_{\theta}(\bchi_{s,h}^n, \bphi_{r,h}^n) 
    - m_{\phi^2/\kappa}(\bchi_{r,h}^n, \bphi_{r,h}^n)  - m_{\rho_f \phi }(d_{\tau}\bchi_{r,h}^n, \bphi_{r,h}^n) 
    - m_{\rho_f \phi}(d_{\tau}\bchi_{s,h}^n, \bphi_{r,h}^n)   \\&
     - m_{\rho_f}(d_{\tau} \bchi_{f,h}^n, \phi_{f,h}^n)
      - m_{\rho_f \phi}(r_n(\bu_r^{\mathrm{P}}), \bphi_{r,h}^n)  - m_{\rho_f \phi}(r_n(\bu_{s}^{\mathrm{P}}), \bphi_{r,h}^n)  - m_{\rho_f}(r_n(\bu_f^{\mathrm{S}}), \phi_{f,h}^n), \\ 
\mathcal{J}_2 := & - a_{f}^{\mathrm{P}}(\bchi_{r,h}^n, \bphi_{r,h}^n)  - a_{f}^{\mathrm{P}}(d_{\tau} \bchi_{y,h}^n, \bphi_{r,h}^n) - a_{f}^{\mathrm{P}}(r_n(\by_{s}^{\mathrm{P}}), \bphi_{r,h}^n), \\ 
    \mathcal{J}_3 := & -\sum_{j=1}^{d-1}\left\langle\mu_f \alpha_{\mathrm{BJS}} \sqrt{Z_j^{-1}}(\bchi_{f,h}^n-d_{\tau} \bchi_{y,h}^n) \cdot \btau_{f, j},(\bphi_{f,h}^n-d_{\tau} \bphi_{y,h}^n) \cdot \btau_{f, j}\right\rangle_{\Sigma} \\ &  - ((1 - \phi)^2 K^{-1} r_n(p^{\mathrm{P}}), \phi_{pp,h}^n)_{\Omega_{\mathrm{P}}}   -\sum_{j=1}^{d-1}\bigl\langle\mu_f \alpha_{\mathrm{BJS}} \sqrt{Z_j^{-1}}\bchi_{r,h}^n \cdot \btau_{f, j}, \bphi_{r,h}^n \cdot \btau_{f, j}\bigr\rangle_{\Sigma} , \\
    \mathcal{J}_4 := & -b^{\mathrm{S}}(\bphi_{f,h}^n, \chi_{fp,h}^n) 
    + b_s^{\mathrm{P}}(d_{\tau} \bchi_{y,h}^n, \phi_{pp,h}^n) + b_s^{\mathrm{P}}(r_n(\by_{s}^{\mathrm{P}}), \phi_{pp,h}^n) , \\  
    \mathcal{J}_5 := & - a_{f}^{\mathrm{P}}(\bchi_{r,h}^n,   d_{\tau} \bphi_{y,h}^n) 
    - a_s^{\mathrm{P}}(\bchi_{y,h}^n,d_{\tau} \bphi_{y,h}^n) 
    - a_{f}^{\mathrm{P}}(d_{\tau} \bchi_{y,h}^n,   d_{\tau} \bphi_{y,h}^n)  + m_{\theta}(\bchi_{r,h}^n, d_{\tau} \bphi_{y,h}^n) 
     - m_{\rho_p}(d_{\tau}\bchi_{y,h}^n, d_{\tau} \bphi_{y,h}^n)  
     \\& + m_{\theta}(\bchi_{s,h}^n, d_{\tau} \bphi_{y,h}^n)  - b_s^{\mathrm{P}}(d_{\tau} \bphi_{y,h}^n, \chi_{pp,h}^n) 
    - m_{\rho_f \phi}(d_{\tau}\bchi_{r,h}^n, d_{\tau} \bphi_{y,h}^n)    - a_{f}^{\mathrm{P}}(r_n(\by_{s}^{\mathrm{P}}), d_{\tau} \bphi_{y,h}^n) \\ &  - a_{\mathrm{BJS}}(\boldsymbol{0}, r_n(\by_{s}^{\mathrm{P}}); \bphi_{f,h}^n, d_{\tau} \bphi_{y,h}^n)   - m_{\rho_f \phi}(r_n(\bu_r^{\mathrm{P}}), d_{\tau} \bphi_{y,h}^n)   - m_{\rho_p}(r_n(\bu_{s}^{\mathrm{P}}), d_{\tau} \bphi_{y,h}^n)  - a_{f}^{\mathrm{P}}( r_n (\by_{s}^{\mathrm{P}})),   d_{\tau} \bphi_{y,h}^n), \\ 
    \mathcal{J}_6 := & - b_{\Gamma}(\bchi_{f,h}^n, \bchi_{r,h}^n, d_{\tau} \bchi_{y,h}^n ; \bphi_{f,h}^n, 0 )  
    - b_{\Gamma}(\bphi_{f,h}^n, \bphi_{r,h}^n, d_{\tau} \bphi_{y,h}^n ; \varsigma \bchi_{f,h}^n, 0) - b_{\Gamma}(\cero, \cero, r_n(\by_{s}^{\mathrm{P}}) ; \bphi_{f,h}^n, 0 ) ,\\ 
    \mathcal{J}_7 := & - b_{\Gamma}(\bchi_{f,h}^n, \bchi_{r,h}^n, d_{\tau} \bchi_{y,h}^n ; \cero, \phi_{fp,h}^n )  
    - b_{\Gamma}(\bphi_{f,h}^n, \bphi_{r,h}^n, d_{\tau} \bphi_{y,h}^n ; \cero, -\chi_{fp,h}^n ) - b_{\Gamma}(\cero, \cero, r_n(\by_{s}^{\mathrm{P}}) ; \cero, \phi_{fp,h}^n  ),  \\ 
    \mathcal{J}_8 := & - c_{\Gamma}(\bchi_{f,h}^n, \bchi_{r,h}^n, d_{\tau} \bchi_{y,h}^n ;  \bphi_{f,h}^n, \bphi_{r,h}^n, d_{\tau} \bphi_{y,h}^n  ) - c_{\Gamma}(\cero, \cero, r_n(\by_{s}^{\mathrm{P}});  \bphi_{f,h}^n, \bphi_{r,h}^n, d_{\tau} \bphi_{y,h}^n  ), \\
  \mathcal{J}_9    := & - ( \bu_f^{\mathrm{S},n} \cdot \nabla  \bu_f^{\mathrm{S},n} -  \bu_{f,h}^{\mathrm{S},n-1} \cdot \nabla  \bu_{f,h}^{\mathrm{S},n},  \bphi_{f,h}^n).
\end{align*}
We employ the inequality~\eqref{ineq:aux} along with the estimate
\begin{equation}\label{ineq:velocity_estimate}
\rho_f\,\phi\,\bigl\lVert \bphi_{r,h}^n + \bphi_{s,h}^n \bigr\rVert_{0,\Omega_{\mathrm{P}}}^2
\geq \rho_f\,\phi\,\left(\tfrac{1}{2}\bigl\lVert \bphi_{r,h}^n \bigr\rVert_{0,\Omega_{\mathrm{P}}}^2 
- \bigl\lVert \bphi_{s,h}^n \bigr\rVert_{0,\Omega_{\mathrm{P}}}^2\right),
\end{equation}
and apply the coercivity properties of the bilinear forms, together with the trace inequality, Hölder’s inequality, and Young’s inequality, to estimate the $\mathrm{LHS}$ of \eqref{ne1}. Furthermore, the terms $\mathcal{J}_1$ through $\mathcal{J}_8$ are estimated analogously, following the approach presented in~\cite{bansal2026lagrange}.

For the nonlinear error term, we have
\begin{align*}
& \bu_f^{\mathrm{S},n} \cdot \bnabla \bu_f^{\mathrm{S},n} - \bu_{f,h}^{\mathrm{S},n-1} \cdot \bnabla \bu_{f,h}^{\mathrm{S},n} 
= \mathcal{S}(\bu_f^{\mathrm{S},n}) \cdot \bnabla \bu_f^{\mathrm{S},n}
+ \bu_f^{\mathrm{S},n-1} \cdot \nabla (\bchi_{f,h}^{n} + \bphi_{f,h}^{n})
+ (\bchi_{f,h}^{n-1} + \bphi_{f,h}^{n-1}) \cdot \bnabla \bu_{f,h}^{\mathrm{S},n},
\end{align*}
where $\mathcal{S}(\bu_f^{\mathrm{S},n}) = \bu_f^{\mathrm{S},n} - \bu_f^{\mathrm{S},n-1}$. Then, using Sobolev and Korn's inequalities, and the assumption
$\| \bu_{f,h}^{\mathrm{S},n} \|_{1,\Omega_{\mathrm{S}}}$ $ 
< \frac{\mu_f}{2 S_f^2 K_f^3}, \quad 1 \leq n \leq N,$
we get
\begin{align*}
 \tau \sum_{n=1}^N \mathcal{J}_9 & : =  -  \tau \sum_{n=1}^N (\bu_f^{\mathrm{S},n} \cdot \bnabla \bu_f^{\mathrm{S},n} - \bu_{f,h}^{\mathrm{S},n-1} \cdot \bnabla \bu_{f,h}^{\mathrm{S},n}, \bphi_{f,h}^{n})_{\Omega_f}\\
&\leq   \tau \sum_{n=1}^N \left( \| \mathcal{S}(\bu_f^{\mathrm{S},n}) \|_{0,4,\Omega_\mathrm{S}} \| \bnabla \bu_f^{\mathrm{S},n} \|_{0,\Omega_\mathrm{S}} \| \bphi_{f,h}^{n} \|_{0,4,\Omega_\mathrm{S}} + \| \bu_f^{\mathrm{S},n-1} \|_{0,4,\Omega_\mathrm{S}} \| \bnabla \bchi_{f,h}^{n} \|_{0,\Omega_\mathrm{S}}   \| \bphi_{f,h}^{n} \|_{0,4,\Omega_\mathrm{S}}   \right. \\& \quad  \left.
+  \| \bu_f^{\mathrm{S},n-1} \|_{0,4,\Omega_\mathrm{S}} \| \bnabla \bphi_{f,h}^{n} \|_{0,\Omega_\mathrm{S}}  \| \bphi_{f,h}^{n} \|_{0,4,\Omega_\mathrm{S}}  +  \| \bchi_f^{n-1} \|_{0,4,\Omega_\mathrm{S}} \| \bnabla \bu_{f,h}^{\mathrm{S},n} \|_{0,\Omega_\mathrm{S}}  \| \bphi_{f,h}^{n} \|_{0,4,\Omega_\mathrm{S}}   
 \right. \\& \left. \quad + \| \bphi_{f,h}^{n-1} \|_{0,4,\Omega_\mathrm{S}} \| \bnabla \bu_{f,h}^{\mathrm{S},n} \|_{0,\Omega_\mathrm{S}}\| \bphi_{f,h}^{n} \|_{0,4,\Omega_\mathrm{S}} \right) \\
&\leq  \tau \sum_{n=1}^N \left( \frac{\mu_f}{2} \| \mathcal{S}( \bnabla \bu_f^{\mathrm{S},n}) \|_{0,\Omega_\mathrm{S}} \| \bnabla \bphi_{f,h}^{n} \|_{0,\Omega_\mathrm{S}}
+ \frac{\mu_f}{2} \| \bnabla \bchi_{f,h}^{n} \|_{0,\Omega_\mathrm{S}} \| \bnabla \bphi_{f,h}^{n} \|_{0,\Omega_\mathrm{S}} \right. \\& \left. + \frac{\mu_f}{2}  \| \bnabla \bphi_{f,h}^{n} \|_{0,\Omega_\mathrm{S}}^2   
+ \frac{\mu_f}{2} \| \bnabla \bchi_f^{n-1} \|_{0,\Omega_\mathrm{S}} \| \bnabla \bphi_{f,h}^{n} \|_{0,\Omega_\mathrm{S}}  + \frac{\mu_f}{2} \| \bnabla \bphi_{f,h}^{n-1} \|_{0,\Omega_\mathrm{S}}   \| \bnabla \bphi_{f,h}^{n} \|_{0,\Omega_\mathrm{S}} \right)   \\
&\leq \mu_f \tau \sum_{n=1}^N \| \bnabla \bphi_{f,h}^{n} \|_{0,\Omega_\mathrm{S}}^2
+  \tau \sum_{n=1}^N \frac{\mu_f}{4} \| \bnabla \bphi_f^{n-1} \|_{0,\Omega_\mathrm{S}}^2
+ \frac{3 \mu_f }{4}  \tau \sum_{n=1}^N   ( \| \bnabla \mathcal{S}(\bu_f^{\mathrm{S},n} )\|_{0,\Omega_\mathrm{S}}^2
 + \| \bnabla \bchi_{f,h}^{n}  \|_{0,\Omega_\mathrm{S}}^2 + \| \bnabla \bchi_{f,h}^{n-1} \|_{0,\Omega_\mathrm{S}}^2 ). 
\end{align*}
By substituting the bounds for $\mathcal{J}_1$--$\mathcal{J}_9$ into~\eqref{ne1} and choosing $\epsilon_1$ sufficiently small, we obtain the desired bound. 
Next, we apply the inf--sup condition~\eqref{inf_sup} with the choice 
$(q_h^{\mathrm{S},n}, q_h^{\mathrm{P},n}) = (\phi_{fp,h}^n,\; \phi_{pp,h}^n),$
and use the error equation obtained by subtracting~\eqref{fully} from~\eqref{n_weak}. 
The analysis follows the approach in~\cite{bansal2026lagrange}, with the additional approximation property
$\tau \sum_{n=1}^N \| \mathcal{S}(\bu_f^{\mathrm{S},n}) \|_{1,\Omega_\mathrm{S}}^2 
\leq C \tau^2 \| \partial_t \bu_f^{\mathrm{S}} \|_{L^2(0, T; H^1(\Omega_{\mathrm{S}}))}^2.$
\cred{This yields 
\begin{align*}
 & \|\bphi_{f,h}^N\|^2_{0,\Omega_{\mathrm{S}}} +  \|\bphi_{r,h}^N\|^2_{0,\Omega_{\mathrm{P}}} + \|\phi_{pp,h}^N\|_{0,\Omega_{\mathrm{P}}}^2  + \sum_{n=1}^N \tau^2 \Big( \|d_{\tau} \bphi_{r,h}^n\|_{0,\Omega_{\mathrm{P}}}^2 + \|d_{\tau} \phi_{pp,h}^n\|_{0,\Omega_{\mathrm{P}}}^2    + \|d_{\tau} \bphi_{y,h}^n\|_{1,\Omega_{\mathrm{P}}}^2   + \|d_{\tau} \bphi_{s,h}^n\|_{0,\Omega_{\mathrm{P}}}^2  \Big) \\
 & \quad + \|\bphi_{y,h}^N\|_{1,\Omega_{\mathrm{P}}}^2 + \|\bphi_{s,h}^N\|_{0,\Omega_{\mathrm{P}}}^2  + \sum_{n=1}^N \tau \biggl( \left| \bphi_{f,h}^n - d_{\tau} \bphi_{y,h}^n \right|_{\mathrm{BJS}}^2 + \left| \bphi_{r,h}^n \right|_{\mathrm{BJS}}^2 + \|\bphi_{r,h}^n\|_{0,\Omega_{\mathrm{P}}}^2  + \frac{( 1 -  (1+ \varsigma)  \widehat{\epsilon_f}^{\prime} C_{\mathrm{tr}} )}{2}   \|\bphi_{f,h}\|_{1,\Omega_{\mathrm{S}}}^2   \\
 &\quad + ( \gamma -  (1+ \varsigma) (\widehat{\epsilon_f}^{\prime} )^{-1} ) \sum_{E \in \mathcal{E}_{\Sigma}} \frac{\mu_f}{h_E} 
\|\bphi_{f,h} \cdot \bn_{\mathrm{S}} + \bphi_{r,h} \cdot \bn_{\mathrm{P}}  + d_{\tau} \bphi_{y,h} \cdot \bn_{\mathrm{S}} \|_{0, E}^2 \biggr)  \\
&\lesssim 
\frac{\tau}{\epsilon_1} \sum_{n=1}^N \bigl( (1 + \varsigma) \|\bchi_{f,h}^n\|_{1,\Omega_{\mathrm{S}}}^2 + \|\bchi_{f,h}^n\|_{0,\Omega_{\mathrm{S}}}^2 + \|\bchi_{r,h}^n\|_{0,\Omega_{\mathrm{P}}}^2 + \|d_{\tau} \bchi_{s,h}^n\|_{0,\Omega_{\mathrm{P}}}^2  + \|d_{\tau} \bchi_{r,h}^n\|_{0,\Omega_{\mathrm{P}}}^2  + \|\bchi_{s,h}^n\|_{0,\Omega_{\mathrm{P}}}^2 + \|d_{\tau} \bchi_{y,h}^n\|_{1,\Omega_{\mathrm{P}}}^2   \\ 
& \quad   + (1 + \varsigma) \|\chi_{f p,h}^n\|_{0,\Omega_{\mathrm{S}}}^2   + h^{-2} \|d_{\tau} \bchi_{y,h}^n\|_{1,\Omega_{\mathrm{P}}}^2   + h^{-2} \|\bchi_{r,h}^n\|_{1,\Omega_{\mathrm{P}}}^2 + \|\bchi_{r,h}^n\|_{1,\Omega_{\mathrm{P}}}^2  + \|r_n(\bu_{s}^{\mathrm{P}})\|_{0,\Omega_{\mathrm{P}}}^2   + \|r_n(\bu_{r}^{\mathrm{P}})\|_{0,\Omega_{\mathrm{P}}}^2    \\ 
& \quad + \|r_n(\bu_{f}^{\mathrm{S}})\|_{0,\Omega_{\mathrm{S}}}^2 
+ \|r_n(\by_{s}^{\mathrm{P}})\|_{1,\Omega_{\mathrm{P}}}^2 + \|r_n(p^{\mathrm{P}})\|_{0,\Omega_{\mathrm{P}}}^2   + h^{-2} \|r_n(\by_{s}^{\mathrm{P}})\|_{1,\Omega_{\mathrm{P}}}^2 \bigr)  + {\epsilon_1}\|\bphi_{y,h}^N\|_{1,\Omega_{\mathrm{P}}}^2    +  {\tau}\sum_{n=1}^{N-1}\|\bphi_{y,h}^n\|_{1,\Omega_{\mathrm{P}}}^2    \\
&\quad  + \epsilon_1 \tau \sum_{n=1}^N   \biggl(  \|\bphi_{f,h}^n\|_{0,\Omega_{\mathrm{S}}}^2   +  \|\bphi_{f,h}^n\|_{1,\Omega_{\mathrm{S}}}^2    + \|\bphi_{fp,h}^n\|_{0,\Omega_{\mathrm{S}}}^2   + \|\bphi_{r,h}^n\|_{0,\Omega_{\mathrm{P}}}^2  + \|\phi_{pp,h}^n\|_{0,\Omega_{\mathrm{P}}}^2  + \left|\bphi_{f,h}^n - d_{\tau} \bphi_{y,h}^n\right|_{\mathrm{BJS}}^2 + \left|\bphi_{r,h}^n \right|_{\mathrm{BJS}}^2  \biggr)   \\ 
&  \quad   + \epsilon_1^{-1} \bigl( \|\bchi_{y,h}^N\|_{1,\Omega_{\mathrm{P}}}^2   +\|\bchi_{r,h}^N\|_{1,\Omega_{\mathrm{P}}}^2 + \|d_{\tau}\bchi_{y,h}^N\|_{1,\Omega_{\mathrm{P}}}^2 + \| \chi_{pp,h}^N \|^2_{0,\Omega_{\mathrm{P}}}  + \|\bchi_{r,h}^N \|_{0,\Omega_{\mathrm{P}}}^2   + \| d_{\tau}\bchi_{y,h}^N \|_{0,\Omega_{\mathrm{P}}}^2   + \| \bchi_{s,h}^N
  \|_{0,\Omega_{\mathrm{P}}}^2   \\ 
& \quad   + \| d_{\tau}\bchi_{r,h}^N \|_{0,\Omega_{\mathrm{P}}}^2  + \| r_N(\bu_r^{\mathrm{P}}) \|_{0,\Omega_{\mathrm{P}}}^2   + \| r_N(\bu_s^{\mathrm{P}}) \|_{0,\Omega_{\mathrm{P}}}^2 
  + \| r_N(\by_s^{\mathrm{P}}) \|_{1,\Omega_{\mathrm{P}}}^2 \bigr)  + C\bigl( \|\bchi_{y,h}^0\|_{1,\Omega_{\mathrm{P}}}^2 
   +\|\bchi_{r,h}^0\|_{1,\Omega_{\mathrm{P}}}^2    +\|d_{\tau}\bchi_{y,h}^0\|_{1,\Omega_{\mathrm{P}}}^2  \\ 
  & \quad      + \| \chi_{pp,h}^0 \|^2_{0,\Omega_{\mathrm{P}}} + \|\bchi_{r,h}^0 \|_{0,\Omega_{\mathrm{P}}}^2    + \| d_{\tau}\bchi_{y,h}^0 \|_{0,\Omega_{\mathrm{P}}}^2 + \| \bchi_{s,h}^0\|_{0,\Omega_{\mathrm{P}}}^2 + \| d_{\tau}\bchi_{r,h}^0 \|_{0,\Omega_{\mathrm{P}}}^2   + \| r_0(\bu_r^{\mathrm{P}}) \|_{0,\Omega_{\mathrm{P}}}^2   + \| r_0(\bu_s^{\mathrm{P}}) \|_{0,\Omega_{\mathrm{P}}}^2 + \| r_0(\by_s^{\mathrm{P}}) \|_{1,\Omega_{\mathrm{P}}}^2 \bigr)  \\ 
   & \quad   + \tau \sum_{n=1}^{N-1}  \bigl(\|d_{\tau}\bchi_{y,h}^n\|_{1,\Omega_{\mathrm{P}}}^2  +  \|d_{\tau}\bchi_{r,h}^n\|_{1,\Omega_{\mathrm{P}}}^2   + \|d_{\tau} d_{\tau}\bchi_{y,h}^n\|_{1,\Omega_{\mathrm{P}}}^2  + \| d_{\tau} \chi_{pp,h}^n \|^2_{0,\Omega_{\mathrm{P}}}  + \|d_{\tau} \bchi_{r,h}^n \|_{0,\Omega_{\mathrm{P}}}^2  + \| d_{\tau} d_{\tau}\bchi_{y,h}^n \|_{0,\Omega_{\mathrm{P}}}^2     \\
   & \quad  + \| d_{\tau}  \bchi_{s,h}^n
  \|_{0,\Omega_{\mathrm{P}}}^2 + \| d_{\tau}  d_{\tau}\bchi_{r,h}^n \|_{0,\Omega_{\mathrm{P}}}^2  + \| d_{\tau}  r_n(\bu_r^{\mathrm{P}}) \|_{0,\Omega_{\mathrm{P}}}^2  +  \| d_{\tau}  r_n(\bu_s^{\mathrm{P}}) \|_{0,\Omega_{\mathrm{P}}}^2  + \| d_{\tau} r_n(\by_s^{\mathrm{P}}) \|_{1,\Omega_{\mathrm{P}}}^2    \bigr)
   \\&  \quad  
+ \epsilon_1 \tau \sum_{n=1}^N \sum_{E \in \mathcal{E}_{\Sigma}} \frac{1}{h_E} 
  \|\bphi_{f,h}^n \cdot \bn_{\mathrm{S}}  + \bphi_{r,h}^n \cdot \bn_{\mathrm{P}} + d_{\tau} \bphi_{y,h}^n \cdot \bn_{\mathrm{P}} \|_{0,E}^2 
  \\ &   \quad  + \epsilon_1^{-1} \tau \sum_{n=1}^N  \sum_{E \in \mathcal{E}_{\Sigma}} \frac{1}{h_E} \bigl( \|\bchi_{f,h}^n \cdot \bn_{\mathrm{S}} + \bchi_{r,h}^n \cdot \bn_{\mathrm{P}} 
+ d_{\tau} \bchi_{y,h}^n \cdot \bn_{\mathrm{P}} \|_{0,E}^2  + \|r_n(\by_{s}^{\mathrm{P}}) \cdot \bn_{\mathrm{P}} \|_{0,E}^2  \bigr)  +  \tau \sum_{n=1}^N \| \mathcal{S}(\bu_f^{\mathrm{S},n} )\|_{1,\Omega_\mathrm{S}}^2. 
\end{align*} 
Finally,} combining all the estimates with the discrete Gronwall inequality~\cite{MR3904522}, the triangle inequality, and the approximation properties in~\eqref{ps}--\eqref{vh}, \eqref{sp3}, and~\eqref{mfe3}--\eqref{sz1} yields the assertion of the theorem.
\end{proof}

\end{document}